\documentclass[12pt]{amsart}

\usepackage{amsmath,amsthm,amssymb,enumitem,verbatim,stmaryrd,xcolor,microtype,graphicx,aliascnt,needspace}
\usepackage[T1]{fontenc}
\usepackage[utf8]{inputenc}
\usepackage[english]{babel}\def\hyph{-\penalty0\hskip0pt\relax} 
\usepackage[top=3.5cm,bottom=3.55cm,left=3.5cm,right=3.5cm]{geometry}
\usepackage[bookmarksdepth=2,linktoc=page,colorlinks,linkcolor={red!70!black},citecolor={red!80!black},urlcolor={blue!80!black}]{hyperref}
\usepackage{tikz}\usetikzlibrary{matrix,arrows,decorations.markings,shapes}
\usepackage{tikz-cd}
\pgfrealjobname{lifttheorems}
\usepackage[mathcal]{euscript}                               

\usepackage{mathptmx}                                        
\usepackage[colorinlistoftodos,bordercolor=orange,backgroundcolor=orange!20,linecolor=orange,textsize=footnotesize]{todonotes} \makeatletter \providecommand \@dotsep{5} \def\listtodoname{List of Todos} \def\listoftodos{\@starttoc{tdo}\listtodoname} \makeatother 

\allowdisplaybreaks 

\usepackage{etoolbox}
\makeatletter
\patchcmd{\@startsection}{\@afterindenttrue}{\@afterindentfalse}{}{}             
\patchcmd{\part}{\bfseries}{\bfseries\LARGE}{}{}
\patchcmd{\section}{\scshape}{\bfseries}{}{}\renewcommand{\@secnumfont}{\bfseries} 
\patchcmd{\@settitle}{\uppercasenonmath\@title}{\large}{}{}
\patchcmd{\@setauthors}{\MakeUppercase}{}{}{}
\addto{\captionsenglish}{} 
\addto{\captionsenglish}{} 
\addto{\captionsenglish}{} 
\makeatother

\usepackage{fancyhdr}

\addto\extrasenglish{}  
\theoremstyle{plain}
\newtheorem{thm}{Theorem}[section] 
\newaliascnt{lemma}{thm}\newtheorem{lemma}[lemma]{Lemma}\aliascntresetthe{lemma}
\newaliascnt{cor}{thm}\newtheorem{cor}[cor]{Corollary}\aliascntresetthe{cor}
\newaliascnt{prop}{thm}\newtheorem{prop}[prop]{Proposition}\aliascntresetthe{prop}
\newtheorem{thmA}{Theorem} 

\newtheorem*{thm*}{Theorem}
\newtheorem*{lem*}{Lemma}
\newtheorem*{cor*}{Corollary}

\theoremstyle{definition}
\newaliascnt{df}{thm}\newtheorem{df}[df]{Definition}\aliascntresetthe{df}
\newaliascnt{rem}{thm}\newtheorem{rem}[rem]{Remark}\aliascntresetthe{rem}
\newaliascnt{ex}{thm}\newtheorem{ex}[ex]{Example}\aliascntresetthe{ex}

\newtheorem*{df*}{Definition}
\newtheorem*{ex*}{Example}
\newtheorem*{rem*}{Remark}

\theoremstyle{remark}

\DeclareRobustCommand{\gobblefour}[5]{}    

\DeclareFontFamily{OT1}{pzc}{}                                
\DeclareFontShape{OT1}{pzc}{m}{it}{<-> s * [1.10] pzcmi7t}{}
\DeclareMathAlphabet{\mathpzc}{OT1}{pzc}{m}{it}
\DeclareSymbolFont{sfoperators}{OT1}{bch}{m}{n} \DeclareSymbolFontAlphabet{\mathsf}{sfoperators} \makeatletter\def\operator@font{\mathgroup\symsfoperators}\makeatother 
\DeclareSymbolFont{cmletters}{OML}{cmm}{m}{it}              
\DeclareSymbolFont{cmsymbols}{OMS}{cmsy}{m}{n}
\DeclareSymbolFont{cmlargesymbols}{OMX}{cmex}{m}{n}
\DeclareMathSymbol{\myjmath}{\mathord}{cmletters}{"7C}     \let\jmath\myjmath 
\DeclareMathSymbol{\myamalg}{\mathbin}{cmsymbols}{"71}     \let\amalg\myamalg
\DeclareMathSymbol{\mycoprod}{\mathop}{cmlargesymbols}{"60}\let\coprod\mycoprod
\DeclareMathSymbol{\myalpha}{\mathord}{cmletters}{"0B}     \let\alpha\myalpha 
\DeclareMathSymbol{\mybeta}{\mathord}{cmletters}{"0C}      \let\beta\mybeta
\DeclareMathSymbol{\mygamma}{\mathord}{cmletters}{"0D}     \let\gamma\mygamma
\DeclareMathSymbol{\mydelta}{\mathord}{cmletters}{"0E}     \let\delta\mydelta
\DeclareMathSymbol{\myepsilon}{\mathord}{cmletters}{"0F}   \let\epsilon\myepsilon
\DeclareMathSymbol{\myzeta}{\mathord}{cmletters}{"10}      \let\zeta\myzeta
\DeclareMathSymbol{\myeta}{\mathord}{cmletters}{"11}       \let\eta\myeta
\DeclareMathSymbol{\mytheta}{\mathord}{cmletters}{"12}     \let\theta\mytheta
\DeclareMathSymbol{\myiota}{\mathord}{cmletters}{"13}      \let\iota\myiota
\DeclareMathSymbol{\mykappa}{\mathord}{cmletters}{"14}     \let\kappa\mykappa
\DeclareMathSymbol{\mylambda}{\mathord}{cmletters}{"15}    \let\lambda\mylambda
\DeclareMathSymbol{\mymu}{\mathord}{cmletters}{"16}        \let\mu\mymu
\DeclareMathSymbol{\mynu}{\mathord}{cmletters}{"17}        \let\nu\mynu
\DeclareMathSymbol{\myxi}{\mathord}{cmletters}{"18}        \let\xi\myxi
\DeclareMathSymbol{\mypi}{\mathord}{cmletters}{"19}        \let\pi\mypi
\DeclareMathSymbol{\myrho}{\mathord}{cmletters}{"1A}       \let\rho\myrho
\DeclareMathSymbol{\mysigma}{\mathord}{cmletters}{"1B}     \let\sigma\mysigma
\DeclareMathSymbol{\mytau}{\mathord}{cmletters}{"1C}       \let\tau\mytau
\DeclareMathSymbol{\myupsilon}{\mathord}{cmletters}{"1D}   \let\upsilon\myupsilon
\DeclareMathSymbol{\myphi}{\mathord}{cmletters}{"1E}       \let\phi\myphi
\DeclareMathSymbol{\mychi}{\mathord}{cmletters}{"1F}       \let\chi\mychi
\DeclareMathSymbol{\mypsi}{\mathord}{cmletters}{"20}       \let\psi\mypsi
\DeclareMathSymbol{\myomega}{\mathord}{cmletters}{"21}     \let\omega\myomega
\DeclareMathSymbol{\myvarepsilon}{\mathord}{cmletters}{"22}\let\varepsilon\myvarepsilon
\DeclareMathSymbol{\myvartheta}{\mathord}{cmletters}{"23}  \let\vartheta\myvartheta
\DeclareMathSymbol{\myvarpi}{\mathord}{cmletters}{"24}     \let\varpi\myvarpi
\DeclareMathSymbol{\myvarrho}{\mathord}{cmletters}{"25}    \let\varrho\myvarrho
\DeclareMathSymbol{\myvarsigma}{\mathord}{cmletters}{"26}  \let\varsigma\myvarsigma
\DeclareMathSymbol{\myvarphi}{\mathord}{cmletters}{"27}    \let\varphi\myvarphi

\DeclareMathOperator{\Hom}{Hom}

\DeclareMathOperator{\Sym}{Sym}
\DeclareMathOperator{\Hex}{Hex}
\DeclareMathOperator{\Stab}{Stab}

\DeclareMathOperator{\sign}{sign}

\DeclareMathOperator{\Ob}{Ob}

\DeclareMathOperator{\PartFields}{{PartFields}}
\DeclareMathOperator{\MPF}{{MockPartFields}}

\DeclareMathOperator{\Pastures}{{Pastures}}
\DeclareMathOperator{\FPastures}{{FinPastures}}
\DeclareMathOperator{\Foundations}{{Foundations}}
\DeclareMathOperator{\Lifts}{{Lifts}}
\DeclareMathOperator{\GRS}{{GRS}}
\DeclareMathOperator{\PvZ}{{PvZ}}
\DeclareMathOperator{\WLUM}{{WLUM}}

\newcommand{\wlum}{{\mathcal{W}}}
\newcommand{\binary}{{\mathcal{B}}}
\newcommand{\ternary}{{\mathcal{T}}}
\newcommand{\grs}{{\mathcal{G}}}
\newcommand{\pvz}{{\mathcal{P}}}

\newcommand\D{{\mathbb D}}

\newcommand\F{{\mathbb F}}
\newcommand\G{{\mathbb G}}
\renewcommand\H{{\mathbb H}}

\newcommand\K{{\mathbb K}}

\newcommand\R{{\mathbb R}}
\renewcommand\S{{\mathbb S}}

\newcommand\U{{\mathbb U}}

\newcommand\W{{\mathbb W}}

\newcommand\Z{{\mathbb Z}}
\newcommand\bI{{\mathbf I}}
\newcommand\bJ{{\mathbf J}}

\newcommand\cB{{\mathcal B}}
\newcommand\cC{{\mathcal C}}
\newcommand\cD{{\mathcal D}}

\newcommand\cG{{\mathcal G}}

\newcommand\cL{{\mathcal L}}

\newcommand\cR{{\mathcal R}}

\newcommand\cT{{\mathcal T}}

\newcommand\cW{{\mathcal W}}
\newcommand\cX{{\mathcal X}}

\newcommand\Lift{{\cL}}

\newcommand\Funpm{{\F_1^\pm}}
\newcommand\id{\textup{id}}
\newcommand\im{\textup{im}}

\renewcommand\geq{\geqslant}
\renewcommand\leq{\leqslant}
\renewcommand\check{\checkmark}
\newcommand{\gen}[1]{\langle #1 \rangle}
\newcommand{\norm}[1]{|#1|}
\newcommand{\past}[2]{#1\!\sslash\!#2}
\newcommand{\pastgen}[3]{#1\langle #2 \rangle \!\sslash\!\{ #3 \}}
\newcommand{\fund}{{\,\beginpgfgraphicnamed{tikz/fund}\begin{tikzpicture}[x=1pt,y=1pt]\draw[line width=0.5] (0,0) circle (2);\end{tikzpicture}\endpgfgraphicnamed}}
\newcommand{\pair}{{\,\beginpgfgraphicnamed{tikz/pair}\begin{tikzpicture}[x=1pt,y=1pt]\draw[line width=0.5] (0,0) circle (2); \draw[line width=0.5] (2.5,0) circle (2); \draw[very thick, white] (2,0) arc (0:70:2); \draw[very thick, white] (0.5,0) arc (180:250:2); \draw[line width=0.5] (2,0) arc (0:70:2); \draw[line width=0.5] (0.5,0) arc (180:250:2); \end{tikzpicture}\endpgfgraphicnamed}}
\newcommand{\hexagon}{{\beginpgfgraphicnamed{tikz/hexagon}\begin{tikzpicture}\node[draw,scale=0.45,regular polygon, regular polygon sides=6](){};  \end{tikzpicture}\endpgfgraphicnamed}}
\newcommand{\hex}[6]{\mathchoice{\left\langle\begin{matrix}&{#1}&{#2}\\[-6.0pt]{#3}&&&{#4}\\[-6.0pt]&{#5}&{#6}\end{matrix}\right\rangle}{\left\langle\begin{smallmatrix}&{#1}&{#2}\\[-4.0pt]{#3}&&&{#4}\\[-4.0pt]&{#5}&{#6}\end{smallmatrix}\right\rangle}{}{}}

\newcommand{\hyperplus}{\mathrel{\,\raisebox{-1.1pt}{\larger[-0]{$\boxplus$}}\,}}

\renewcommand{\implies}{\ensuremath{\Rightarrow}}

\renewcommand\emptyset\varnothing

\title{Lift theorems for representations of matroids over pastures}

\author{Matthew Baker}
\address{\rm Matthew Baker, School of Mathematics, Georgia Institute of Technology, Atlanta, USA}
\email{mbaker@math.gatech.edu}

\author{Oliver Lorscheid}
\address{\rm Oliver Lorscheid, University of Groningen, the Netherlands, and IMPA, Rio de Janeiro, Brazil}
\email{oliver@impa.br}

\begin{document}

\begin{abstract}
Pastures are a class of field-like algebraic objects which include both partial fields hyperfields and have nice categorical properties. 
We prove several lift theorems for representations of matroids over pastures, including a generalization of Pendavingh and van Zwam's Lift Theorem for partial fields. By embedding the earlier theory into a more general framework, we are able to establish new results even in the case of lifts of partial fields, for example the conjecture of Pendavingh--van Zwam that their lift construction is idempotent. We give numerous applications to matroid representations, e.g.
we show that, up to projective equivalence, every pair consisting of a hexagonal representation and an orientation lifts uniquely to a near-regular representation.
The proofs are different from the arguments used by Pendavingh and van Zwam, relying instead on a result of Gelfand--Rybnikov--Stone inspired by Tutte's homotopy theorem.
\end{abstract}

\maketitle

\begin{small} \tableofcontents \end{small}


\section*{Introduction}
\label{introduction}

\subsection*{Overview}
Our goal in this paper is to introduce a new lifting technique into matroid representation theory, and to explore some of its combinatorial implications. Although the technique applies to much more general algebraic structures (which we call {\em pastures}) than partial fields, in this introductory subsection we will stick to the more ``classical'' setting of partial fields, since even in that case some of our results seem to be new.

To fix some notation and terminology, given a matroid $M$ and a partial field $P$, we let $\cX_M(P)$ denote the corresponding {\em rescaling class space}, which is the set of projective equivalence classes\footnote{In the more general context of matroids over pastures, we refer to {\em rescaling equivalence classes} rather than projective equivalence classes; see \autoref{sec:rescaling} for details.} of representations of $M$ over $P$. We use the following notation for some familiar partial fields in matroid theory: 
\begin{itemize}
\item[] $\F_q$: the finite field of order $q$
\item[] $\F_1^\pm$: the regular partial field
\item[] $\D$: the dyadic partial field
\item[] $\H$: the hexagonal (or sixth-root-of-unity) partial field
\item[] $\U$: the near-regular partial field
\item[] $\G$: the golden ratio partial field
\end{itemize}
We also denote by $\K$ the Krasner hyperfield and by $\S$ the sign hyperfield.

Here is a sampling of some concrete results about matroid representations that can be obtained with our new method:

\begin{thmA}\label{thmA}
 Let $M$ be a matroid.
 \begin{enumerate}
 \item There is a canonical bijection between $ \cX_M(\G) $ and $ \cX_M(\F_{4}) \times \cX_M(\F_{5})$. In other words, up to projective equivalence, every pair consisting of quaternary representation and a quinternary representation lifts uniquely to a golden ratio representation.
 \item If $M$ is ternary, then $\cX_M(\F_4) = \cX_M(\H)$, $\cX_M(\F_5) = \cX_M(\D)$, and $\cX_M(\F_8) = \cX_M(\U)$.
 In other words, up to projective equivalence, every quarternary representation of $M$ lifts uniquely to a hexagonal representation, every quinternary representation of $M$ lifts uniquely to a dyadic representation, and
 every octernary representation of $M$ lifts uniquely to a near-regular representation.
 \item If $q,p_1,p_2$ are prime powers with $3 \nmid q$ and $q-2 = (p_1 - 2)(p_2 - 2)$, then for every ternary matroid $M$ there is a canonical bijection between $ \cX_M(\F_q)$ and $\cX_M(\F_{p_1}) \times \cX_M(\F_{p_2})$.
  Such identifications occur, for example, for $(q,p_1,p_2) \in \{  (  8, 4, 5),\, ( 29, 5,11),\, ( 32, 4,17),\, ( 53, 5,19) \}$.
  \end{enumerate}
\end{thmA}

For (1), Vertigan proved (cf. \cite[Thm.\ 4.9]{Pendavingh-vanZwam10b}) that a matroid is golden ratio if and only if it is both quaternary and quinternary; we have not seen the uniqueness assertion stated in the literature but it can be deduced from the techniques of \cite{Pendavingh-vanZwam10a,Pendavingh-vanZwam10b}.
For (2), it was previously known that such a lift {\em exists} in each case.\footnote{See \cite[Thm.\ 1.1]{Whittle97} for ternary plus quinternary implies dyadic, \cite[Thm.\ 1.2]{Whittle97} for ternary plus quaternary implies hexagonal, and \cite[Thm.\ 3.2]{Whittle05} for ternary plus octernary implies near-regular.} 
So the main novelty in this case is that we're able to establish {\em uniqueness} in addition to existence. To the best of our knowledge, both the existence and uniqueness assertions implicit in (3) are new.

Our method of proof for existence is substantially different from the previous work in the subject, in that we make systematic use of Tutte's homotopy theory along with `abstract nonsense' about the category of pastures. 

\medskip

Part (1) of \autoref{thmA} will be proved in \autoref{thm:Vertigan}, part (2) in \autoref{thm: unique lifts of representations of ternary matroids}, and part (3) in \autoref{thm: identifications of products of rescaling class spaces}.

\medskip

Our approach to proving such results is based on embedding partial fields into the larger category of {\em pastures}, which contain hyperfields as well as partial fields and admit both products and tensor products. In addition to providing a more structured framework for thinking about such results, and thereby allowing us to prove uniqueness as well as existence assertions, our approach allows us to treat oriented matroids (for example) in the same way one would treat matroids over a partial field. Indeed, oriented matroids are just matroids over the sign hyperfield $\S$, and the rescaling class space $\cX_M(\S)$ is the set of {\em reorientation classes} of $M$. The generalized setting of pastures allows us to obtain results such as the following:

\begin{thmA}\label{thmB}
Let $M$ be a matroid. 
\begin{enumerate}
\item If $M$ is ternary, then $\cX_M(\F_7) = \cX_M(\D\otimes\H)$. In other words,  up to projective equivalence, every septernary representation of $M$ lifts uniquely to a $\D\otimes\H$-representation.
\item If $M$ has no minor isomorphic to $U_{2,5}$ or $U_{3,5}$, then every reorientation class lifts uniquely to a projective equivalence class of $\D$-representations.
\item There is a natural bijection between $\cX_M(\U)$ and $\cX_M(\H) \times \cX_M(\S)$. In other words, up to projective equivalence, every pair consisting of a hexagonal representation and a reorientation class lifts uniquely to a near-regular representation.
\end{enumerate}
\end{thmA}

Once again, for (1) and (2) existence of lifts was previously known (they follow from \cite[Thm.\ 1.3]{Whittle97} and the Lee--Scobee theorem \cite{Lee-Scobee99}, respectively), 
so the novelty here is primarily in the uniqueness assertions and the method of proof. As far as we know, both the existence and uniqueness assertions in (3) are new.

\medskip

Part (1) of \autoref{thmB} will be proved in \autoref{thm: unique lifts of representations of ternary matroids}, part (2) in \autoref{thm: unique lifts of representations of matroids wlum}, and part (3) in \autoref{cor: rescaling class spaces of U, H, and S}.

\subsection*{The main new technique}
Our starting point for the proof of the lifting results described in the previous section is a generalization of the Lift Theorem of Pendavingh and van Zwam from partial fields to pastures. The Lift Theorem associates to each partial field $P$ a partial field $\Lift_{\pvz}(P)$ and a homomorphism $\Lift_{\pvz}(P) \to P$ with the property that every representation of a matroid $M$ over $P$ lifts to a representation of $M$ over $\Lift_{\pvz}(P)$. By generalizing the Lift Theorem to pastures, we not only widen the scope of the result, we also obtain a more precise version which allows us to prove the idempotence of $\Lift_{\pvz}(P)$ conjectured in \cite[Conj.\ 6.7]{Pendavingh-vanZwam10b}.
We denote our generalized lift of a pasture $P$ by $\Lift_{\grs}(P)$, since our proof that every matroid representation lifts (uniquely) from $P$ to $\Lift_{\grs}(P)$ relies heavily on the results of Gelfand--Rybnikov--Stone (\cite{Gelfand-Rybnikov-Stone95}), as amplified and reinterpreted in \cite{Baker-Lorscheid20}. The work of Gelfand--Rybnikov--Stone is itself based on earlier work of Tutte (\cite{Tutte58a}) and Wenzel (\cite{Wenzel89}, \cite{Wenzel91}).

\medskip

Unfortunately, the general nature of our souped-up Lift Theorem -- which applies to all pastures and all matroids -- means that in certain concrete situations of interest it fails to give sharp results. 
For this reason, we define various other lifts which only provide information about a restricted class of matroids, but which give optimal results when they apply.

\medskip

For example, the $\GRS$-lift $\Lift_{\grs}(\S)$ of the sign hyperfield $\S$ is equal to $\S$ itself, which furnishes no information. However, for each pasture $P$ we also define a {\em $\WLUM$-lift} $\Lift_{\cW}(P)$, which has the property that for each matroid $M$ without large uniform minors (i.e., with no minor isomorphic to $U_{2,5}$ or $U_{3,5}$), 
every rescaling equivalence class of $P$-representations lifts uniquely to $\Lift_{\cW}(P)$.
The $\WLUM$-lift of $\S$ is equal to $\D$, and the generalized Lee--Scobee theorem follows.

\subsection*{A crash course on pastures}

In order to state our lifting results more precisely, we first recall some basic facts about pastures.
We give just a brief sketch here; see \autoref{section: Background} below for more details.

\medskip

A \emph{pointed monoid} is a (multiplicatively written) commutative monoid $P$ with identity element $1$ and an \emph{absorbing element} $0$ that satisfies $0\cdot a=0$ for all $a\in P$. We write $P^\times$ for the group of invertible elements in $P$. We denote by $\Sym_3(P)$ the quotient of $P^3$ by the $S_3$-action that permutes coefficients, and we write $a+b+c$ for the class of $(a,b,c)$ in $\Sym_3(P)$. 

\medskip

A \emph{pasture} is a pointed monoid $P$ such that every nonzero element is invertible (i.e., $P^\times=P-\{0\}$), together with a subset $N_P$ of $\Sym_3(P)$ (called the \emph{nullset of $P$}) such that:
 \begin{enumerate}[label={(P\arabic*)}]
  \item\label{P1} $a+0+0\in N_P$ if and only if $a=0$.
  \item\label{P2} If $a+b+c\in N_P$ and $d \in P^\times$ then $ad+bd+cd \in N_P$.
  \item\label{P3} There is a unique element $-1\in P^\times$ such that $1+(-1)+0\in N_P$.
 \end{enumerate}

\medskip

We call $N_P$ the \emph{nullset of $P$}, and say that \emph{$a+b+c$ is null}, and write symbolically $a+b+c=0$, if $a+b+c\in N_P$. We write $-a$ for $(-1)\cdot a$ and $a+b-c=0$ or $a+b=c$ for $a+b+(-c)=0$. 
We often write $a+b \in N_P$ instead of $a+b+0 \in N_P$.

\medskip

A \emph{morphism of pastures} is a multiplicative map $f:P_1\to P_2$ with $f(0)=0$ and $f(1)=1$ such that $f(a)+f(b)+f(c)\in N_{P_2}$ whenever $a+b+c\in N_{P_1}$. 
This defines the category $\Pastures$ of pastures.
 
\begin{ex*}
We can associate with a field $K$ the following pasture $P$: as multiplicative monoids, we define $P=K$; the nullset of $P$ consists of all $a+b+c\in\Sym_3(P)$ such that $a+b+c=0$ in $K$.
\end{ex*}

\begin{ex*}
A {\em partial field} is given by a pair $(G,R)$ of a ring $R$ together with a subgroup $G$ of the unit group $R^\times$ that contains $-1$. The associated pasture is $P=G\cup\{0\}$, as a monoid with zero, together with the nullset $N_P$ consisting of all elements $a+b+c\in\Sym_3(P)$ such that $a+b+c=0$ in $R$. 

The \emph{regular partial field} corresponds to the pair $(\{ \pm 1 \},\Z)$. As a pasture, the underlying monoid of $\Funpm$ is $\{ 0, 1, -1 \}$ with the usual multiplication, and the nullset is $N_{\Funpm} \ = \ \{1 + (-1)\}$. The regular partial field is an initial object of $\Pastures$, i.e., there is a unique morphism from $\Funpm$ to $P$ for every pasture $P$.
\end{ex*}

\begin{ex*}
 The \emph{Krasner hyperfield} is the pasture $\K$ whose underlying monoid is $\{0,1\}$ with the usual multiplication, and whose nullset is $N_{\K}=\{1+1,1+1+1\}$. It is a terminal object of $\Pastures$, i.e., there is a unique morphism from $P$ to $\K$ for every pasture $P$.
 
 The \emph{sign hyperfield} is the pasture $\S$ whose underlying monoid is $\{0,1, -1 \}$ with the usual multiplication, and whose nullset is $N_{\S}=\{1+ (-1), 1+1+(-1), 1+(-1)+(-1) \}$. 
\end{ex*}

\subsubsection*{Fundamental pairs, fundamental elements, and hexagons}

\medskip

A \emph{fundamental pair} in a pasture $P$ is a pair $(a,b)\in (P^\times)^2$ such that $a+b-1 \in N_P$. We denote the set of fundamental pairs in $P$ by $P^\pair$.

\medskip

A \emph{fundamental element} of $P$ is an element $a \in P^\times$ belonging to some fundamental pair. We denote the set of fundamental elements of $P$ by $P^\fund$.

\medskip

There is an action of the dihedral group $D_3 =\gen{\rho,\sigma\mid \rho^3=\sigma^2=(\sigma\rho)^2=e}$ of order 6 on the set of fundamental pairs defined by $\sigma (a,b) = (b,a)$ and
$\rho (a,b) = (-a^{-1}b,a^{-1})$.
A \emph{hexagon} of $P$ is an orbit of this action.
 
\subsubsection*{Generators and relations}

One can define pastures as algebras over $\Funpm$ given by certain generators and relations. 

\medskip

If $P$ is a pasture and $\{t_i\}_{i\in I}$ a set of indeterminates, there is a \emph{free $P$-algebra} on $\{t_i\}$, denoted $P\gen{t_i\mid i\in I}$, which satisfies a variant of the universal property for free algebras (more precisely, the functor $I\mapsto P\gen{t_i\mid i\in I}$ is left adjoint to the functor $Q\mapsto Q^\times$ from $P$-pastures to sets).

\medskip

If $S\subset \Sym_3(P)$ is a set of elements of the form $a+b+c$ with $ab\neq 0$, one can define the \emph{quotient $\past PS$ of $P$ by $S$}, which satisfies the expected universal property for quotients.

\medskip

Combining these operations, one can present every pasture by generators and relations as $\past{\F_1^\pm \gen{t_i \; | \; i \in I }}S$ for suitable generators $\{t_i\}$ and relations $S\subset \Sym_3(P)$.

\begin{ex*}
We have the following presentations for various partial fields (identified with the corresponding pastures) that will be important in the sequel:
\begin{align*}
 & \text{the \emph{dyadic partial field}}       & \D \ &= \ \pastgen\Funpm{z}{z+z-1}; \\
 & \text{the \emph{hexagonal partial field}}    & \H \ &= \ \pastgen\Funpm{z}{z^3 + 1, \; z+z^{-1}-1}; \\
 & \text{the \emph{near-regular partial field}} & \U \ &= \ \pastgen\Funpm{x,y}{x+y-1}; \\
 & \text{the \emph{golden ratio partial field}} & \G \ &= \ \pastgen\Funpm{z}{z^2+z-1}.
\end{align*}
\end{ex*}

\subsection*{Matroids over pastures}

We recall the following facts from \cite{Baker-Lorscheid20} (see also \autoref{section: Background} below):

\begin{enumerate}
\item Given a matroid $M$ and a pasture $P$, one can define the notion of a {\em $P$-representation of $M$} generalizing the usual notion of matroid representability over partial fields. 
\item One can define an equivalence relation called {\em rescaling equivalence} which generalizes the usual notions of projective equivalence over partial fields and reorientation equivalence for oriented matroids. The set of rescaling equivalence classes of representations of $M$ over $P$ is denoted by $\cX_M(P)$, which extends our previous notation for partial fields.
\item The functor from pastures to sets taking a pasture $P$ to the set $\cX_M(P)$ is representable by a pasture $F_M$ called the {\em foundation} of $M$. In other words, there is a natural bijection $\cX_M(P) \cong \Hom(F_M,P)$ which is functorial in $P$. 
\end{enumerate}

\subsection*{Reflections and coreflections}

A full subcategory $\cD$ of a category $\cC$ is called {\em coreflective} if the inclusion functor from $\cD$ to $\cC$ has a right adjoint.
Concretely, what this means is that every object $X \in \Ob(\cC)$ admits a functorial ``lift'' $\cL_{\cD} X$, together with a morphism $\lambda_X : \cL_{\cD} X \to X$, satisfying:

\medskip
\noindent\textit{(Universal Property of Coreflections).}
For every morphism $\varphi : Y \to X$ with $Y \in \Ob(\cD)$, there is a unique morphism $\hat\varphi: Y\to\cL_{\cD} X$ such that $\varphi=\lambda_X\circ\hat\varphi$, i.e.\ the diagram
\[
 \begin{tikzcd}[column sep=80pt, row sep=20pt]
                                           & \Lift_\cD X \ar[d,"\lambda_X"] \\
  Y \ar[ur,"\hat\varphi"] \ar[r,"\varphi"'] & X
 \end{tikzcd}
\]
commutes.

\medskip

For example, the universal cover $\hat{X}$ of a semilocally simply connected topological space $X$ provides a coreflection onto the subcategory of simply connected spaces, with $\lambda_X : \hat{X} \to X$ the universal covering map.

\medskip

Given an inclusion of $\cD$ as a full subcategory of $\cC$, any two coreflections from $\cC$ onto $\cD$ are naturally isomorphic (this is a well-known general property of adjoint functors). Moreover, it follows from the universal property of coreflections that $\cL_{\cD}(Y)=Y$ for every $Y \in \Ob(\cD)$, and in particular that the lift construction is \emph{idempotent}, i.e., 
$\cL_{\cD} (\cL_{\cD} X) \simeq \cL_{\cD} X$ for every $X \in \Ob(\cC)$.

\medskip

Although less central to the paper, we will also make use of {\em reflective} subcategories. 
A full subcategory $\cD$ of a category $\cC$ is called {\em reflective} if the inclusion functor from $\cD$ to $\cC$ has a left adjoint.
Concretely, what this means is that every object $X \in \Ob(\cC)$ admits a functorial ``reflection'' $\cR_{\cD} X$, together with a morphism $\rho_X : X \to \cR_{\cD} X$, satisfying:

\medskip
\noindent\textit{(Universal Property of Reflections).}
For every morphism $\varphi : X \to Y$ with $Y \in \Ob(\cD)$, there is a unique morphism $\hat\varphi: \cR_{\cD} X \to Y$ such that $\varphi=\hat\varphi\circ\rho_X$.

\medskip

For example, the category of abelian groups is a reflective subcategory of the category of groups, with the reflection given by the canonical abelianization map $G \to G^{\rm ab}$.

\subsection*{The \texorpdfstring{$\GRS$}{GRS} coreflection}

We note that since the foundation $F_M$ of $M$ represents the functor $\cX_M(\cdot)$ from pastures to sets, it follows formally from `abstract nonsense' that:

\medskip
\begin{quote}
If $\cD$ is a coreflective subcategory of the category $\Pastures$ of pastures, then for every matroid $M$ with $F_M \in \cD$ and every pasture $P$, every rescaling class of $P$-representations of $M$ lifts uniquely to $\cL_{\cD} P$.
\end{quote}
\medskip

Our first main result about coreflective subcategories of $\Pastures$ is the following:

\par\needspace{2\baselineskip}  
\begin{thmA}\label{thmC} \ 
\begin{enumerate}
\item There is a canonical coreflection $\cL_{\grs} : \Pastures \to \cG$ onto a certain full subcategory $\cG$ of $\Pastures$, containing all foundations of matroids, taking a pasture $P$ to its {\em $\GRS$-lift}.
\item There is a full subcategory $\MPF$ of $\Pastures$, properly containing the category $\PartFields$ of partial fields, which admits a reflection $\Pi : \MPF \to \PartFields$.
\item When $P$ is a partial field, its $\GRS$-lift $\cL_{\grs}(P)$ belongs to $\MPF$, and the associated partial field $\Pi(\cL_{\grs}(P))$ is equal to the Pendavingh-van Zwam lift $\cL_{\pvz}(P)$.
\end{enumerate}
\end{thmA}

Part (1) of \autoref{thmC} will be proved in \autoref{prop: universal property of GRS-lifts}, part (2) in \autoref{lem:MPF reflection}, and part (3) in \autoref{prop: PvZ-lift as quotient of GRS-lift}.

\medskip

As a formal consequence of (3), we obtain a proof of the Pendavingh--van Zwam idempotence conjecture, along with a new proof of the lift theorem from \cite{Pendavingh-vanZwam10b}:

\begin{cor*}\ 
\begin{enumerate}
\item $\Lift_{\pvz}$ is an idempotent functor from the category of partial fields to itself, i.e., $\Lift_{\pvz}(\Lift_{\pvz}(P))= \Lift_{\pvz}(P)$ for every partial field $P$.
\item For every partial field $P$ and every matroid $M$, every projective equivalence class of $P$-representations of $M$ lifts {uniquely} to $\Lift_{\pvz}(P)$.
\end{enumerate}
\end{cor*}

Part (1) of this corollary will be proved in \autoref{cor:idempotence conjecture}
 and part (2) in \autoref{thm: PvZ-lift theorem}.

\medskip

\subsection*{A more restrictive but more precise collection of coreflections}

To state the results in this section, it is convenient to restrict ourselves to the category $\FPastures$ of {\em finitary pastures}. We say that a pasture $P$ is {\em finitary} if $P^\times$ is finitely generated and $N_P / P^\times$ is finite. (The restriction to such pastures is not necessary, but it makes it easier to state our results.)

In an ``ideal world,'' there would be an explicitly computable coreflection from $\FPastures$ onto the subcategory $\Foundations$ consisting of all foundations of matroids. If we had such a coreflection, then by computing $\Lift_{\Foundations}(P_1 \times P_2)$ we could, for example, formulate sharp versions of all theorems of the form ``A matroid $M$ is representable over the pastures $P_1$ and $P_2$ if and only if it is representable over $P$.''
Unfortunately, there may not be such a coreflection, but category theory gives us a best possible substitute, a coreflection from $\Pastures$ onto the so-called {\em coreflective hull} of $\Foundations$ (cf.~\autoref{sec:generallifttheorem}), which we denote by $\Lifts$. It is not easy to explicitly compute $\Lifts$ or the coreflection onto it, however,
so we seek to approximate such an ideal result. 

There are two ways of doing this: ``from above''  (meaning constructing a coreflection onto a subcategory $\cD$ containing $\Lifts$) or ``from below'' (meaning constructing a coreflection onto a subcategory $\cD$ contained in $\Lifts$).
The $\GRS$-lift, which is an approximation from above, allows us to prove possibly non-sharp lifting results which hold for {\em all} matroids.
Approximations from below, such as the ones we are about to describe, allow us on the other hand to prove sharp lifting results for a {\em restricted} class of matroids.

\begin{thmA}\label{thmD}
The following subcategories of $\FPastures$ are coreflective:
\begin{enumerate}
\item The subcategory $\cB$ consisting of the foundations of all binary matroids. (Explicitly, the objects of $\cB$ are $\F_1^\pm$ and $\F_2$.)
\item The subcategory $\cT$ consisting of the foundations of all ternary matroids. (Explicitly, the objects of $\cT$ are all pastures of the form $P_1 \otimes \cdots \otimes P_k$ with $P_i \in \{ \F_3, \D, \H, \U \}$.)
\item The subcategory $\cW$ consisting of the foundations of all matroids without large uniform minors. (Explicitly, the objects of $\cW$ are all pastures of the form $P_1 \otimes \cdots \otimes P_k$ with $P_i \in \{ \F_2, \F_3, \D, \H, \U \}$.)
\end{enumerate}
Moreover, in each case the corresponding lift $\cL_{\cD} P$ of a pasture $P$ can be explicitly described.
\end{thmA}

Part (1) of \autoref{thmD} will be proved in \autoref{prop: universal property of binary lifts}, part (2) in \autoref{prop: universal property of ternary lifts}, and part (3) in \autoref{prop: universal property of WLUM-lifts}.

\medskip

Such `abstract nonsense' has useful concrete consequences. For example:

\begin{enumerate}
\item The binary lift $\cL_{\cB} \S$ of the sign hyperfield $\S$ is the regular partial field $\F_1^\pm$. In particular, we get a simple `explanation' for the celebrated fact that every binary orientable matroid is regular.
\item The ternary lift $\cL_{\cT} \S$ of the sign hyperfield is the dyadic partial field $\D$. In other words, every reorientation class of a ternary orientable matroid lifts uniquely to a rescaling class of dyadic representations. As $\cL_{\cW} \S$ is also isomorphic to $\D$, the same holds more generally for orientable matroids without large uniform minors. 
(This is the ``Generalized Lee--Scobee Theorem'' from \cite{Baker-Lorscheid20}.)
\item The ternary lift $\cL_{\cT} \F_4$ of the finite field of order $4$ is isomorphic to the hexagonal partial field $\H$. In other words, every quarternary representation of a ternary matroid $M$ lifts uniquely to a hexagonal representation, up to rescaling equivalence. 
\end{enumerate}

Numerous other concrete examples are given in \autoref{thm: unique lifts of representations of ternary matroids} and \autoref{thm: unique lifts of representations of matroids wlum} below.

\subsection*{Products of rescaling class spaces}

As another application of our lifting techniques, combining the definition of a coreflection, the fact that the foundation represents the functor $\cX_M(\cdot)$, and the universal property of products yields in a formal way:

\begin{cor*}
If $\cD$ is a coreflective subcategory of $\Pastures$ or $\FPastures$ then for every matroid $M$ with $F_M \in \cD$ and every triple of pastures $(P,P_1,P_2)$ with $\cL_{\cD} P \cong \cL_{\cD} (P_1 \times P_2)$, there is a natural bijection between $\cX_M(P)$ and $\cX_M(P_1) \times \cX_M(P_2)$.
\end{cor*}

For example, since $\cL_{\cT} \F_8$ and $\cL_{\cT}(\F_4 \times \F_5)$ are both isomorphic to the near-regular partial field $\U$, we find that 
 \[
  \cX_M(\F_8) \ = \ \cX_M(\F_4) \times \cX_M(\F_5)
 \]
 for every ternary matroid $M$.
(As mentioned earlier, there is a similar bijection whenever $q,p_1,p_2$ are prime powers with $3 \nmid q$ and $q-2 = (p_1 - 2)(p_2 - 2)$.)

\medskip

Similarly, since the $\GRS$-lift $\cL_{\cG}(\F_4 \times \F_5)$ is isomorphic to $\G = \cL_{\cG}(\G)$, we find that 
 \[
  \cX_M(\G) \ = \ \cX_M(\F_4) \times \cX_M(\F_5)
 \]
 for every (not necessarily ternary) matroid $M$.

\subsection*{Constructing the coreflections}

To conclude our introduction to the ideas contained in this paper, we give a brief outline of how the coreflections onto $\cG,\cB,\cT,\cW$ are defined.

\subsubsection*{Approximation from above}

Roughly speaking, the $\GRS$-lift of a pasture is defined by taking the same fundamental elements, with the same additive relations, but only including $2$-term and $3$-term multiplicative relations rather than all multiplicative relations. 

\medskip

More precisely, if $P$ is a pasture, its {\em $\GRS$-lift} is defined to be 
\[
\cL_{\grs}(P) := \past{\F_1^\pm \gen{t_a \; | \; a \in P^\fund}}S,
\]
where $S$ consists of the following relations:
 \begin{enumerate}[label={($\grs$\arabic*)}]
\item $1+1$, if $1+1 \in N_P$.
\item $t_a t_{a^{-1}}$ for $a \in P^\fund$.
\item $t_a + t_b - 1$ whenever $a+b-1 \in N_P$.
\item $t_a t_b t_c + 1$ whenever $a+b^{-1}-1 \in N_P$ and $abc + 1 \in N_P$.
\item $t_a t_b t_c - 1$ whenever $a,b,c \in P^\fund$ and $abc - 1 \in N_P$.
\end{enumerate}

The canonical morphism $\lambda_P : \cL_{\grs}(P) \to P$ sends $t_a$ to $a$, and induces a bijection on fundamental elements.

\medskip

We denote by $\cG$ the set of all pastures of the form $\cL_{\grs}(P)$ for some pasture $P$.
It follows from the results of \cite{Baker-Lorscheid20} that the foundation of every matroid belongs to $\cG$, 
and it's fairly straightforward to prove that $\cL_{\grs}(P)$ defines a coreflection from $\Pastures$ onto $\cG$.

\subsubsection*{Approximation from below}

The coreflection onto the subcategory $\cB = \{ \F_1^\pm, \F_2 \}$ of foundations of binary matroids is defined by setting $\cL_{\cB}(P) = \F_2$ if $1+1 \in N_P$ and $\cL_{\cB}(P) = \F_1^\pm$ otherwise.

To define the coreflections onto $\cT,\cW$, we use {\em fundamental pairs} rather than {\em fundamental elements} (as in the definition of the $\GRS$-lift).
For simplicity, we only consider the subcategory $\cD = \cT$ of ternary matroids here; the case $\cD = \cW$ is similar.

More precisely, if $P$ is a pasture, its {\em ternary lift} is defined to be
\[
\cL_{\cT}(P) := \past{\F_1^\pm \gen{t_{a,b} \; | \; (a,b) \in P^\pair}}S,
\]
where $S$ consists of the following relations:
\begin{enumerate}[label={($\ternary$\arabic*)}]
\item $t_{a,b}+t_{b,a}=1$.
\item $t_{a,b}\cdot t_{a^{-1},-a^{-1}b}=1$.
\item $t_{a,b}\cdot t_{-a^{-1}b,a^{-1}}\cdot t_{b^{-1},-ab^{-1}} =-1$.
\end{enumerate}

The canonical morphism $\lambda_P : \cL_{\cT}(P) \to P$ sends $t_{a,b}$ to $a$, and induces a bijection on fundamental pairs.

The ternary lift of $P$ decomposes as a tensor product of $P_\Xi$ over all hexagons $\Xi$ of $P$, where $P_\Xi \in \{ \F_3, \D, \H, \U \}$ depends only on the ``type'' of $\Xi$.
This is used to show that the set $\cT$ of all pastures of the form $\cL_{\cT}(P)$ for some $P \in \FPastures$ consists of all pastures of the form $P_1 \otimes \cdots \otimes P_k$ with $P_i \in \{ \F_3, \D, \H, \U \}$.
Using the classification of ternary foundations from \cite{Baker-Lorscheid20}, $\cT$ is precisely the set of foundations of ternary matroids.

It is once again fairly straightforward to prove that $\cL_{\cT}(P)$ defines a coreflection from $\FPastures$ onto $\cT$.

\subsection*{Content overview}

In \autoref{section: Background}, we present background material on pastures and foundations of matroids which is needed for what follows. 
In \autoref{section: GRS lift}, we first explore the coreflective hull of $\Foundations$ in $\Pastures$, which is of limited utility at the moment since it is rather difficult to compute. We then define the $\GRS$-lift of a pasture $P$ and establish its basic properties. The $\GRS$-lift is more explicit and easier to compute, but less precise than the lift to the coreflective hull. 
We also provide a comparison to the Pendavingh--van Zwam lift of a partial field and prove the Pendavingh--van Zwam idempotence conjecture. Finally, we define and establish the basic properties of the binary lift, which is too elementary to be truly useful but which provides a simple example of ``approximation from below''.
In \autoref{section: Hexagons}, we study the hexagons of a pasture $P$ in preparation for the definition of the ternary and $\WLUM$-lifts of a pasture $P$, which are presented in \autoref{section: Ternary lifts}. (These are more sophisticated and more interesting approximations from below.)
Applications to rescaling classes of matroids over various particular pastures are given in \autoref{section: Applications}. 


\subsection*{Acknowledgements}

We thank Rudi Pendavingh for inspiring conversations and Don Zagier for his suggestions on \autoref{rem: Zagier}. The first author was supported by a Simons Foundation Collaboration Grant and the second author was supported by a Marie Sk\l odowska-Curie Individual Fellowship.


\section{Background}
\label{section: Background}

In this section, we explain some notions from the introduction in more detail; also see \cite{Baker-Lorscheid18, Baker-Lorscheid20}.

\subsection{Algebras and quotients}

Let $P$ be a pasture with null set $N_P$ and $\{x_i\}_{i\in I}$ a set of indeterminates. The \emph{free $P$-algebra in $\{x_i\}$} is the pasture $P\gen{x_i\mid i\in I}$ whose unit group is $P\gen{x_i\mid i\in I}^\times=P^\times\times\gen{x_i\mid i\in I}$, where $\gen{x_i\mid i\in I}$ is the (multiplicatively written) free abelian group generated by the symbols $x_i$, and whose null set is
\[
 N_{P\gen{x_i\mid i \in I}} \ = \ \big\{ da+db+dc \, \big| \, d\in \gen{x_i\mid i\in I}, \, a+b+c\in N_P \big\},
\]
where $da$ stands for $(a,d)\in P\gen{x_i\mid i\in I}^\times$ if $a\neq0$ and for $0$ if $a=0$. This pasture comes with a canonical morphism $P\to P\gen{x_i\mid i\in I}$ of pastures that sends $a$ to $1a$. If $\{x_i\}=\{x_1,\dotsc,x_s\}$ is finite, then we usually write $P\gen{x_1,\dotsc,x_s}$ for $P\gen{x_i\mid i\in I}$.

Let $S\subset \Sym_3(P)$ be a set of elements of the form $a+b+c$ with $ab\neq 0$. We define the \emph{quotient $\past PS$ of $P$ by $S$} as the following pasture. Let $\tilde N_{\past PS}$ be the smallest subset of $\Sym_3(P)$ that is closed under property \ref{P2} and that contains $N_P$ and $S$. Since all elements $a+b+c$ in $S$ have at least two nonzero terms by assumption, $\tilde N_{\past PS}$ also satisfies \ref{P1}. Axiom \ref{P3} leads to the following quotient construction for $P^\times$.

We define the unit group $(\past PS)^\times$ of $\past PS$ as the quotient of the group $P^\times$ by the subgroup generated by all elements $a$ for which $a-1+0\in \tilde N_{\past PS}$. The underlying monoid of $\past PS$ is, by definition, $\{0\}\cup(\past PS)^\times$, and it comes with a surjection $\pi:P\to \past PS$ of monoids. We denote the image of $a\in P$ by $\bar a=\pi(a)$, and define the null set of $\past PS$ as the subset
\[
 N_{\past PS} \ = \ \big\{ \bar a+\bar b+\bar c \, \big| \, a+b+c\in \tilde N_{\past PS} \big\}
\]
of $\Sym_3(\past PS)$. The quotient $\past PS$ of $P$ by $S$ comes with a canonical map $P\to\past PS$ that sends $a$ to $\bar a$ and is a morphism of pastures.

If $S\subset\Sym_3(P\gen{x_i\mid i\in I})$ is a subset of relations of the form $a+b+c$ with $ab\neq 0$, then the composition of the canonical morphisms for the free algebra and for the quotient yields a canonical morphism 
\[
 \pi: \ P \ \longrightarrow \ P\gen{x_1,\dots,x_s} \ \longrightarrow \ \past{P\gen{x_i\mid i\in I}}S.
\]
We denote by $\pi_0:\{x_i\mid i\in I\}\to \past{P\gen{x_i\mid i\in I}}S$ the map that sends $x_i$ to $\bar x_i$.

\begin{prop}\label{prop: universal property of algebras and quotients}
 Let $P$ be a pasture, $\{x_i\}_{i\in I}$ an indexed set and $S\subset\Sym_3(P\gen{x_i\mid i\in I})$ a subset of elements of the form $a+b+c$ with $ab\neq 0$. Let $f:P\to Q$ be a morphism of pastures and $f_0:\{x_i\}\to Q^\times$ a map with the property that $a\prod x_i^{\alpha_i}+b\prod x_i^{\beta_i}+c\prod x_i^{\gamma_i}\in S$ with $a,b,c\in P$ and $(\alpha_i),(\beta_i),(\gamma_i)\in\bigoplus_{i\in I} \Z$ implies that
 \[\textstyle
  f(a)\prod f_0(x_i)^{\alpha_i}+ f(b)\prod f_0(x_i)^{\beta_i}+f(c)\prod f_0(x_i)^{\gamma_i}\in N_{Q}.
 \]
 Then there is a unique morphism $\hat f:\past{P\gen{x_1,\dotsc,x_s}}S\to Q$ such that the diagrams
 \[
  \begin{tikzcd}
   P \ar[r,"f"] \ar[d,"\pi"'] & Q \\
   \past{P\gen{x_i\mid i\in I}}S \ar[ur,"\hat f"']
  \end{tikzcd}
  \qquad \text{and} \qquad
  \begin{tikzcd}
   \{x_i\}_{i\in I} \ar[r,"f_0"] \ar[d,"\pi_0"'] & Q \\
   \past{P\gen{x_i\mid i\in I}}S \ar[ur,"\hat f"']
  \end{tikzcd}
 \]
 commute.
\end{prop}

\begin{proof}
 This is proven in \cite[Prop.\ 2.6]{Baker-Lorscheid20} for finite $\{x_i\}_{i\in I}=\{x_1,\dotsc,x_s\}$. The general case is analogous.
\end{proof}


\subsection{Examples}

The \emph{regular partial field} is the pasture
\[
 \Funpm \ = \ \{0,1,-1\} \qquad \text{with nullset} \qquad N_{\Funpm} \ = \ \{1-1\},
\]
which is an initial pasture, i.e.\ there is a unique morphism $\Funpm\to P$ for every pasture $P$. In particular, every other pasture $P$ is an $\Funpm$-algebra and $P\simeq \pastgen\Funpm{x_i\mid i\in I}{S}$ for some $I$ and $S\subset\Sym_3\Big(\Funpm\gen{x_i\mid i\in I}\Big)$. The \emph{Krasner hyperfield} is the pasture
\[
 \K \ = \ \past\Funpm{\{1+1,1+1+1\}}, 
\]
whose underlying monoid is $\{0,1\}$ and whose nullset is $N_{\K}=\{1+1,1+1+1\}$. It is a terminal pasture, i.e.\ there is a unique morphism $t_P:P\to\K$ for every pasture $P$, which we call the \emph{terminal map}.

\subsubsection*{Fields as pastures}

We can associate with a field $K$ the following pasture $P$: as multiplicative monoids, we define $P=K$; the nullset of $P$ consists of all $a+b-c\in\Sym_3(P)$ such that $a+b=c$ in $K$. Note that a map $f:K_1\to K_2$ between fields is a field homomorphism if and only if it is a morphism between the associated pastures $f:P_1\to P_2$.

For example, we have
\[
 \F_2 \ = \ \past\Funpm{\{1+1\}} \qquad \text{and} \qquad \F_3 \ = \ \past\Funpm{\{1+1+1\}}.
\]

\subsubsection*{Partial fields as pastures}

Following \cite{Pendavingh-vanZwam10b} (see also \cite{Baker-Lorscheid18,Pendavingh-vanZwam10a,Semple-Whittle96}), a partial field is given by a pair $(G,R)$ of a ring $R$ together with a subgroup $G$ of the unit group $R^\times$ that contains $-1$. The associated pasture is $P=G\cup\{0\}$, as a pointed monoid, together with the nullset $N_P$ that consists of all elements $a+b-c\in\Sym_3(P)$ such that $a+b=c$ in $R$. Let $(G_1,R_1)$ and $(G_2,R_2)$ be partial fields with respective associated pastures $P_1$ and $P_2$. Then a map $f:G_1\to G_2$ between partial fields is homomorphism if and only if the rule $0\mapsto 0$ extends it a morphism between the associated pastures $f:P_1\to P_2$.

Let $(G,R)$ be a partial field and $P$ the associated pasture. The universal ring of $(G,R)$ in the sense of \cite[section 4.2]{Pendavingh-vanZwam10a} can be expressed as $R_{(G,R)}=\Z[P^\times]/\gen{N_P}$ where we identify $0\in P$ with the zero in the group semiring $\Z[P^\times]$.

Examples of partial fields are the regular partial field $\Funpm$, as well as
\begin{align*}
 & \text{the \emph{near-regular partial field}} & \U \ &= \ \pastgen\Funpm{x,y}{x+y-1}; \\
 & \text{the \emph{dyadic partial field}}       & \D \ &= \ \pastgen\Funpm{z}{z+z-1}; \\
 & \text{the \emph{hexagonal partial field}}    & \H \ &= \ \pastgen\Funpm{z}{z^3 + 1, \; z+z^{-1}-1}; \\
 & \text{the \emph{golden ratio partial field}} & \G \ &= \ \pastgen\Funpm{z}{z^2+z-1}.
\end{align*}

\begin{rem}\label{rem: universal ring of a pasture}
 We can define for every pasture $P$ a \emph{universal ring} $R_P=\Z[P^\times]/\gen{N_P}$, which comes with a multiplicative map $P\to R_P$. This lets us characterize pastures that come from partial fields: $P$ is a pasture associated with a partial field $(G,P)$ if and only if $P\to R_P$ is injective and if $N_P=\gen{N_P}\cap\Sym_3(P)$, i.e.\ $N_P$ contains every element $a+b+c\in R_P$ of the ideal $\gen{N_P}$ where $a,b,c\in P$. In this case, the identification $P^\times=G$ defines an isomorphism of partial fields $(P^\times,R_P)\to (G,R)$. By abuse of terminology, we will say in the following that a pasture $P$ is a partial field if $P\to R_P$ is injective and $N_P=\gen{N_P}\cap\Sym_3(P)$.
\end{rem}

\begin{ex}
The pasture $P=\pastgen\Funpm{x,y}{x+y-1,x^3+xy+1}$ embeds into its universal ring $R_P=\Z[x,y]/\gen{x+y-1,x^3+xy+1}\simeq\Z[x]/\gen{x^3-x^2+x+1}$, but the relation
 \[
  x^3-x^2+y \ = \ (x+y-1)-x\cdot(x+y-1)+(x^3+xy+1) \ = \ 0
 \]
 does not hold in $P$, which shows that $P$ is not a partial field.
\end{ex}


\subsubsection*{Hyperfields as pastures}
Hyperfields (introduced by Krasner in \cite{Krasner56}) are, roughly speaking, like fields except that addition is allowed to be multi-valued (see also \cite{Baker-Bowler19,Baker-Lorscheid18}).
A hyperfield $K$ can be identified with the pasture $P$ that equals $K$ as multiplicative monoid and whose nullset consists of all $a+b-c\in\Sym_3(P)$ such that $c\in a\hyperplus b$.

Examples of hyperfields are the Krasner hyperfield $\K$, as well as
\begin{align*}
 & \text{the \emph{sign hyperfield}}      & \S \ &= \ \past\Funpm{\{1+1-1\}}; \\
 & \text{the \emph{weak sign hyperfield}} & \W \ &= \ \past\Funpm{\{1+1+1,\ 1+1-1\}}.
\end{align*}


\subsection{Products and tensor products}
\label{sec:prodandtensorsprod}

Let $\{P_i\}_{i\in I}$ be a family of pastures. The \emph{product of $\{P_i\}$} is defined as follows. For empty $I$, we set $\prod_{i\in I} P_i=\K$. If $I$ is non-empty, then we define the pointed monoid
\[
 \prod_{i\in I} P_i \ = \ \{0\} \cup \bigg(\prod_{i\in I} P_i^\times\bigg)
\]
where $\prod P_i^\times$ is the Cartesian product of the abelian groups $P_i^\times$ and $0\cdot(a_i)=0$ for all $(a_i)\in \prod P_i^\times$. This monoid comes together with \emph{canonical projections} $\pi_j:\prod P_i\to P_j$ that are defined by $\pi_j\Big((a_i)\Big)=a_j$ for $(a_i)\in \prod P_i^\times$ and $\pi_j(0)=0$. The nullset of $\prod P_i$ is
\[\textstyle
 N_{\prod P_i} \ = \ \Big\{ a+b+c\in \Sym(\prod P_i) \, \Big| \, \pi_j(a)+\pi_j(b)+\pi_j(c)\in N_{P_j} \text{ for all }j\in I \Big\}.
\]
Note that the canonical projections $\pi_j:\prod P_i\to P_j$ are morphisms of pastures. The product $\prod P_i$ satisfies the following universal property (\cite[Lemma 2.2]{Creech21}):

\begin{lemma}\label{lemma: universal property of products}
 Let $\{P_i\}_{i\in I}$ be a family of pastures. For every family $\{\varphi_i:Q\to P_i\}_{i\in I}$ of pasture morphisms, there is a unique pasture morphism $\Phi:Q\to\prod P_i$ such that the diagram
 \[
  \begin{tikzcd}[column sep=80pt, row sep=20pt]
   Q \ar[r,"\Phi"] \ar[dr,"\varphi_j"'] & \prod P_i \ar[d,"\pi_j"] \\
                                        & P_j
  \end{tikzcd}
 \]
 commutes for every $j\in I$.
\end{lemma}

\begin{lemma}\label{lemma: product of partial fields is a partial field}
 The product of partial fields is a partial field.
\end{lemma}

\begin{proof}
 This follows from \cite[Lemma 2.17]{Pendavingh-vanZwam10b}, observing that the construction of the product of partial fields agrees with the construction of products of pastures.
\end{proof}

Let $\{P_i\}_{i\in I}$ be a family of pastures. For empty $I$, we set $\bigotimes_{i\in I} P_i=\Funpm$. If $I$ is non-empty, then we define the pointed monoid
\[
 \widehat P \ = \ \{0\} \cup \bigg(\bigoplus_{i\in I} P_i^\times\bigg)
\]
where $\bigoplus P_i^\times$ is the direct sum of the abelian groups $P_i^\times$ and $0\cdot(a_i)=0$ for all $(a_i)\in \bigoplus P_i^\times$. This monoid comes together with monoid injections $\hat\iota_j:P_j\to \widehat P$ that are defined by $\hat\iota(0)=0$ and $\hat\iota_j(a)=(a_i)$ with $a_j=a$ and $a_i=1$ for $i\neq j$ if $a\neq 0$. The \emph{tensor product of $\{P_i\}$} is defined as 
\[
 \bigotimes_{i\in I} P_i \ = \ \past{\widehat P}S
\]
where
\[
 S \ = \ \left\{ a+b+c\in \Sym(\widehat P) \, \left| \, \begin{array}{c}a+b+c \ = \ \hat\iota_j(a')+\hat\iota_j(b')+\hat\iota_j(c')\\ \text{for a $j\in I$ and $a'+b'+c'\in N_{P_j}$}\end{array} \right.\right\}.
\]
Note that the underlying monoid of $\bigotimes P_i$ is the quotient of $\widehat P$ by the equivalence relation generated by $\hat\iota_i(-1)\sim \hat\iota_j(-1)$ for all $i,j\in I$. The composition of $\hat\iota_j$ with the quotient map $\widehat P\to\bigotimes P_i$ defines the $j$-th canonical inclusion $\iota_j:P_j\to \bigotimes P_i$, which is a morphism of pastures. 

The tensor product $\bigotimes P_i$ satisfies the following universal property (\cite[Lemma 3.5]{Creech21}):

\begin{lemma}\label{lemma: universal property of tensor products}
 Let $\{P_i\}_{i\in I}$ be a family of pastures. For every family $\{\varphi_i:P_i\to Q\}_{i\in I}$ of pasture morphisms, there is a unique pasture morphism $\Phi:\bigotimes P_i\to Q$ such that the diagram
 \[
  \begin{tikzcd}[column sep=80pt, row sep=20pt]
   \bigotimes P_i \ar[r,"\Phi"] & Q \\
   P_j \ar[u,"\iota_j"] \ar[ur,"\varphi_j"'] 
  \end{tikzcd}
 \]
 commutes for every $j\in I$.
\end{lemma}

\begin{rem}
 Note that the tensor product of (partial) fields is not necessarily a (partial) field. For example, none of 
 \[
  \F_2\otimes\F_3 \ \simeq \ \K, \qquad \F_2\otimes\D \qquad \text{and} \qquad \F_3\otimes\D 
 \]
 is a partial field. This is obvious for $\K$. In the latter two cases, assume that $\F_q\otimes\D$ (for $q=2,3$) occurs as a submonoid of a ring $R$ with $z+z=1$ and $1+1=0$ (if $q=2$) or $1+1=-1$ (if $q=3$). Then $z^{-1}=1+1\in\{0,-1\}$ in $R$, which contradicts the fact that $z^{-1}\notin\{0,-1\}$ in $\F_q\otimes\D$.
\end{rem}


\subsection{Matroid representations over pastures}

Given two subsets $I$ and $J$ of $E$, we denote by $I-J=\{i\in I\mid i\notin J\}$ the complement of $J$ in $I$. For an ordered tuple $\bJ=(j_1,\dotsc,j_s)$ in $E^s$, we denote by $|\bJ|$ the subset $\{j_1,\dotsc,j_s\}$ of $E$. Given $k$ elements $e_{1},\dotsc,e_k\in E$, we denote by $\bJ e_{1}\dotsb e_{k}$ the $s+k$-tuple $(j_1,\dotsc,j_{s},e_1,\dotsc,e_k)\in E^{s+k}$. For a subset $J$ of $E$, we denote by $Je_{1}\dotsb e_{k}$ the subset $J\cup\{e_1,\dotsc,e_k\}$ of $E$. In particular, we have $|\bJ e_{1}\dotsb e_{k}|=|\bJ|e_{1}\dotsb e_{k}$ for $\bJ\in E^s$.

\begin{df}
 Let $M$ be a matroid of rank $r$ on $E=\{1,\dotsc,n\}$ and $P$ a pasture. A \emph{$P$-representation of $M$} is a function $\Delta:E^r\to P$ such that
 \begin{enumerate}
  \item $\Delta(j_1,\dotsc,j_r)\neq0$ if and only if $\{j_1,\dotsc,j_r\}$ is a basis of $M$;
  \item $\Delta$ is \emph{alternating}, i.e.\
        \[
         \Delta(j_{\sigma(1)},\dotsc,j_{\sigma(r)}) \ = \ \sign(\sigma)\Delta(j_1,\dotsc,j_r)
        \]
        for all $(j_1,\dotsc,j_r)\in E^r$ and $\sigma\in S_r$ where we consider $\sign(\sigma)\in\{\pm1\}$ as an element of $P$;
  \item $\Delta$ satisfies the \emph{$3$-term Pl\"ucker relations}
        \[
         \Delta(\bJ e_1e_2)\cdot\Delta(\bJ e_3e_4)-\Delta(\bJ e_1e_3)\cdot\Delta(\bJ e_2e_4)+\Delta(\bJ e_1e_4)\cdot \Delta(\bJ e_2e_3) \ = \ 0
        \]
        for all $\bJ\in E^{r-2}$ and $e_1,\dotsc,e_4\in E$.
 \end{enumerate}
 A matroid $M$ is said to be \emph{representable over $P$} if it has a $P$-representation $\Delta:E^r\to P$.
\end{df}

Note that the definition of representability agrees with the usual terminology of representability over a partial field $P$, i.e.\ a matroid $M$ is representable over a partial field $P$ if and only if $M$ is representable by a $P$-matrix in the sense of \cite{Pendavingh-vanZwam10b} (cf.\ \cite[Prop.\ 3.9]{Baker-Lorscheid18} for a proof). Moreover, $M$ is representable over $\S$ if and only if $M$ is orientable, and $M$ is representable over $\W$ if and only if $M$ is weakly orientable (cf.\ \cite{Baker-Bowler19}).

Given a $P$-representation $\Delta:E^r\to P$ of $M$ and a pasture morphism $\varphi:P\to Q$, we define the \emph{push-forward of $\Delta$ along $\varphi$} as the map
\[
 \begin{array}{cccc}
  \varphi_\ast(\Delta): & E^r & \longrightarrow & Q \\
                        & \bI & \longmapsto     & \varphi\Big(\Delta(\bI)\Big),
 \end{array}
\]
which is easily verified to be a $Q$-representation of $M$. 

In particular, this shows that if $M$ is representable over $P$ and if there is a pasture morphism $P\to Q$, then $M$ is representable over $Q$.


\subsection{Rescaling classes} 
\label{sec:rescaling}

Let $M$ be a matroid of rank $r$ on $E$ and $P$ a pasture. Two $P$-representation $\Delta:E^r\to P$ and $\Delta':E^r\to P$ of $M$ are \emph{rescaling equivalent} if there exist $c \in P^\times$ and a map $d : E \to P^\times$ such that $\Delta'(e_1,\ldots,e_r) = c \cdot d(e_1) \cdots d(e_r) \cdot \Delta(e_1,\ldots,e_r)$ for all $(e_1,\ldots,e_r) \in E^r$. Note that this relation is an equivalence relation on the set of all $P$-representations of $M$. 

\begin{df}
 Let $\Delta:E^r\to P$ be a $P$-representation of $M$. The \emph{rescaling class of $\Delta$} is the class $[\Delta]$ of $P$-representations $\Delta':E^r\to M$ that are rescaling equivalent to $\Delta$. The \emph{rescaling class space of $M$ over $P$} is the set $\cX_M(P)$ of rescaling classes $[\Delta]$ of $P$-representations $\Delta:E^r\to P$ of $M$.
\end{df}

The definition of $\cX_M(P)$ is functorial in $P$, in the sense that a pasture morphism $\varphi:P\to Q$ defines a map $\cX_M(P)\to\cX_M(Q)$ that sends the rescaling class $[\Delta]$ of a $P$-representation $\Delta:E^r\to P$ of $M$ to the rescaling class $\varphi_\ast([\Delta])=[\varphi_\ast(\Delta)]$. In other words, $\cX_M(-)$ is a functor from the category of pastures to the category of sets.

\begin{rem}
 The fundamental fact for many applications is that $M$ is representable over $P$ if and only if $\cX_M(P)$ is not empty, cf.\ \cite[section 6]{Baker-Lorscheid20}. In this paper, we study morphisms $\varphi:P\to Q$ of pastures that induce a bijection $\varphi_\ast:\cX_M(P)\to\cX_M(Q)$ for certain classes of matroids $M$, which leads to a unique lifting of representations of $M$ up to rescaling equivalence.
\end{rem}


\subsection{The foundation of a matroid}

The foundation $F_M$ of a matroid $M$ has been introduced in \cite{Baker-Lorscheid18}; cf.\ \cite{Baker-Lorscheid20} for the description of $F_M$ as a pasture. For the purpose of this paper, it suffices to define the foundation in terms of its universal property, which characterizes it up to unique isomorphism.

\begin{df}
 Let $M$ be a matroid. A \emph{foundation of $M$} is a pasture $F_M$ together with a functorial identification $\Hom(F_M,P)=\cX_M(P)$. 
\end{df}

In other words, the foundation of $M$ is a pasture $F_M$ together with a \emph{universal rescaling class $C$ of $M$ over $F_M$}, which corresponds to $\id_{F_M}\in\Hom(F_M,F_M)$, such that for every pasture $P$ and every rescaling class $[\Delta]\in\cX_M(P)$ there is a unique pasture morphism $\varphi:F_M\to P$ with $\varphi_\ast(C)=[\Delta]$. In particular, this means that $M$ is representable over $P$ if and only if there exists a morphism $F_M\to P$.

\begin{thm}[{\cite[Cor.\ 7.26]{Baker-Lorscheid18}}]\label{thm: characterizing property of the foundation}
 Every matroid $M$ has a foundation $F_M$, which is unique up to a unique isomorphism that is compatible with the functorial identification $\Hom(F_M,P)=\cX_M(P)$. 
\end{thm}


\section{Lifts to coreflective subcategories}
\label{section: GRS lift}

Every coreflective subcategory of $\Pastures$ gives rise to a lift theorem for representations of all matroids whose foundation is contained in the coreflective subcategory. 
Namely, a coreflective subcategory $\cC$ of $\Pastures$ comes together with a coreflection $\cL:\Pastures\to\cC$ and a morphism (the counit of the adjunction) $\lambda_P:\cL P\to P$ for every pasture $P$. The universal property for the coreflection $\cL$ implies that every matroid representation of a matroid with foundation in $\cC$ \emph{lifts uniquely up to rescaling equivalence along $\lambda_P$}, a notion that is defined as follows.

\begin{df}
 Let $\varphi:Q\to P$ be a pasture morphism, $M$ a matroid of rank $r$ on $E$ and $\Delta:E^r\to P$ a $P$-representation of $M$. The $P$-representation $\Delta$ \emph{lifts to $Q$ (along $\varphi$)} if there is a $Q$-representation $\widehat\Delta:E^r\to Q$ of $M$ such that $\Delta=\varphi_\ast(\widehat\Delta)$. We call $\widehat\Delta$ a \emph{lift of $\Delta$ (along $\varphi$)}. The lift $\widehat\Delta$ of $\Delta$ is \emph{unique up to rescaling equivalence} if every other lift of $\Delta$ is rescaling equivalent to $\widehat\Delta$.
\end{df}

The notion of lifts of $P$-representations coincides with the notions of lifts of representations over partial fields as well as with the notion of lifts of matroid orientations along the map $\sign:\R\to\S$, which is naturally a morphism of pastures.

If the pasture morphism $\varphi:Q\to P$ is not injective, then lifts along $\varphi$ usually fail to be unique since rescaling a given lift by elements of the kernel of $\varphi$ produces further lifts. Therefore we are interested in uniqueness up to rescaling equivalence.

In this section, we explain the relation between coreflective subcategories and lift theorems and provide several instances of such theorems.


\subsection{The lift theorem for matroids}
\label{sec:generallifttheorem}

The strongest lift theorem that applies to all matroids, which we call simply the \emph{lift theorem for matroids}, can be derived from the following fact from category theory; cf.\ Lemma 6.1 and the following remark in \cite{Tholen87}. A category $\cC$ is called \emph{cowell-powered} if for every object $A$ in $\cC$ the class of epimorphisms with domain $A$ modulo isomorphisms is a set.

\begin{lemma}\label{lemma: coreflective hull}
 Let $\cC$ be a cocomplete and cowell-powered category. Let $\cD$ be a small and full subcategory of $\cC$ and $L(\cD)$ the closure of $\cD$ under colimits (computed in $\cC$). Then $L(\cD)$ is the smallest coreflective subcategory of $\cC$ that contains $\cD$, and it is called the \emph{coreflective hull of $\cD$ in $\cC$}.
\end{lemma}


The category $\Pastures$ is cocomplete (cf.\ \cite{Creech21}) and cowell-powered.\footnote{That $\Pastures$ is cowell-powered can be proven as follows: an eqimorphism of pastures $\pi:P\to Q$ is the same as a surjective morphism. The surjective morphisms with fixed domain $P$, modulo isomorphisms, are in bijection with the inverse images $\pi^{-1}(N_Q)=\{a+b+c\in\Sym_3(P)\mid \pi(a)+\pi(b)+\pi(c)\in N_Q\}$, which are subsets of $\Sym_3(P)$. Thus the class of epimorphisms $\pi:P\to Q$ with domain $P$, modulo isomorphisms, is in bijection with a subset of the power set of $\Sym_3(P)$, and is therefore a set.} The full subcategory $\Foundations$ of all foundations of matroids in $\Pastures$ is small since the class of all matroids forms a countable set. 

Thus we can apply \autoref{lemma: coreflective hull} to define $\Lifts$ as the coreflective hull of $\Foundations$ in $\Pastures$, and we denote the corresponding coreflection by $\cL:\Pastures\to\Lifts$. A \emph{lift} is a pasture in $\Lifts$. The properties of the coreflection imply that every pasture $P$ comes with an \emph{associated lift} $\cL P$ and a canonical morphism $\lambda_P:\cL P\to P$, which satisfy the following universal property: every pasture morphism $\alpha:L\to P$ from a lift $L$ to $P$ factors into a uniquely determined morphism $\hat\alpha:L\to\cL P$ composed with $\lambda_P:\cL P\to P$.

The lift has the following relevance for matroid representations.

\begin{thm}[Lift theorem for matroids]\label{thm: lifts}
 Let $M$ be a matroid and $P$ a pasture with lift $\lambda_P:\cL P\to P$. Then every representation $\Delta:E^r\to P$ of $M$ lifts uniquely up to rescaling equivalence to $\cL P$ along $\lambda_P$.
\end{thm}

\begin{proof}
 A representation $\Delta:E^r\to P$ induces a pasture morphism $\alpha:F_M\to P$ from the foundation $F_M$ of $M$ to $P$. By its very definition, $\Lifts$ contains $F_M$. Thus the universal property of lifts yields a unique morphism $\hat\alpha:F_M\to \cL P$ such that $\alpha=\lambda_P\circ\hat\alpha$. Since $F_M$ represents the rescaling classes of $M$, this means that the representation $\Delta$ of $M$ in $P$ lifts uniquely to $\cL P$ up to rescaling equivalence.
\end{proof}

It is clear from the construction that $\cL P$ is the strongest idempotent functorial lift from $\Pastures$ to a subcategory that contains all foundations of $P$-matroids such that every $P$-representation of a matroid lifts uniquely up to rescaling equivalence to $\cL P$. Unfortunately, we do not know at present how to compute $\cL P$ in general.

In the upcoming sections, we develop techniques to approximate $\cL P$ from above and below. By an approximation from above, we mean a coreflective subcategory $\Lifts_{\grs}$ of $\Pastures$ that contains $\Lifts$. The $\GRS$-lift is such an approximation from above, which can describe in terms of an explicit construction; cf.\ \autoref{subsection: GRS-lift}.

By an approximation from below, we mean a coreflective subcategory that is contained in $\Lifts$. In this case, we can only lift representations of those matroids whose foundation is contained in the smaller subcategory. We will explain explicit constructions of such lifts for the coreflective hulls of the foundations of binary, ternary and $\WLUM$ matroids. Before we embark on the more subtle construction of lifts for ternary and $\WLUM$ matroids in \autoref{section: Ternary lifts}, we explain lifts for binary matroids in \autoref{subsection: The lift theorem for binary matroids} as a first example of an approximation from below.


\subsection{The \texorpdfstring{$\GRS$}{GRS}-lift}
\label{subsection: GRS-lift}

In this section, we approximate $\Lifts$ from above by a coreflective subcategory $\Lifts_{\grs}$ for which we can explicitly construct the coreflection $\cL_{\grs}:\Pastures\to\Lifts_{\grs}$.

\begin{df}
 Let $P$ be a pasture. A \emph{fundamental element in $P$} is a unit $a\in P^\times$ such that $a+b-1\in N_P$ for some $b\in P^\times$. We denote the set of fundamental elements in $P$ by $P^\fund$. 
\end{df}

Note that a pasture morphism $f:P\to Q$ maps fundamental elements to fundamental elements since $a+b-1\in N_P$ implies that $f(a)+f(b)-1\in N_Q$. We denote the restriction of $f$ to the respective subsets of fundamental elements by $f^\fund:P^\fund\to Q^\fund$.

\begin{df}
 The \emph{GRS-lift of $P$} is the pasture 
 \[
  \Lift_{\grs} P \ = \ \pastgen\Funpm{t_a\mid a\in P^\fund}{S}
 \]
 where $S$ consists of the relations
 \begin{enumerate}[label={($\grs$\arabic*)}]
  \item\label{GRS1} $1+1$ if $-1 = 1$ in $P$;
  \item\label{GRS2} $t_a \cdot t_{a^{-1}} - 1$ for all $a \in P^\fund$;
  \item\label{GRS3} $t_a + t_b - 1$ whenever $a+b-1 \in N_P$;
  \item\label{GRS4} $t_a t_b t_c + 1$ whenever $a+b^{-1}-1 \in N_P$ and $abc=-1$ in $P$;
  \item\label{GRS5} $t_a t_b t_c - 1$ whenever $abc=1$ in $P$;
 \end{enumerate}
together with the canonical morphism
 \[  
  \begin{array}{cccc}
   \lambda_{\GRS,P}: & \Lift_{\grs} P & \longrightarrow & P \\
                & t_a & \longmapsto     & a.
  \end{array}
 \]
\end{df}

If the context is clear, we will use the shorthand notation $\lambda_P=\lambda_{\GRS,P}$. It is straightforward to check, using the definition of $\Lift_{\grs} P$, that $\lambda_P$ is a morphism of pastures and that $\lambda_P^\fund : \Lift_{\grs} P^\fund \to P^\fund$ is a bijection.

\begin{ex} \label{ex:Glift}
$\Lift_{\grs}(\F_4\times\F_5)\simeq\G$. Indeed, the fundamental elements of $\F_4$ are $\alpha$ and $\alpha^2$, where $\alpha^2+\alpha=1$, and the fundamental elements of $\F_5$ are $2,3,4$. It follows from the explicit description of the product in \autoref{sec:prodandtensorsprod} that the fundamental elements of $\F_4 \times \F_5$ are
$a = (\alpha,2), b=(\alpha,4), c=(\alpha,3)$ and their multiplicative inverses $(\alpha^2,3),(\alpha^2,4),(\alpha^2,2)$, respectively, with $a+b^{-1} = 1$, $b+c^{-1}= 1$, $c+a^{-1}= 1$. 
The only 3-term multiplicative relations of the form $xyz=-1$ with $x,y,z$ belonging to $\{ a,b,a^{-1},b^{-1},c,c^{-1} \}$ and $x+y^{-1}=1$ are $abc = -1$ and the inverse relation $a^{-1} b^{-1} c^{-1} = -1$.
Similarly, the only 3-term multiplicative relations of the form $xyz=1$ with $x,y,z \in \{ a,b,a^{-1},b^{-1},c,c^{-1} \}$ are
$a^2 b = 1$, $c^2b=1$ and their respective inverses. 
It follows (using \ref{GRS2} to eliminate $t_{a^{-1}},t_{b^{-1}},t_{c^{-1}}$ from the set of generators) that   
\[
\Lift_{\grs} P \ = \ \pastgen\Funpm{t_a,t_b,t_c}{t_a + t_{b}^{-1} - 1, t_b + t_{c}^{-1} - 1, t_c + t_{a}^{-1} - 1,t_a t_b t_c + 1, t_a^2 t_b - 1, t_c^2 t_b - 1}.
\]

The relations $t_b=t_a^{-2}$ and $t_c=-t_a^{-1}t_b^{-1}=-t_a$ allow us to eliminate $t_b$ and $t_c$ from the set of generators; simplifying the other relations accordingly and writing $z=t_a$ yields (after some bookkeeping)
\[
 \Lift_{\grs}(\F_4\times\F_5) \ = \ \pastgen\Funpm{z}{z^2+z-1} \ = \ \G.
\]
\end{ex}

\begin{df}
 We define $\Lifts_{\grs}$ as the full subcategory of $\Pastures$ whose objects are those pastures $P$ for which $\lambda_{P} : \Lift_{\grs} P \to P$ is an isomorphism. We call a pasture $P$ a \emph{$\GRS$-lift} if it is in $\Lifts_{\grs}$.
\end{df}

\begin{prop}\label{prop: universal property of GRS-lifts}
 The association $\Lift_{\grs}$ defines a coreflection from $\Pastures$ to $\Lifts_{\grs}$, i.e., for every morphism $\alpha : L \to P$ from a $\GRS$-lift $L$ to a pasture $P$, there is a unique $\hat\alpha : L \to \Lift_{\grs} P$ such that $\alpha = \lambda \circ \hat\alpha$.
\end{prop}

\begin{proof}
Define $\hat\alpha^\fund : L^\fund \to \Lift_{\grs} P^\fund$ by $a \mapsto t_{\alpha(a)}$. This extends uniquely to a group homomorphism $\hat\alpha^\times : L^\times \to \Lift_{\grs} P^\fund$ since:
\begin{itemize}
\item $L^\fund$ generates $L^\times$.
\item $\hat\alpha^\times(a^{-1}) = t_{\alpha(a^{-1})} = t_{\alpha(a)^{-1}} = t_{\alpha(a)}^{-1} = \hat\alpha^\times(a)^{-1}$.
\item $c^{-1} = ab$ if and only if $abc = 1$, which implies $\alpha(a) \alpha(b) \alpha(c) = 1$. Therefore $\hat\alpha^\times(ab) = \hat\alpha^\times(c^{-1}) = t_a t_b t_c t_{c^{-1}} = t_a t_b = \hat\alpha^\times(a) \hat\alpha^\times(b)$.
\item $abc = -1$ and $a+b = 1$ implies $\alpha(a) \alpha(b) \alpha(c) = -1$ and $\alpha(a) + \alpha(b) = 1$, so that $t_a t_b t_c = -1$ and thus $\hat\alpha(a) \hat\alpha(b) \hat\alpha(c) = -1$.
\item If $-1 = 1$ in $L$ then $-1 = 1$ in $P$ and thus $-1 = 1$ in $\Lift_{\grs} P$.
\end{itemize}

The group homomorphism $\hat\alpha^\times$ extends uniquely to a morphism of pastures $\hat\alpha : L \to \Lift_{\grs} P$ by sending $0$ to $0$, since $a + b = 1$ in $L$ implies $\alpha(a) + \alpha(b) = 1$ in $P$ and thus $t_{\alpha(a)} + t_{\alpha(b)} - 1 = \hat\alpha(a) + \hat\alpha(b) - 1 \in N_{\Lift_{\grs} P}$.

This proves the existence of the lift $\hat\alpha$, and uniqueness is clear by construction.
\end{proof}

\begin{thm}[$\GRS$-lift theorem for matroids]\label{thm: GRS-lifts}
 Let $M$ be a matroid and $P$ a pasture with $\GRS$-lift $\lambda_P:\cL_{\grs} P\to P$. Then every representation $\Delta:E^r\to P$ of $M$ lifts uniquely up to rescaling equivalence to a representation $\hat\Delta:E^r\to \cL_{\grs} P$ along $\lambda_P$.
\end{thm}

\begin{proof}
 Let $F_M$ be the foundation of $M$ and $\alpha:F_M\to P$ be the morphism induced by $\Delta$. Since $F_M$ represents the rescaling classes of $M$, the claim of the theorem amounts to the same as the assertion that $\alpha$ factors into a uniquely determined morphism $\hat\alpha:F_M\to\cL_{\grs} P$ composed with $\lambda_P$. This follows from \autoref{prop: universal property of GRS-lifts} once we have proven that $F_M$ is in $\Lifts_{\grs}$.
 
 This latter claim follows from the author's version \cite[Thm.\ 4.19]{Baker-Lorscheid20} of Theorem 4 in Gelfand-Rybnikov-Stone's paper \cite{Gelfand-Rybnikov-Stone95}, which exhibits a complete set of relations between cross ratios. Since the foundation $F_M$ is generated by its cross ratios and all relations from \cite[Thm.\ 4.19]{Baker-Lorscheid20} are preserved by the $\GRS$-lift, we conclude that $\cL_{\grs} F_M\simeq F_M$, which concludes the proof.
\end{proof}

As a concrete application of \autoref{thm: GRS-lifts}, we have the following sharpening of Vertigan's Theorem (proved in \cite[Theorem 4.9]{Pendavingh-vanZwam10b}) that a matroid is representable over both $\F_4$ and $\F_5$ if and only if it is representable over the golden ratio partial field $\G$:

\begin{thm} \label{thm:Vertigan}
Let $M$ be a matroid, and let $\cX_M(P)$ denote the rescaling class space of $M$ over $P$. 
There is a canonical bijection between $ \cX_M(\G) $ and $ \cX_M(\F_{4}) \times \cX_M(\F_{5})$, i.e.,
every pair consisting of a projective equivalence class of quaternary (resp. quinternary) representations lifts uniquely to a projective equivalence class of golden ratio representations.
\end{thm}

\begin{proof}
This follows from \autoref{ex:Glift} and \autoref{thm: GRS-lifts}, together with the identity $\cX_M(P)=\Hom(F_M,P)$ and the universal property of products.
\end{proof}


\subsection{Relation to the Pendavingh--van Zwam lift}

The analogues of the results for the $\GRS$-lifts also hold for Pendavingh-van Zwam lifts of partial fields, as introduced in \cite{Pendavingh-vanZwam10b}. Instead of repeating an adapting the same arguments to the partial field context, we use some `abstract nonsense' arguments from category theory to compare the two lifts and deduce the latter facts from the more general theorems for $\GRS$-lifts. 

Recall from \cite[section 2.2]{Baker-Lorscheid18} that a pasture $P$ is a partial field if and only if:
\begin{enumerate}
\item The natural map $P \to R_P = \Z[P^\times] / \langle N_P \rangle$ is injective. ($R_P$ is called the \emph{universal ring} of $P$ and $\langle N_P \rangle$ denotes the ideal generated by $N_P$ in the ring $\Z[P^\times]$.)
\item For all $a,b,c \in P$ with $a+b+c \in \langle N_P \rangle$, we have $a+b+c \in N_P$.
\end{enumerate}

\begin{df}
The category $\MPF$ of \emph{mock partial fields} is the full subcategory of $\Pastures$ whose objects are those pastures $P$ with $R_P \neq 0$.
\end{df}

\begin{lemma} \label{lem:MPF}
If $P$ is a pasture and there is a morphism $f : P \to P'$ to some partial field $P'$, then $P \in \MPF$.
\end{lemma}

\begin{proof}
Since $f : P \to P'$ induces a ring homomorphism $R_P \to R_{P'}$, we must have $R_P \neq 0$.
\end{proof}

\begin{df}
For $P \in \MPF$, we define the \emph{associated partial field} to be $\Pi(P) := (G,R_P)$ where $G$ is the image of the natural morphism $P^\times \to R_P$.
\end{df}

In the following, we consider $\PartFields$ as a subcategory of $\Pastures$. In particular, we identify the partial field $\Pi(P)=(G,R_P)$ with the pasture $P'=G\cup\{0\}$ with nullset $N_{P'}=\{a+b+c\in\Sym_3(P')\mid a+b+c=0\text{ in }R_P\}$. Note that the map $P^\times\to R_P$ defines a surjective pasture morphism $\pi_P:P\to\Pi(P)$.

\begin{lemma} \label{lem:MPF reflection}
 Let $P$ be a mock partial field with associated morphism $\pi_P:P\to\Pi(P)$. Then for every morphism $f : P \to Q$ into a partial field $Q$, there is a unique morphism $\bar{f} : \Pi(P) \to Q$ such that $f = \bar{f} \circ \pi_P$. In other words, the natural morphism $\pi_P : P \to \Pi(P)$ defines a reflection $\Pi:\MPF \to \PartFields$. 
\end{lemma}

\begin{proof}
 Provided that $\bar f$ exists, its uniqueness follows from the surjectivity of $\pi_P$. The existence can be verified as follows. The morphism $f:P\to Q$ induces a ring homomorphism $f_\Z:R_P\to R_Q$ such that the diagram
 \[
  \begin{tikzcd}[column sep=80pt, row sep=25pt]
   P \ar[r,"f"] \ar[d,"\iota_P"] & Q \ar[d,"\iota_Q"] \\
   R_P \ar[r,"f_\Z"] & R_Q
  \end{tikzcd}
 \]
 commutes, where the vertical arrows are the natural maps. Since $\iota_Q:Q\to R_Q$ is injective, as $Q$ is a partial field, and since the image of $\iota_P(P)\hookrightarrow R_P\to R_Q$ is contained in $\iota_Q(Q)$, we obtain a group homomorphism $\bar f:G\to Q^\times$ for $G=\iota_P(P)^\times$. By definition of the associated partial field $\Pi(P)=(G,R_P)$, its universal ring is $R_{\Pi(P)}=R_P$ and $\bar f$ extends to a homomorphism $\bar f_\Z=f_\Z:R_{\Pi(P)}\to R_Q$, which certifies that $\bar f:G\to Q^\times$ is indeed a morphism of partial fields. It follows from the definition of $\bar f$ that $f=\bar f\circ\pi_P$.
\end{proof}

We recall the definition of the \emph{Pendavingh--van Zwam lift $\Lift_{\pvz} P$} of a partial field $P=(G,R_P)$ from from \cite{Pendavingh-vanZwam10b}. The universal ring of $\Lift_{\pvz} P$ is
\[
 R_{\cL_{\pvz}P} \ = \ \Z[\,t_a^{\pm1}\mid a\in P^0\,]\,/\, I
\]
where $I$ is the ideal generated by the elements
 \begin{enumerate}[label={($\pvz$\arabic*)}]
  \item\label{PvZ1} $1+1$ if $-1 = 1$ in $P$;
  \item\label{PvZ2} $t_a \cdot t_{a^{-1}} - 1$ for all $a \in P^\fund$;
  \item\label{PvZ3} $t_a + t_b - 1$ whenever $a+b-1 \in N_P$;\addtocounter{enumi}{1}
  \item\label{PvZ5} $t_a t_b t_c - 1$ whenever $abc=1$ in $P$;
 \end{enumerate}
 and its unit group is the subgroup $G=\gen{-1,\,t_a\mid a\in P^0}$ of $R_{\cL_{\pvz}P}^\times$. It comes together with the canonical morphism
 \[  
  \begin{array}{cccc}
   \lambda_{\pvz,P}: & \Lift_{\pvz} P & \longrightarrow & P \\
                     & t_a & \longmapsto     & a.
  \end{array}
 \]
 
In general, the $\GRS$-lift and the $\PvZ$-lift of a partial field do not coincide; in particular, the $\GRS$-lift of a partial field is not a partial field in general; cf.\ \autoref{ex: GRS-lift of a partial field that is not a partial field}. However, we find the following relation between the two lifts.
 
\begin{prop}\label{prop: PvZ-lift as quotient of GRS-lift}
 If $P$ is a partial field, then $\Lift_{\grs} P$ is a mock partial field, $\Pi(\Lift_{\grs} P) \simeq \Lift_{\pvz} P$ and $\lambda_{\grs,P}=\lambda_{\pvz,P}\circ\pi_{\Lift_\grs P}$, i.e.\
 \[
  \begin{tikzcd}[column sep=80pt, row sep=25pt]
   \Lift_{\grs} P \ar[r,"\pi_{\Lift_\grs P}"] \ar[rd,"\lambda_{\grs,P}"'] & \Pi(\Lift_{\grs} P) \ar[r,"\sim"] &  \Lift_{\pvz} P \ar[dl,"\lambda_{\pvz,P}"] \\
                                                                         & P
  \end{tikzcd}
 \]
 commutes.
\end{prop}

\begin{proof}
The fact that $\Lift_{\grs} P$ is a mock partial field follows by \autoref{lem:MPF} from the existence of a morphism $\Lift_{\grs} P \to P$.
 The remaining statements follow easily from the definition of the associated partial field and a comparison of the defining relations of the $\GRS$-lift $\cL_{\grs}P$ with the corresponding relations of the $\PvZ$-lift $\cL_{\pvz}P$, with the caveat that the definition of $\cL_{\pvz}P$ does not list an analogue of \ref{GRS4} (note that we numbered the other axioms coherently, i.e.\ \ref{GRS1} corresponds to \ref{PvZ1}, and so forth, but there is no relation ($\pvz$4).)
 
The reason the proposition is valid despite the caveat is that the relations of type \ref{GRS4} are implied by the other relations when $P$ is a partial field. Indeed, given fundamental elements $a,b,c\in P^\fund$ with $abc=-1$ and $a+b^{-1}-1=0$, we conclude that $c=-a^{-1}b^{-1}$ and thus $c+a^{-1}-1=-a^{-1}(a+b^{-1}-1)=0$. Thus we have $t_a+t_b^{-1}-1=0$ and $t_c+t_a^{-1}-1=0$ in $R_{\cL_\pvz P}$, using $t_{a^{-1}}=t_a^{-1}$ by \ref{PvZ1}, which yields
 \[
  t_b \ = \ \tfrac{1}{1-t_a}, \qquad t_c \ =\ 1-t_{a}^{-1} \ = \ \tfrac{t_a-1}{t_a} \qquad \text{and} \qquad t_at_bt_c \ = \ t_a\cdot\tfrac{1}{1-t_a}\cdot \tfrac{t_a-1}{t_a} \ = \ -1,
 \]
 as desired. The equality $\lambda_{\grs,P}=\lambda_{\pvz,P}\circ\pi_{\Lift_\grs P}$ follows at once from the definition of these morphisms.
\end{proof}

\begin{lemma} \label{lemma: GRSPvZcomparison}
Let $P$ be a pasture and $\pi_{\Lift_\grs P}:\Lift_\grs P\to\Lift_\pvz P$ the quotient map.
\begin{enumerate}
\item\label{GRSPvZ1} The canonical map $\hat\pi_{\Lift_\grs P}:\Lift_\grs P\to\Lift_\grs\Lift_\pvz(P)$ with $\pi_{\Lift_\grs P}=\lambda_{\grs,\Lift_\pvz P}\circ\hat\pi_{\Lift_\grs P}$ is an isomorphism.
\item\label{GRSPvZ2} If $\Lift_{\grs} P$ is a partial field, then $\Lift_{\grs} P = \Lift_{\pvz} P$.
\end{enumerate}
\end{lemma}

\begin{proof}
 As an idempotent endofunctor on $\Pastures$, $\Lift_\grs$ is the identity on $\Lifts_{\grs}$, and therefore applying $\Lift_\grs$ to the commutative diagram
 \[
  \begin{tikzcd}[column sep=80pt, row sep=25pt]
   \Lift_\grs P \ar[r,"\hat\pi_{\Lift_\grs P}"] \ar[dr,"\pi_{\Lift_\grs P}"] \ar[d,"\lambda_{\grs,P}"'] & \Lift_\grs\Lift_\pvz P \ar[d,"\lambda_{\grs,\Lift_\pvz P}"] \\
   P & \Lift_\pvz P \ar[l,"\lambda_{\pvz,P}"]
  \end{tikzcd}
  \quad \text{yields} \quad
  \begin{tikzcd}[column sep=80pt, row sep=25pt]
   \Lift_\grs P \ar[r,"\hat\pi_{\Lift_\grs P}"] \ar[d,"\id"'] & \Lift_\grs\Lift_\pvz P \ar[d,"\id"] \\
   \Lift_\grs P  & \Lift_\grs\Lift_\pvz P \ar[l,"\Lift_\grs\lambda_{\pvz,P}"]
  \end{tikzcd}
 \]
 which shows that $\hat\pi_{\Lift_\grs P}$ is an isomorphism with inverse $\Lift_\grs\lambda_{\pvz,P}$, establishing \eqref{GRSPvZ1}.
 
 Claim \eqref{GRSPvZ2} follows at once from \autoref{prop: PvZ-lift as quotient of GRS-lift} and the fact that $\Pi(P)=P$ for a partial field $P$.
\end{proof}

As a formal consequence of these results, we obtain a proof of Conjecture 6.7 in \cite{Pendavingh-vanZwam10b}:

\begin{cor} \label{cor:idempotence conjecture}
 $\Lift_{\pvz}$ is an idempotent functor from the category of partial fields to itself, i.e., $\Lift_{\pvz}(\Lift_{\pvz}(P))= \Lift_{\pvz}(P)$ for every partial field $P$.
\end{cor}

\begin{proof}
 This follows at once from the canonical isomorphisms in \autoref{prop: PvZ-lift as quotient of GRS-lift} and \autoref{lemma: GRSPvZcomparison}: $\Lift_\pvz \Lift_\pvz P\simeq\Pi(\cL_\grs\Lift_\pvz P)\simeq\Pi(\cL_\grs P)\simeq\Lift_\pvz P$.
\end{proof}

Moreover, we find a new proof of Pendavingh-van Zwam's lift theorem for partial fields from \cite{Pendavingh-vanZwam10b}. Note that Pendavingh and van Zwam noted already in \cite[end of Section 4.1]{Pendavingh-vanZwam10b} that it should be possible to give an alternate proof of the lift theorem for partial fields by making use of Tutte's homotopy theorem. Since our construction of GRS-lifts relies heavily on the homotopy theorem, our new proof confirms their expectation.

\begin{thm}[$\PvZ$-lift theorem for matroid representations over partial fields]\label{thm: PvZ-lift theorem}
 For every partial field $P$ and every matroid $M$, every projective equivalence class of $P$-representations of $M$ lifts uniquely to $\Lift_{\pvz}(P)$.
\end{thm}

\begin{proof}
 Since the foundation $F_M$ represents rescaling classes, i.e.\ $\cX_M(P)=\Hom(F_M,P)$, the claim of the theorem amounts to the existence and uniqueness of a morphism $\bar\alpha:F_M\to P$ with $\alpha=\lambda_{\pvz,P}$ for any given $\alpha:F_M\to P$.
 
 Fix $\alpha:F_M\to P$. The existence of $\bar\alpha$ can be established as follows. By \autoref{thm: GRS-lifts}, there is a unique $\hat\alpha:F_M\to\Lift_\grs P$ such that $\alpha=\lambda_{\grs,P}\circ\hat\alpha$. If we define $\bar\alpha=\pi_{\Lift_\grs P}\circ\hat\alpha$, then the commutativity of the diagram
 \[
  \begin{tikzcd}[column sep=80pt, row sep=15pt]
       & \Lift_\pvz P \ar[d,"\lambda_{\pvz,P}" pos=0.4] & \Lift_\grs P \ar[l,"\pi_{\Lift_\grs P}" pos=0.5] \ar[dl,"\lambda_{\grs,P}"]\\
   F_M \ar[r,"\alpha"] \ar[ur,dashed,"\bar\alpha" pos=0.6] \ar[urr,bend left=30,dashed,"\hat\alpha"] & P  
  \end{tikzcd}
 \]
 yields $\alpha=\lambda_{\grs,P}\circ\hat\alpha=\lambda_{\pvz,P}\circ\pi_{\Lift_\grs P}\circ\hat\alpha=\lambda_{\pvz,P}\circ\bar\alpha$, as desired.

 In order to establish uniqueness, we consider a morphism $\beta:F_M\to \Lift_\pvz P$ with $\alpha=\lambda_{\pvz,P}\circ\beta$. Let $\hat\beta:F_M\to\Lift_\cG\Lift_\pvz P$ be the unique morphism with $\beta=\lambda_{\grs,\Lift_\pvz P}\circ\hat\beta$, as given by \autoref{thm: GRS-lifts}. Jointly with the isomorphism $\Lift_\grs\lambda_{\pvz,P}:\Lift_\grs\Lift_\pvz P\to\Lift_\grs P$, this yields the commutative diagram
 \[
  \begin{tikzcd}[column sep=80pt, row sep=15pt]
     &                                                & \Lift_\grs\Lift_\pvz P \ar[dl,"\lambda_{\grs,\Lift_\pvz P}" pos=0.3] \ar[d,"\Lift_\grs\lambda_{\pvz,P}"]  \\
     & \Lift_\pvz P \ar[d,"\lambda_{\pvz,P}" pos=0.4] & \Lift_\grs P \ar[l,"\pi_{\Lift_\grs P}" pos=0.5] \ar[dl,"\lambda_{\grs,P}"]\\
   F_M \ar[r,"\alpha"] \ar[ur,"\beta" pos=0.6] \ar[uurr,bend left=20,dashed,"\hat\beta" pos=0.6] & P  
  \end{tikzcd}
 \]
 and the equality $\beta=\lambda_{\grs,\Lift_\pvz P}\circ\hat\beta=\pi_{\Lift_\grs P}\circ\Lift_\grs\lambda_{\pvz,P}\circ\hat\beta=\pi_{\Lift_\grs P}\circ\hat\alpha=\bar\alpha$, using the uniqueness of the morphism $\Lift_\grs\lambda_{\pvz,P}\circ\hat\beta=\hat\alpha:F_M\to\Lift_\grs P$ with $\lambda_{\grs,P}\circ\Lift_\grs\lambda_{\pvz,P}\circ\hat\beta=\alpha=\lambda_{\grs,P}\circ\hat\alpha$. This completes the proof.
\end{proof}

\begin{ex}\label{ex: GRS-lift of a partial field that is not a partial field}
 The following is an example of a partial field $P$ whose $\GRS$-lift is not a partial field. Its universal ring is
 \[
  R_P \ = \ \Z[a^{\pm1},b^{\pm1},c^{\pm1},d^{\pm1},e^{\pm1},f^{\pm1}] \, / \, I := \gen{a+b-1,\, c-d-1,\, be-f-1,\, e-cf-1}
 \]
 and its unit group is the subgroup $G=\gen{-1,a,b,c,d,e,f}$ of $R^\times$. This is indeed a partial field, i.e.\ $R_P\neq\{0\}$, since the association
 \[
  \varphi(a)=3,\quad \varphi(b)=-2,\quad \varphi(c)=-2,\quad \varphi(d)=-3,\quad \varphi(e)=-1,\quad \varphi(f)=1
 \]
 extends to a ring homomorphism $\varphi:R_P\to\Z$. We find that $I$ contains
 \[
  e(a+b-1)+f(c-d-1)+(-1)(be-f-1)+(e-cf-1) \ = \ ae-df.
 \]
 Thus $ae=df$ in $G$, but neither $ae$ nor $df$ is a fundamental element. Therefore $ae$ and $df$ are distinct elements in $\cL_\grs P$. Since $R_{\Lift_\grs P}=R_{\Lift_\pvz P}=R_P$, this shows that the map $\Lift_\grs P\to R_{\Lift_\grs P}$ is not injective, and thus $\Lift_\grs P$ is not a partial field.
\end{ex}


\subsection{The lift theorem for binary matroids}
\label{subsection: The lift theorem for binary matroids}

In this section, we explain what we mean by a lower approximation to $\cL:\Pastures\to\Lifts$ with the example of binary lifts. This might be seen as the easiest non-trivial example of this nature, and serves as a prelude to the more involved constructions of ternary and $\WLUM$ lifts in \autoref{section: Ternary lifts}.

Let $\Lifts_{\binary}$ be the full subcategory of $\Pastures$ whose objects are all pastures that are isomorphic to either $\Funpm$ or $\F_2$, which we call \emph{binary lifts}. Given a pasture $P$, we define $\cL_{\binary} P$ to be $\F_2$ if $-1=1$, and $\Funpm$ if $-1\neq 1$ in $P$. In either case there is a unique map $\lambda_P=\lambda_{\binary,P}:\cL_{\binary}P\to P$. We call $\cL_{\binary}$ together with $\lambda_{\binary,P}$ the \emph{binary lift of $P$}.

\begin{prop}\label{prop: universal property of binary lifts}
 Let $\alpha:L\to P$ be a pasture morphism from a binary lift $L$ to a pasture $P$. Then there is a unique morphism $\hat \alpha:L\to\cL_{\binary}P$ such that $\alpha=\lambda_{\binary,P}\circ\hat \alpha$. In other words, $\Lifts_{\binary}$ is a coreflective subcategory of $\Pastures$ whose coreflection is defined by $\cL_{\binary}$.
\end{prop}

\begin{proof}
 Provided there exists a morphism $\hat\alpha:L\to\Lift_\binary P$, it is unique and satisfies $\alpha=\lambda_{\binary,P}\circ\hat\alpha$ since there is at most one morphism from either $\Funpm$ and $\F_2$ into any other pasture. If $L\simeq\Funpm$, then $L$ is initial in $\Pastures$ and the existence of $\hat\alpha$ is clear. If $L\simeq\F_2$, then $-1=1$ in $L$ and therefore $-1=\alpha(-1)=\alpha(1)=1$ in $P$. Thus $\Lift_\binary P=\F_2$, which establishes the existence of $\hat\alpha:L\to\Lift_\binary P$.
\end{proof}

\begin{thm}[Lift theorem for binary matroids]\label{thm: lift theorem for binary matroids}
 Let $M$ be a binary matroid and $P$ a pasture. Then every $P$-representation of $M$ lifts uniquely up to rescaling equivalence along $\lambda_{\binary,P}:\cL_{\binary}P\to P$.
\end{thm}

\begin{proof}
 Since the projective equivalence classes of $M$ over $P$ correspond bijectively to morphisms from the foundation $F_M$ of $M$ into $P$, the assertion of the theorem amounts to claim that for every morphism $\alpha:F_M\to P$, there is a unique morphism $\hat\alpha:F_M\to \Lift_\binary P$ with $\alpha=\lambda_{\binary,P}\circ\hat \alpha$. By \cite[Thm.\ 7.32]{Baker-Lorscheid18}, the foundation of a binary matroid $M$ is isomorphic to one of $\Funpm$ and $\F_2$. Thus the latter claim follows at once from \autoref{prop: universal property of binary lifts}.
\end{proof}


\section{Hexagons}
\label{section: Hexagons}

The construction of ternary and $\WLUM$-lifts is based on the notion of a hexagon in a pasture. In this section we discuss hexagons and their types, the relation between hexagons and fundamental pairs, and the behavior of hexagons in partial fields.

\subsection{Definitions}
\label{subsection: Definitions}

Let $P$ be a pasture. An \emph{ordered hexagon in $P$} is a $6$-tuple $(a,b,c,d,e,f)$ of elements $a,b,c,d,e,f\in P$ that satisfy the relations
 \begin{align*}
  a+b \ &= \ 1, & ac \ &= \ 1, & ade \ &= \ -1, \\
  c+e \ &= \ 1, & bd \ &= \ 1, & bcf \ &= \ -1, \\
  d+f \ &= \ 1, & ef \ &= \ 1, & 
 \end{align*}
 which can be illustrated as in \autoref{fig: hexagon}. 

\begin{figure}[tb]
 \centering\includegraphics{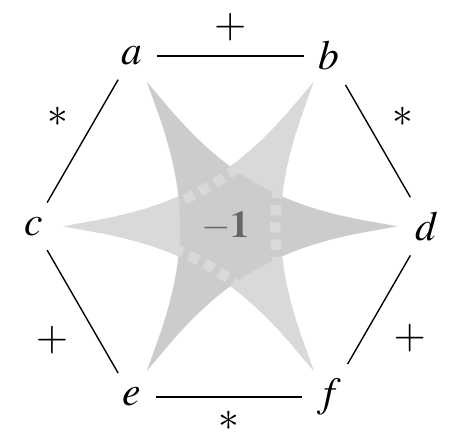}
 \caption{The relations of an ordered hexagon $(a,b,c,d,e,f)$ in $P$}
 \label{fig: hexagon}
\end{figure}
 
 Note that the relations $ac=bd=ef=1$ imply that $a,b,c,d,e,f\in P^\times$ and the relations $a+b=c+e=d+f=1$ imply that $a,b,c,d,e,f\in P^\fund$.

 The dihedral group $D_3=\gen{\rho,\sigma\mid \rho^3=\sigma^2=(\sigma\rho)^2=e}$ acts on ordered hexagons $(a,b,c,d,e,f)$ in $P$ by
 \[
  \rho.(a,b,c,d,e,f) \ = \ (e,c,f,a,d,b) \quad \text{and} \quad \sigma.(a,b,c,d,e,f) \ = \ (b,a,d,c,f,e),
 \]
 preserving the relations between $a,\dotsc,f$. Geometrically, this action can be seen as a symmetry of the hexagon in \autoref{fig: hexagon} that preserves the edge labels and the inner triangles: $\rho$ is a anti-clockwise rotation by 120 degrees and $\sigma$ is a reflection along the vertical axis.
 
 A \emph{hexagon in $P$} is an orbit of this action. We write $\Xi=\hex abcdef$ for the $D_3$-orbit of an ordered hexagon $(a,b,c,d,e,f)$ and $\Xi^\fund=\{a,b,c,d,e,f\}$. We denote the set of all hexagons in $P$ by $\Hex(P)$.

\begin{lemma}\label{lemma: image hexagon}
 Let $\varphi:P\to Q$ be a morphism of pastures and $\Xi=\hex abcdef$ a hexagon in $P$. Then 
 \[
  \hex{\varphi(a)}{\varphi(b)}{\varphi(c)}{\varphi(d)}{\varphi(e)}{\varphi(f)}
 \]
 is a hexagon in $Q$. This defines a map $\varphi^\hexagon:\Hex(P)\to\Hex(Q)$.
\end{lemma}

\begin{proof}
 This follows immediately from the fact that a morphism of pastures is multiplicative and preserves nullsets.
\end{proof}


\subsection{Fundamental pairs}
\label{subsection: Fundamental pairs}

\begin{df}
 Let $P$ be a pasture. A \emph{fundamental pair in $P$} is a pair $(a,b)\in (P^\times)^2$ such that $a+b=1$. We denote the set of fundamental pairs in $P$ by $P^\pair$.
\end{df}
 
\begin{lemma}\label{lemma: ordered hexagons from pairs of fundamental elements}
 The association
 \[
  \begin{array}{cccc}
   \Omega: & P^\pair & \longrightarrow & \{\text{ordered hexagons in $P$}\} \\
           & (a,b)   & \longmapsto     & (a,b,\frac 1a,\frac 1b,-\frac ba,-\frac ab) 
  \end{array}
 \]
 is a bijection.
\end{lemma}

\begin{proof}
 To start with, note that $(a,b,\frac 1a,\frac 1b,-\frac ba,-\frac ab)$ is indeed an ordered hexagon in $P$: by definition of a fundamental pair, we have $a+b-1=0$ and thus also $-1-\frac ba+\frac 1a=0$ and $-\frac ab-1+\frac 1b=0$, which verifies the additive relations; the multiplicative relations of an ordered hexagon are immediate. Thus $\Omega$ is well-defined as a map.
 
 On the other hand, note that the defining relations for an ordered hexagon $(a,b,c,d,e,f)$ imply that 
 \[\textstyle
  c \ = \ \frac 1a, \qquad d \ = \ \frac 1b, \qquad e \ = \ -\frac 1{ad} \ = \ -\frac ba, \qquad f \ = \ -\frac 1{bc} \ = \ -\frac ab,
 \]
 which shows that $\Omega$ is surjective. The injectivity of $\Omega$ is evident.
\end{proof}

\begin{df}
 Let $P$ be a pasture and $(a,b)\in P^\pair$ a fundamental pair. We define the \emph{hexagon associated with $(a,b)$} as
 \[
  \Xi(a,b) \ = \Big[\Omega(a,b)\Big] \ = \ \hex{a}{b}{\frac 1a}{\frac 1b}{-\frac ba}{-\frac ab}.
 \]
 Let $\Xi\in\Hex(P)$. A \emph{fundamental pair in $\Xi$} is a fundamental pair $(a,b)\in P^\pair$ such that $\Xi=\Xi(a,b)$. We denote the set of fundamental pairs in $\Xi$ by $\Xi^\pair$.
\end{df}

\begin{lemma}\label{lemma: fundamental pairs in a hexagon}
 Let $P$ be a pasture and $\Xi\in\Hex(P)$. If $(a,b)\in\Xi^\pair$, then
 \[\textstyle
  \Xi^\pair \ = \ \Big\{(a,b),\ \ (b,a),\ \ (\frac 1a,-\frac ba),\ \ (-\frac ba,\frac 1a),\ \ (\frac 1b,-\frac ab),\ \ (-\frac ab,\frac 1b) \Big\}.
 \]
\end{lemma}

\begin{proof}
 This is immediate from the defining relations of a hexagon.
\end{proof}

By the definition of a pasture $P$, its nullset $N_P$ is invariant under multiplication by elements in $P^\times$, and so is the subset $N_P^\times$ of all $a+b+c\in N_P$ with $abc\in P^\times$. The dihedral group $D_3$ acts on $P^\pair$ via
\[\textstyle
 \rho.(a,b) \ = \ (\frac 1b,-\frac ab) \quad \text{and} \quad \sigma.(a,b) \ = \ (b,a).
\]

\begin{prop}\label{prop: nullset and hexagons}
 Let $P$ be a pasture. The associations
 \[
  \Phi:\Big[(a,b)\Big]\longmapsto [a+b-1], 
 \]
 \[\textstyle
  \Xi:\Big[(a,b)\Big]\longmapsto \Xi(a,b), \qquad \Psi:[a+b+c]\longmapsto \hex{-\frac ac}{-\frac bc}{-\frac ca}{-\frac cb}{-\frac ba}{-\frac ab}
 \]
 define a commutative diagram 
 \[
  \begin{tikzcd}
   P^\pair/D_3 \ar[rr,"\Phi"] \ar[rd,"\Xi"'] &         & N_P^\times/P^\times \ar[dl,"\Psi"] \\
                                         & \Hex(P)
  \end{tikzcd}
 \]
 of bijections.
\end{prop}

\begin{proof}
 To see that $\Phi$ is well-defined, note that
 \begin{align*}
  \Phi\Big(\Big[\sigma.(a,b)\Big]\Big) \ &= \ b+a-1 \ = \ a+b-1, \\
  \Phi\Big(\Big[\rho.(a,b)\Big]\Big)   \ &= \ \tfrac 1b-\tfrac ab-1 \ = \ -\tfrac 1b(a+b-1).
 \end{align*}
 Since the bijection $\Omega: P^\pair \to \{\text{ordered hexagons in $P$}\}$ from \autoref{lemma: ordered hexagons from pairs of fundamental elements} is $D_3$\hyph equivariant, it gives rise to the bijection $\Xi:P^\pair/D_3\to\Hex(P)$. In particular, it follows that $\Xi$ is well-defined.
 
 The map $\Psi$ is well-defined because all entries of $\Psi\Big([a+b+c]\Big)$ are invariant under multiplying $a$, $b$ and $c$ by a common scalar, and because $(-a/c,-b/c)$ is a fundamental pair since $-\frac 1c(a+b+c)=-\frac ac-\frac bc-1$.

 The commutativity of the diagram follows from
 \[\textstyle
  \Psi\circ\Phi\Big(\Big[(a,b)]\Big) \ = \ \Psi\Big([a+b-1]\Big) \ = \ \hex{a}{b}{\frac 1a}{\frac 1b}{-\frac ba}{-\frac ab} \ = \ \Xi(a,b).
 \]
 It is immediately verified that $\Xi^{-1}\circ\Psi$ is an inverse bijection to $\Phi$ and that $\Phi\circ\Xi^{-1}$ is an inverse bijection to $\Psi$. This completes the proof.
\end{proof}


\subsection{The four types of hexagons}
\label{subsection: Four types of hexagons}

We investigate the hexagons and fundamental pairs of the four pastures $\U$, $\D$, $\H$ and $\F_3$, which form the prototypes for hexagons and fundamental pairs in all other pastures.

Note that for each of these four pastures $P$, the nullset $N_P$ is defined by a single $3$-term relation $a+b+c\in N_P^\times$ and thus $P$ contains a unique $D_3$-orbit of fundamental elements and a unique hexagon $\Xi_P$ by \autoref{prop: nullset and hexagons}.

\subsubsection*{The near-regular partial field}
The unique hexagon of the near-regular partial field $\U=\pastgen\Funpm{x,y}{x+y-1}$ is 
\[\textstyle
 \Xi_\U \ = \ \hex xy{\frac 1x}{\frac 1y}{-\frac yx}{-\frac xy}
\]
and has $6$ distinct elements; cf.\ \autoref{fig: hexagons of U, D, H and F3}. Thus the set of fundamental pairs 
\[\textstyle
 \U^\pair \ = \ \Big\{ (x,y),\ \ (y,x),\ \ (\frac 1x,-\frac yx),\ \ (-\frac yx,\frac 1x),\ \ (\frac 1y,-\frac xy),\ \ (-\frac xy,\frac 1y) \Big\}
\]
has $6$ distinct elements. 

\subsubsection*{The dyadic partial field}
The unique hexagon of the dyadic partial field $\D=\pastgen\Funpm{z}{z+z-1}$ is 
\[\textstyle
 \Xi_\D \ = \ \hex zz{\frac 1z}{\frac 1z}{-1}{-1}
\]
and has $3$ distinct elements; cf.\ \autoref{fig: hexagons of U, D, H and F3}. Thus the set of fundamental pairs
\[\textstyle
 \D^\pair \ = \ \Big\{ (z,z),\, (\frac 1z,-1),\, (-1,\frac 1z) \Big\}
\]
has $3$ distinct elements.

\subsubsection*{The hexagonal partial field}
The unique hexagon of the hexagonal partial field $\H=\pastgen\Funpm{z}{z+z^{-1}-1}$ is
\[\textstyle
 \Xi_\H \ = \ \hex z{\frac 1z}{\frac 1z}zz{\frac 1z}
\]
and has $2$ distinct elements; cf.\ \autoref{fig: hexagons of U, D, H and F3}. Thus the set of fundamental pairs 
\[\textstyle
 \H^\pair \ = \ \Big\{ (z,\frac 1z),\, (\frac 1z,z) \Big\}
\]
has $2$ distinct elements. 

\subsubsection*{The finite field with \texorpdfstring{$3$}{3} elements}
The unique hexagon of the finite field $\F_3=\past\Funpm{\gen{-1-1-1}}$ is
\[
 \Xi_{\F_3} \ = \ \hex{-1}{-1}{-1}{-1}{-1}{-1}
\]
and has one element $-1$; cf.\ \autoref{fig: hexagons of U, D, H and F3}. Thus the set of fundamental pairs 
\[
 \F_3^\pair \ = \ \Big\{ (-1,-1) \Big\}
\]
has one element. 

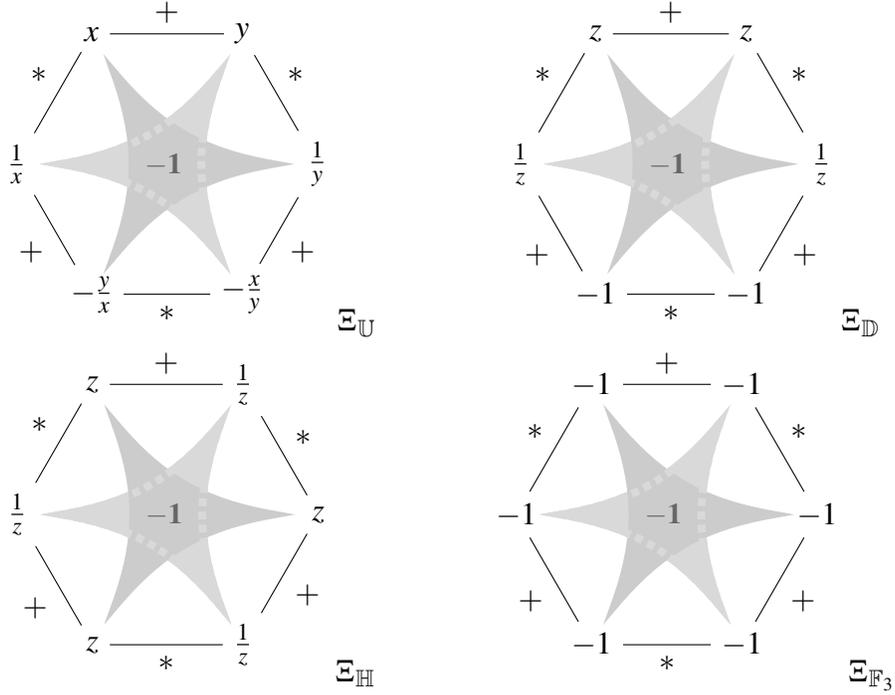
\begin{figure}[t]
\[
 \begin{array}{cc}
 \beginpgfgraphicnamed{tikz/fig2}
  \begin{tikzpicture}[x=1cm,y=1cm]
   \draw[line width=3pt,color=gray!30,fill=gray!30,bend angle=20] (60:1.4) to[bend left] (180:1.4) to[bend left] (300:1.4) to[bend left] cycle;
   \draw[line width=3pt,color=gray!40,fill=gray!40,bend angle=20] (0:1.4) to[bend left] (120:1.4) to[bend left] (240:1.4) to[bend left] cycle;
   \draw[line width=3pt,color=gray!30,dotted,bend angle=20] (60:1.4) to[bend left] (180:1.4) to[bend left] (300:1.4) to[bend left] cycle;
   \node (-1) at (0,0) {\small \textcolor{black!60}{{$\mathbf{-1\ }$}}};
   \node (a) at (120:2) {$x$};
   \node (b) at ( 60:2) {$y$};
   \node (c) at (180:2) {$\frac 1x$};
   \node (d) at (  0:2) {$\frac 1y$};
   \node (e) at (240:2) {$-\frac yx$};
   \node (f) at (300:2) {$-\frac xy$};
   \path (a) edge node[auto] {$+$} (b);
   \path (b) edge node[auto] {$*$} (d);
   \path (d) edge node[auto] {$+$} (f);
   \path (f) edge node[auto] {$*$} (e);
   \path (e) edge node[auto] {$+$} (c);
   \path (c) edge node[auto] {$*$} (a);
  \end{tikzpicture}
 \endpgfgraphicnamed 
 \Xi_\U \hspace{0.5cm} & \hspace{0.5cm} 
 \beginpgfgraphicnamed{tikz/fig3}
  \begin{tikzpicture}[x=1cm,y=1cm]
   \draw[line width=3pt,color=gray!30,fill=gray!30,bend angle=20] (60:1.4) to[bend left] (180:1.4) to[bend left] (300:1.4) to[bend left] cycle;
   \draw[line width=3pt,color=gray!40,fill=gray!40,bend angle=20] (0:1.4) to[bend left] (120:1.4) to[bend left] (240:1.4) to[bend left] cycle;
   \draw[line width=3pt,color=gray!30,dotted,bend angle=20] (60:1.4) to[bend left] (180:1.4) to[bend left] (300:1.4) to[bend left] cycle;
   \node (-1) at (0,0) {\small \textcolor{black!60}{{$\mathbf{-1\ }$}}};
   \node (a) at (120:2) {$z$};
   \node (b) at ( 60:2) {$z$};
   \node (c) at (180:2) {$\frac 1z$};
   \node (d) at (  0:2) {$\frac 1z$};
   \node (e) at (240:2) {$-1$};
   \node (f) at (300:2) {$-1$};
   \path (a) edge node[auto] {$+$} (b);
   \path (b) edge node[auto] {$*$} (d);
   \path (d) edge node[auto] {$+$} (f);
   \path (f) edge node[auto] {$*$} (e);
   \path (e) edge node[auto] {$+$} (c);
   \path (c) edge node[auto] {$*$} (a);
  \end{tikzpicture}
 \endpgfgraphicnamed 
 \Xi_\D \\
 \beginpgfgraphicnamed{tikz/fig4}
  \begin{tikzpicture}[x=1cm,y=1cm]
   \draw[line width=3pt,color=gray!30,fill=gray!30,bend angle=20] (60:1.4) to[bend left] (180:1.4) to[bend left] (300:1.4) to[bend left] cycle;
   \draw[line width=3pt,color=gray!40,fill=gray!40,bend angle=20] (0:1.4) to[bend left] (120:1.4) to[bend left] (240:1.4) to[bend left] cycle;
   \draw[line width=3pt,color=gray!30,dotted,bend angle=20] (60:1.4) to[bend left] (180:1.4) to[bend left] (300:1.4) to[bend left] cycle;
   \node (-1) at (0,0) {\small \textcolor{black!60}{{$\mathbf{-1\ }$}}};
   \node (a) at (120:2) {$z$};
   \node (b) at ( 60:2) {$\frac 1z$};
   \node (c) at (180:2) {$\frac 1z$};
   \node (d) at (  0:2) {$z$};
   \node (e) at (240:2) {$z$};
   \node (f) at (300:2) {$\frac 1z$};
   \path (a) edge node[auto] {$+$} (b);
   \path (b) edge node[auto] {$*$} (d);
   \path (d) edge node[auto] {$+$} (f);
   \path (f) edge node[auto] {$*$} (e);
   \path (e) edge node[auto] {$+$} (c);
   \path (c) edge node[auto] {$*$} (a);
  \end{tikzpicture}
 \endpgfgraphicnamed 
 \Xi_\H \hspace{0.5cm} & \hspace{0.5cm} 
 \beginpgfgraphicnamed{tikz/fig5}
  \begin{tikzpicture}[x=1cm,y=1cm]
   \draw[line width=3pt,color=gray!30,fill=gray!30,bend angle=20] (60:1.4) to[bend left] (180:1.4) to[bend left] (300:1.4) to[bend left] cycle;
   \draw[line width=3pt,color=gray!40,fill=gray!40,bend angle=20] (0:1.4) to[bend left] (120:1.4) to[bend left] (240:1.4) to[bend left] cycle;
   \draw[line width=3pt,color=gray!30,dotted,bend angle=20] (60:1.4) to[bend left] (180:1.4) to[bend left] (300:1.4) to[bend left] cycle;
   \node (-1) at (0,0) {\small \textcolor{black!60}{{$\mathbf{-1\ }$}}};
   \node (a) at (120:2) {$-1$};
   \node (b) at ( 60:2) {$-1$};
   \node (c) at (180:2) {$-1$};
   \node (d) at (  0:2) {$-1$};
   \node (e) at (240:2) {$-1$};
   \node (f) at (300:2) {$-1$};
   \path (a) edge node[auto] {$+$} (b);
   \path (b) edge node[auto] {$*$} (d);
   \path (d) edge node[auto] {$+$} (f);
   \path (f) edge node[auto] {$*$} (e);
   \path (e) edge node[auto] {$+$} (c);
   \path (c) edge node[auto] {$*$} (a);
  \end{tikzpicture}
 \endpgfgraphicnamed 
 \Xi_{\F_3}
 \end{array}
\]
\caption{The hexagons of $\U$, $\D$, $\H$ and $\F_3$}
\label{fig: hexagons of U, D, H and F3}
\end{figure}

\begin{df}\label{def: orbit length}
 Let $P$ be a pasture and $\Xi\in\Hex(P)$. The \emph{orbit length of $\Xi$} is the number $\mu_\Xi=\#\Xi^\pair$ of fundamental pairs in $\Xi$. The hexagon $\Xi$ is \\
 \begin{tabular}{ll}
  \qquad{\raisebox{1.1pt}{\tiny$\bullet$}}\, \emph{of near-regular type} & if $\mu_\Xi=6$; \\
  \qquad{\raisebox{1.1pt}{\tiny$\bullet$}}\, \emph{of dyadic type}       & if $\mu_\Xi=3$; \\
  \qquad{\raisebox{1.1pt}{\tiny$\bullet$}}\, \emph{of hexagonal type}    & if $\mu_\Xi=2$; \\
  \qquad{\raisebox{1.1pt}{\tiny$\bullet$}}\, \emph{of ternary type}      & if $\mu_\Xi=1$.
 \end{tabular}
\end{df}
 
\begin{prop}\label{prop: types of hexagons}
 Let $P$ be a pasture and $\Xi\in\Hex(P)$. Then $\mu_\Xi\in\{1,2,3,6\}$ and 
 \begin{itemize}
  \item $\mu_\Xi=1$ if and only if $\Xi=\hex{-1}{-1}{-1}{-1}{-1}{-1}$;
  \item $\mu_\Xi=2$ if and only if $\Xi=\hex a{a^{-1}}{a^{-1}}aa{a^{-1}}$ for some $a\neq -1$;
  \item $\mu_\Xi=3$ if and only if $\Xi=\hex aa{a^{-1}}{a^{-1}}{-1}{-1}$ for some $a\neq -1$;
  \item $\mu_\Xi=6$ if and only if $\{a,\frac 1b,-\frac ba\}$ does not contain $-1$ and has at least $2$ elements for some $(a,b)\in\Xi^\pair$.
 \end{itemize}
\end{prop}
 
\begin{proof}
 By \autoref{prop: nullset and hexagons}, $\Xi^\pair$ is an orbit of the $D_3$-action on $P^\pair$. Thus the cardinality of $\Xi^\pair$ divides $\# D_3=6$, i.e.\ $\mu_\Xi=\#\Xi^\pair\in\{1,2,3,6\}$, as claimed. Recall from \autoref{lemma: fundamental pairs in a hexagon} that 
 \[\textstyle
  \Xi^\pair \ = \ \Big\{(a,b),\ \ (b,a),\ \ (\frac 1a,-\frac ba),\ \ (-\frac ba,\frac 1a),\ \ (\frac 1b,-\frac ab),\ \ (-\frac ab,\frac 1b) \Big\}
 \]
 for any $(a,b)\in\Xi^\pair$.
 
 If $\Xi$ is of ternary type, i.e.\ $\mu_\Xi=1$, then $a=b=-\frac ba=-1$ and thus $\Xi=\hex{-1}{-1}{-1}{-1}{-1}{-1}$. Conversely, if $\Xi=\hex{-1}{-1}{-1}{-1}{-1}{-1}$, then $\Xi^\pair=\Big\{(-1,-1)\Big\}$ has only one element and thus $\Xi$ is of ternary type.
 
 If $\Xi$ is of hexagonal type, i.e.\ $\mu_\Xi=2$, then $\Stab_{D_3}(a,b)=\gen\rho$ is the unique subgroup of index $2$ for every $(a,b)\in\Xi^\pair$. Thus $a=\frac 1b=-\frac ab$ and $\frac 1a=b=-\frac ba$, and therefore $\Xi=\hex a{a^{-1}}{a^{-1}}aa{a^{-1}}$. Since $\#\Xi^\pair=\mu_\Xi=2$, we conclude that $a\neq-1$. Conversely, if $\Xi=\hex a{a^{-1}}{a^{-1}}aa{a^{-1}}$ for some $a\neq-1$, then $\Xi^\pair=\Big\{(a,\frac 1a),(\frac 1a,a)\Big\}$ has two elements and thus $\Xi$ is of hexagonal type.
 
 If $\Xi$ is of dyadic type, i.e.\ $\mu_\Xi=3$, then $\Stab_{D_3}(a,b)=\gen\sigma$ for some $(a,b)\in\Xi^\pair$ since all index $3$-subgroups of $D_3$ are conjugate. Thus $a=b$ as well as $\frac 1a=\frac 1b$ and $-\frac ba=-\frac ab=-1$, which shows that $\Xi=\hex aa{a^{-1}}{a^{-1}}{-1}{-1}$. Since $\#\Xi^\pair=\mu_\Xi=3$, we have $a\neq-1$. Conversely, if $\Xi=\hex a{a^{-1}}{a^{-1}}aa{a^{-1}}$ for some $a\neq-1$, then $\Xi^\pair$ has $3$ distinct elements $(a,a)$, $(\frac 1a,-1)$ and $(-1,\frac 1a)$ and thus $\Xi$ is of dyadic type.
 
 To establish the characterization of hexagons of near-regular type, consider $(a,b)\in\Xi^\pair$. If $-1\in\{a,\frac 1b,-\frac ab\}$, then either
 \begin{itemize}
  \item $a=-1$,\quad $\frac 1a=-1$,\quad $-\frac ba=b$,\quad $-\frac ab=\frac 1b$\quad and\quad $\Xi=\hex{-1}b{-1}{b^{-1}}b{b^{-1}}$;\quad or
  \item $\frac 1b=-1$,\quad $b=-1$,\quad $-\frac ba=\frac 1a$,\quad $-\frac ab=a$\quad and\quad $\Xi=\hex a{-1}{a^{-1}}{-1}{a^{-1}}a$;\quad or
  \item $-\frac ab=-1$,\quad $b=a$,\quad $\frac 1b=\frac 1a$,\quad $-\frac ba=-1$\quad and\quad $\Xi=\hex aa{a^{-1}}{a^{-1}}{-1}{-1}$.
 \end{itemize}
 In each case $\Xi$ is of dyadic or ternary type. If $a=\frac 1b=-\frac ba$, then $\frac 1a=b=-\frac ab$ and $\Xi=\hex a{a^{-1}}{a^{-1}}aa{a^{-1}}$ is of hexagonal or ternary type. Thus if $\Xi$ is of near-regular type, then $\{a,\frac 1b,-\frac ba\}$ does not contain $-1$ and has at least $2$ elements.
 
 Conversely, if $\{a,\frac 1b,-\frac ba\}$ does not contain $-1$ and has at least $2$ elements, then $\Xi$ is not of dyadic, hexagonal or ternary type, and therefore it is of near-regular type. This completes the proof.
\end{proof}


\subsection{Hexagons in partial fields}
\label{subsection: Hexagons in partial fields}

As opposed to the general case (cf.\ \autoref{subsection: Examples of hexagons}), hexagons behave rather nicely in partial fields.

\begin{prop}\label{prop: hexagons in partial fields}
 Let $P$ be a partial field with universal ring $R_P=\Z[P^\times]/\gen{N_P}$ and let $\Xi,\tilde\Xi\in\Hex(P)$. Then for any $a\in\Xi^\fund$, the element $1-a\in R_P$ is in $P^\times$ and 
 \[
  \Xi \ = \ \hex{a}{1-a}{\frac1a}{\frac1{1-a}}{\frac{a-1}a}{\frac a{a-1}}.
 \]
 Moreover $0,1\notin\Xi^\fund$, and if $\Xi^\fund\cap\tilde\Xi^\fund\neq\emptyset$, then $\Xi=\tilde\Xi$. The natural inclusion
 \[
  \coprod_{\Xi\in\Hex(P)}\Xi^\fund \ \longrightarrow \ P-\{0,1\} 
  \]
 is a bijection if and only if $P$ is a field.
\end{prop}

\begin{proof}
 Let $a\in\Xi$. Then $\Xi$ is the class of an ordered hexagon of the form $(a,b,c,d,e,f)\in P^6$ since $D_3$ acts transitively on the coordinates of an ordered hexagon. Using the defining relations of an ordered hexagon, we find that $a+b=1$, which shows that the element $1-a=b$ of $R_P$ is contained in $P$. The other defining relations of an ordered hexagon show that 
 \begin{align*}
  ac \ = \ 1 \quad &\implies \quad c \ = \ \tfrac 1a; & bcf \ = \ -1 \quad &\implies \quad f \ = \ -\tfrac ab; \\
  bd \ = \ 1 \quad &\implies \quad d \ = \ \tfrac 1b; & ade \ = \ -1 \quad &\implies \quad e \ = \ -\tfrac ba; 
 \end{align*}
 and thus 
 \[
  \Xi \ = \ \hex{a}{1-a}{\frac1a}{\frac1{1-a}}{\frac{a-1}a}{\frac a{a-1}}
 \]
 as claimed. In particular, $1-a=b\in P^\times$ and $\Xi$ is uniquely determined by $a$. Thus if $a\in \Xi^\fund\cap\tilde\Xi^\fund$, then $\Xi=\tilde\Xi$. Since $0c=0\neq 1$, we conclude that $0\notin\Xi^\fund$, and since $1+b=1$ implies $b=0$, we conclude that $1\notin\Xi^\fund$. Thus we obtain a natural injection 
 \[
  \Psi: \ \coprod_{\Xi\in\Hex(P)}\Xi^\fund \ \longrightarrow \ P-\{0,1\}.
 \]
 If $P$ is a field, then every element $a\in P-\{0,1\}$ is a fundamental element and $a\in\Xi(a,1-a)$. Thus the injection $\Psi$ is surjective.
 
 If $\Psi$ is bijective, then consider $a,b\in P$. Let $R_P=\Z[P^\times]/\gen{N_P}$ be the universal ring of $P$ and $c=a+b$, considered as an element of $R_P$. We need to show that $c\in P$. This is clear if $a=0$ (then $c=b$), if $b=0$ (then $c=a$) or if $b=-a$ (then $c=0$). Thus we can assume that $a$, $b$ and $c$ are nonzero. Thus $a\in P^\times$, and 
 \[\textstyle
  0 \ = \ -\frac 1a (a+b-c) \ = \ -1-\frac ba+\frac ca
 \]
 as an equality in $R_P$, where $-\frac ba\in P^\times$. Since both $b$ and $c$ are nonzero, we conclude that $-\frac ba\notin\{0,1\}$. Since $\Psi$ is surjective, we conclude that there is a hexagon $\Xi$ in $P$ such that $-\frac ba\in\Xi^\fund$. Thus $1-(-\frac ba)=1+\frac ba$ is an element of $\Xi^\fund\subset P^\times$, and consequently
 \[\textstyle
  c \ = \ a+b \ = \ a(1+\frac ba)
 \]
 is an element of $P$, as desired.
\end{proof}

While an arbitrary pasture might have fewer fundamental elements than fundamental pairs (as is the case for $\S$ and $\W$; cf.\ \autoref{subsection: Examples of hexagons}), we have equality for partial fields:

\begin{lemma}\label{lemma: orbit length of a hexagon equals number of fundamental elements}
 Let $P$ be a partial field and $\Xi$ a hexagon in $P$. Then $\mu_\Xi=\#\Xi^\fund$.
\end{lemma}

\begin{proof}
 Since the equality of two fundamental pairs $(a,b)$ and $(c,d)$ in $\Xi$ implies the equality of the fundamental elements $a$ and $b$, we have $\#\Xi^\fund\leq\#\Xi^\pair=\mu_\Xi$. So if $\Xi=\hex{-1}{-1}{-1}{-1}{-1}{-1}$ is of ternary type, i.e.\ $\mu_\Xi=1$, there is nothing to prove.
 
 In the following, we freely use that $P$ is embedded as a submonoid into its universal ring $R_P=\Z[P^\times]/\gen{N_P}$, which allows us to make sense of sums and differences of fundamental elements as elements in $R_P$.
 
 Let us consider the case $\Xi=\hex{a}{a^{-1}}{a^{-1}}{a}{a}{a^{-1}}$ for some $a\in P^\times$. Then $\mu_\Xi\leq2$. If $\#\Xi^\fund=2$, then the claim follows from $2=\#\Xi^\fund\leq\mu_\Xi\leq2$. If $\#\Xi^\fund<2$, then $a^{-1}=a$ and $a^2=1$. The hexagon relations include $a^3=-1$, and thus $a=a\cdot a^2=-1$. This shows that $\Xi=\hex{-1}{-1}{-1}{-1}{-1}{-1}$ is of ternary type, for which we have proven our claim already.
 
 We continue with the case $\Xi=\hex{\frac12}{\frac12}{2}{2}{-1}{-1}$ for some $2\in P^\times$. Then $\mu_\Xi\leq3$. If $\#\Xi^\fund=3$, then the claim follows from $3=\#\Xi^\fund\leq\mu_\Xi\leq3$. If $\#\Xi^\fund<3$, then at least two of the elements $2$, $\frac12$ and $-1$ are equal to each other. If $2=-1$ or $\frac12=-1$, then $2=-1=\frac12$ and $\Xi=\hex{-1}{-1}{-1}{-1}{-1}{-1}$ is of ternary type, for which our claim is established. If $2=\frac12$, then $3=2^2-1=1-1=0$. Thus $\frac12=2=1+1=-1$ and $\Xi=\hex{-1}{-1}{-1}{-1}{-1}{-1}$ is also in this case of ternary type.
 
 We continue with the general case of a hexagon of the shape 
 \[
  \Xi \ = \ \hex{a}{1-a}{\frac1a}{\frac1{1-a}}{\frac{a-1}a}{\frac a{a-1}}.
 \]
 If $\#\Xi^\fund=6$, then the claim follows from $6=\#\Xi^\fund\leq\mu_\Xi\leq6$. If $\#\Xi^\fund<6$, then at least $2$ elements among $a$, $1-a$, $\frac1a$, $\frac1{1-a}$, $\frac{a-1}a$ and $\frac a{a-1}$ are equal to each other. Using the $D_3$-symmetry, we can assume that $a$ is equal to one of the other elements. We inspect each possibility in the following.

 If $a=\frac1{1-a}$ or $a=\frac{a-1}a$, then $a^2-a+1=0$. Thus $a^3=a\cdot(a-1)=a^2-a=-1$ and $1-a=-a^2=a^{-1}$. This shows that $\Xi=\hex{a}{a^{-1}}{a^{-1}}{a}{a}{a^{-1}}$, for which we have already verified our claim.
 
 If $a=1-a$, then $a+a=1$ and thus $a^{-1}=1+1=2$. Thus $\tfrac{1}{1-a}=2$ and $\frac{a-1}{a}=-1=\frac{a}{a-1}=-1$. It follows that $\Xi=\hex{\frac12}{\frac12}{2}{2}{-1}{-1}$, for which we have already verified our claim.
 
 If $a=\frac{a}{a-1}$, then $a^2-a=a$ and therefore $a=1+1=2$. Thus $1-a=\frac1{1-a}=-1$ and $\Xi=\hex{2}{-1}{\frac12}{-1}{\frac12}{2}$, for which we have already verified our claim.
 
 If $a=a^{-1}$, then $a^2=1$ and $\frac{a}{a-1}=\frac{a^2}{a^2-a}=\frac{1}{1-a}$. The hexagon relations imply $\frac{1}{1-a}+\frac{1}{1-a}=\frac{1}{1-a}+\frac{a}{a-1}=1$. We conclude that $\frac{a}{a-1}=\frac1{1-a}=\frac12$ and $\frac{a-1}{a}=1-a=2$. The relation $a=a\cdot2\cdot \frac12=-1$ shows that $\Xi=\hex{-1}{2}{-1}{\frac12}{2}{\frac12}$, for which we have already verified our claim. This completes the proof.
\end{proof}

Hexagons of ternary, hexagonal and dyadic type are uniquely determined in a field since the respective defining relations $z+1=0$, $\ z^2-z+1=0$ and $z-1-1=0$ have unique solutions (up to multiplicative inverses in the second case). This extends to partial fields for ternary and dyadic types, but fails the hexagonal type, as illustrated in \autoref{ex: partial fields with many hexagons of hexagonal type}. The following assumption on partial fields guarantees uniqueness for hexagons of hexagonal type as well, cf.\ \autoref{prop: hexagons in integral partial fields}.

\begin{df}
 A partial field $P$ is \emph{integral} if it injects into a field.
\end{df}

\begin{rem}\label{rem: non-integral partial fields}
 Not every partial field is integral. For example, the pasture $P=\F_3\otimes\H$ is a partial field with universal ring 
 \[
  R_P \ = \ \F_3[z^{\pm1}]/(z^3+1,z+z^{-1}-1) \ \simeq \ \F_3[\epsilon]/(\epsilon^2)
 \]
 where $z=\epsilon-1$. However, since $1+1+1=0$ in $P$, there are only morphisms $\varphi:P\to k$ from $P$ into fields $k$ of characteristic $3$, and thus $\varphi(z)=-1$, which shows that $\varphi$ cannot be injective.
 
 Another example is the pasture $P=\pastgen\Funpm{x}{x^2-1}$ whose universal ring is $R_P=\Z[x]/\gen{x^2-1}\simeq\Z\oplus\Z$. Thus every pasture morphism $f:P\to K$ to a field $K$ must map $x$ to $1$ or $-1$, which shows that $f$ cannot be injective.
 
 If the universal ring $R_P$ is an integral domain, then $P$ is integral since it embeds as a subpasture\footnote{A \emph{subpasture} of a pasture $P$ is a submonoid $Q$ of $P$ together with a subset $N_Q\subset\Sym_3(Q)$ such that $Q=Q^\times\cup\{0\}$ and such that $a+b+c\in N_{Q}$ for all $a+b+c\in N_P$ with $a,b,c\in Q$. } into the field of fractions of $R_P$. The converse is not true, i.e.\ there are subpastures $P$ of fields whose universal ring is not an integral domain; moreover, there are integral partial fields that do not embed as a subpasture into any field. However, examples of these types are a somewhat involved and we omit a description.
\end{rem}

\begin{prop}\label{prop: hexagons in integral partial fields}
 Let $P$ be a partial field. Then the following holds:
 \begin{enumerate}
  \item\label{type1} $P$ has a (unique) hexagon of ternary type if and only if $1+1+1=0$.
  \item\label{type2} $P$ has at most one hexagon of dyadic type, and none if $1+1+1=0$.
  \item\label{type3} If $P$ is integral, then $P$ has at most one hexagon of hexagonal type, and none if $1+1+1=0$.
 \end{enumerate}
\end{prop}

\begin{proof}
 By \autoref{prop: types of hexagons}, a hexagon $\Xi$ in $P$ is of ternary type if and only if $\Xi=\hex{-1}{-1}{-1}{-1}{-1}{-1}=\Xi(-1,-1)$, which is, in particular, uniquely determined. Thus $P$ has a hexagon of ternary type if and only if $(-1,-1)$ is a fundamental pair of $P$, which means that $-1-1-1=0$, or equivalently $1+1+1=0$, in $P$.

 If $\Xi$ is a hexagon in $P$ of dyadic type, then $\Xi=\hex{a}{a}{a^{-1}}{a^{-1}}{-1}{-1}=\Xi(a^{-1},-1)$ for some $a\neq -1$ by \autoref{prop: types of hexagons}. Thus $(a^{-1},-1)\in P^\pair$ and $a^{-1}-1-1=0$. This shows that $a^{-1}$ is uniquely determined in $R_P$, as $a^{-1}=2$, and that $P$ has at most one hexagon of dyadic type. If $1+1+1=0$, then $2=-1$ and thus $\Xi(2,-1)=\Xi(-1,-1)$ is of ternary type.
 
 If $P$ is integral, then there exists an injective morphism $P\to K$ into a field $K$, which allows us to consider the elements of $P$ as elements of $K$. If $\Xi$ is a hexagon in $P$ of hexagonal type, then $\Xi=\hex a{a^{-1}}{a^{-1}}aa{a^{-1}}$ for some $a\neq-1$ by \autoref{prop: types of hexagons}. Thus $a+a^{-1}=1$, which means that $a$ and $a^{-1}$ are the two distinct roots $\zeta_6$ and $\zeta_6^{-1}$ of the sixth cyclotomic polynomial
 \[
  T^2-T+1 \ = \ T\cdot (T+T^{-1}-1).
 \]
Consequently $\Xi$ is uniquely determined as $\Xi=\Xi(\zeta_6,\zeta_6^{-1})$. If $1+1+1=0$ in $P$, then $K$ is a field of characteristic $3$ and does not have any primitive $6$-th roots of unity. Therefore $P$ does not have a hexagon of hexagonal type if $1+1+1=0$.
\end{proof} 
 
\begin{ex}\label{ex: partial fields with many hexagons of hexagonal type}
 Either claim in \autoref{prop: hexagons in integral partial fields}, part \eqref{type3}, fails to be true for arbitrary partial fields, as the following examples attest. By \autoref{lemma: product of partial fields is a partial field}, the product $\F_4\times\F_4$ is a partial field. Since $\F_4$ has a unique hexagon of hexagonal type (cf.\ \autoref{subsection: Examples of hexagons}), $\F_4\times\F_4$ has two hexagons of hexagonal types by \autoref{prop: hexagons in products}.
 
 As a tensor product (cf.\ \autoref{rem: properties of lifts}), the partial field $\F_3\otimes\H$ (cf.\ \autoref{rem: non-integral partial fields}) has two hexagons, one of ternary type stemming from $\F_3$ and one of hexagonal type stemming from $\H$.
\end{ex}

\begin{cor}\label{cor: numbers of hexagons in Fq by type}
 Let $q$ be a prime power. Then $\F_q$ has 
 \begin{itemize}
  \item a hexagon of dyadic type if and only if $q$ is odd and not divisible by $3$;
  \item a hexagon of hexagonal type if and only if $q-1$ is divisible by $3$;
  \item a hexagon of ternary type if and only if $q$ is divisible by $3$;
  \item $\lfloor\frac{q-2}6\rfloor$ hexagons of near-regular type.
 \end{itemize}
 This information is organized in \autoref{table: number of heaxagon in Fq by type} according to the value of $q\pmod6$.
\end{cor}

\begin{table}[htb]
 \caption{The types of hexagons occurring in $\F_q$}
 \label{table: number of heaxagon in Fq by type}
 \begin{tabular}{|c|c|c|c|c|}
  \hline
  $q\pmod6$ & dyadic & hexagonal & ternary & near-regular \\
  \hline\hline
  $1$       & $\check$ & $\check$    & $\times$  & if $q\geq 13$ \\
  \hline
  $2$       & $\times$ & $\times$    & $\times$  & if $q\geq 8$ \\
  \hline
  $3$       & $\times$ & $\times$    & $\check$  & if $q\geq 9$ \\
  \hline
  $4$       & $\times$ & $\check$    & $\times$  & if $q\geq 16$ \\
  \hline
  $5$       & $\check$ & $\times$    & $\times$  & if $q\geq 11$ \\
  \hline
 \end{tabular}
\end{table}

\begin{proof}
 By \autoref{prop: types of hexagons}, $\F_q$ has a hexagon of dyadic type if and only if there is a $z\in\F_q^\times-\{-1\}$ such that $z+z-1=0$, i.e.\ $z^{-1}=1+1$ is invertible in $\F_q$ and different from $-1$, which is the case if and only if $q$ is odd and not divisible by $3$.
 
 By \autoref{prop: types of hexagons}, $\F_q$ has a hexagon of hexagonal type if and only if there is a $z\in\F_q^\times-\{-1\}$ such that $z+z^{-1}=1$ and $z^3=-1$. Thus $z$ is a sixth root of unity with $z^3=-1$, but $z\neq -1$. Therefore $\zeta_3=-z$ is a primitive third root of unity, which shows that $q-1=\#\F_q^\times$ is divisible by $3$. Conversely, if $q-1$ is divisible by $3$, then $\F_q$ contains a primitive third root of unity $\zeta_3$ and $z=-\zeta_3$ satisfies the relations $z+z^{-1}=1$ and $z^3=-1$. Note that in characteristic $2$, we have $z=\zeta_3$ and $z^3=1=-1$.
 
 By \autoref{prop: types of hexagons}, $\F_q$ has a hexagon of ternary type if and only if $-1-1-1=0$, which is the case precisely when $q$ is divisible by $3$.
 
 By \autoref{prop: hexagons in integral partial fields}, we have
 \[
  \sum_{\substack{\Xi\in\Hex(\F_q)\\\text{not of near-regular type}}} \#\Xi^\fund \ \leq \ 5.
 \]
 Since $\Xi^\fund$ has $6$ elements for every near-regular hexagon $\Xi$ in $\F_q$ and since every element of $\F_q$ but $0$ and $1$ appears in a unique hexagon by \autoref{prop: hexagons in partial fields}, we conclude that the number of near-regular hexagons in $\F_q$ is $\lfloor\frac{q-2}6\rfloor$, which completes the proof.
\end{proof}


\subsection{Examples}
\label{subsection: Examples of hexagons}

By \autoref{prop: hexagons in partial fields}, we know that every element of a finite field $\F_q$ different from $0$ and $1$ occurs in a unique hexagon. This allows us to determine the hexagons of a finite field as the hexagons $\Xi(a,1-a)$ for $a\in\F_q-\{0,1\}$. We show the outcome for $\F_q$ with $2\leq q\leq 13$: 

\begin{align*}
 \Hex(\F_2) \ &= \ \emptyset; \\
 \Hex(\F_3) \ &= \ \textstyle \bigg\{ \underset{\text{(ternary)}}{\hex222222} \bigg\}; \\
 \Hex(\F_4) \ &= \ \textstyle \bigg\{ \underset{\text{(hexagonal)}}{\hex{\alpha}{\alpha^2}{\alpha^2}{\alpha}{\alpha}{\alpha^2}} \bigg\} 
 && \hspace{-2.5cm} \text{where $\alpha$ is a root of $T^2+T-1$}; \\
 \Hex(\F_5) \ &= \ \textstyle \bigg\{ \underset{\text{(dyadic)}}{\hex332244} \bigg\}; \\
 \Hex(\F_7) \ &= \ \textstyle \bigg\{ \underset{\text{(dyadic)}}{\hex442266},\quad \underset{\text{(hexagonal)}}{\hex355335} \bigg\}; \\
 \Hex(\F_8) \ &= \ \textstyle \bigg\{ \underset{\text{(near-regular)}}{\hex{\alpha}{\alpha^3}{\alpha^6}{\alpha^4}{\alpha^2}{\alpha^5}} \bigg\} 
 && \hspace{-2.5cm} \text{where $\alpha$ is a root of $T^3+T-1$}; \\
 \Hex(\F_9) \ &= \ \textstyle \bigg\{ \underset{\text{(ternary)}}{\hex222222},\quad \underset{\text{(near-regular)}}{\hex{\alpha}{\alpha^2}{\alpha^7}{\alpha^6}{\alpha^3}{\alpha^5}}\bigg\} 
 && \hspace{-2.5cm} \text{where $\alpha$ is a root of $T^2+T-1$}; \\
 \Hex(\F_{11}) \ &= \ \textstyle \bigg\{ \underset{\text{(dyadic)}}{\hex6622{10}{10}},\quad \underset{\text{(near-regular)}}{\hex394587}\bigg\}; \\
 \Hex(\F_{13}) \ &= \ \textstyle \bigg\{ \underset{\text{(dyadic)}}{\hex7722{12}{12}},\quad \underset{\text{(hexagonal)}}{\hex4{10}{10}44{10}},\quad \underset{\text{(near-regular)}}{\hex3{11}9658} \bigg\}. \\
\end{align*}

Using \autoref{prop: nullset and hexagons}, we can derive from this list descriptions of finite fields as pastures in terms of generators and relations. For instance, we have
\[
 \F_4 \ \simeq \ \pastgen\Funpm{z}{z^2+z-1,z^3-1,1+1} \quad \text{and} \quad \F_5 \ \simeq \ \pastgen\Funpm{z}{z+z-1,z^2+1}.
\]

Conversely, we can use \autoref{prop: nullset and hexagons} to determine the hexagons of pastures that are given in terms of generators and relations. For instance,
\begin{align*}
 \Hex(\G) \ &= \ \textstyle \bigg\{\underset{\text{(near-regular)}}{\hex{z}{z^2}{z^{-1}}{z^{-2}}{-z}{-z^{-1}}} \bigg\}; \\
 \Hex(\S) \ &= \ \textstyle \bigg\{\underset{\text{(dyadic)}}{\hex1111{-1}{-1}} \bigg\}; \\
 \Hex(\W) \ &= \ \textstyle \bigg\{\underset{\text{(dyadic)}}{\hex1111{-1}{-1}}, \quad \underset{\text{(ternary)}}{\hex{-1}{-1}{-1}{-1}{-1}{-1}} \bigg\}; \\
 \Hex(\K) \ &= \ \textstyle \bigg\{\underset{\text{(ternary)}}{\hex111111} \bigg\}. \\
\end{align*}
Note that the two hexagons of $\W$ do not have disjoint support.


\section{The ternary and \texorpdfstring{$\WLUM$}{WLUM}-lifts}
\label{section: Ternary lifts}

By \cite[Thm.\ 5.9]{Baker-Lorscheid20}, the foundation of a matroid without large uniform minors (for short: \emph{$\WLUM$-matroid}) is isomorphic to the tensor product of finitely many copies of $\U$, $\D$, $\H$, $\F_3$ and $\F_2$. In this section, we will show that the coreflective hull $\Lifts_{\wlum}$ of such finite tensor products is the full subcategory of all pastures that are isomorphic to a (possibly infinite) tensor product of these five pastures. The coreflection, a.k.a.\ the $\WLUM$-lift, $\cL_{\wlum}:\Pastures\to\Lifts_{\wlum}$ can be described and computed explicitly, which in turn leads to various applications in \autoref{section: Applications}. A variation of this theme leads to the ternary lift $\cL_{\ternary}:\Pastures\to\Lifts_{\ternary}$, which is the coreflection onto the coreflective hull of all foundations of ternary matroids.

In this section, we follow a bottom-up approach: we first work towards an explicit description of the ternary lift of a pasture and in turn prove that it is the coreflection onto a coreflective category, which we identify, \emph{a posteriori}, with the coreflective hull of the foundations of ternary matroids. From these results, we deduce the corresponding facts for $\WLUM$-matroids.


\subsection{The lift of a hexagon}
\label{subsection: The lift of a hexagon}

The basic concept that we employ in the construction of the ternary lift and the $\WLUM$-lift is the following.

\begin{df}
 Let $P$ be a pasture and $\Xi\in\Hex(P)$. The \emph{lift of $\Xi$} is the pasture 
 \[
  \Lift_\Xi \ = \ \pastgen\Funpm{t_{a,b}\mid (a,b)\in\Xi^\pair}{S}
 \]
 where $S$ consists of the relations
 \begin{enumerate}[label={(L\arabic*)}]
  \item\label{L1} $t_{a,b}+t_{b,a}=1$;
  \item\label{L2} $t_{a,b}\cdot t_{a^{-1},-a^{-1}b}=1$;
  \item\label{L3} $t_{a,b}\cdot t_{-a^{-1}b,a^{-1}}\cdot t_{b^{-1},-ab^{-1}} =-1$;
 \end{enumerate}
 for all $(a,b)\in\Xi^\pair$, together with the canonical morphism
 \[  
  \begin{array}{cccc}
   \lambda_\Xi: & \Lift_\Xi  & \longrightarrow & P \\
                & t_{a,b} & \longmapsto     & a .
  \end{array}
 \]
\end{df}

\begin{prop}\label{prop: isomorphism types of hexagon lifts}
 Let $P$ be a pasture and $\Xi\in\Hex(P)$. Then $\Lift_\Xi$ is isomorphic to 
 \begin{itemize}
  \item $\U$ if $\Xi$ is of near-regular type,
  \item $\D$ if $\Xi$ is of dyadic type,
  \item $\H$ if $\Xi$ is of hexagonal type,
  \item $\F_3$ if $\Xi$ is of ternary type,
 \end{itemize}
 and the map $\lambda_P^\pair:\Lift_\Xi^\pair\to P^\pair$ restricts a bijection $\Lift_\Xi^\pair\to\Xi^\pair$.
\end{prop}

\begin{proof}
 By \ref{L1} we have $t_{a,b}+t_{b,a}-1\in N_{\Lift_\Xi}$, and thus the association $x\mapsto t_{a,b}$ and $y\mapsto t_{b,a}$ defines a morphism
 \[
  \varphi: \ \U \ = \ \pastgen\Funpm{x,y}{x+y-1} \ \longrightarrow \ \Lift_\Xi
 \]
 of pastures, which is surjective since  
 \begin{align*}
  t_{a^{-1},-a^{-1}b} \ &= \ t_{a,b}^{-1} \ = \ \varphi(x^{-1}), \\
  t_{b^{-1},-ab^{-1}} \ &= \ t_{b,a}^{-1} \ = \ \varphi(y^{-1}), \\
  t_{-a^{-1}b,a^{-1}} \ &= \ -t_{a,b}^{-1} \cdot t_{b^{-1},-ab^{-1}}^{-1} \ = \ -\varphi(x)^{-1}\cdot\varphi(y^{-1})^{-1} \ = \ \varphi(-x^{-1}y), \\
  t_{-ab^{-1},b^{-1}} \ &= \ -t_{b,a}^{-1} \cdot t_{a^{-1},-a^{-1}b}^{-1} \ = \ -\varphi(x^{-1})^{-1}\cdot\varphi(x)^{-1} \ = \ \varphi(-xy^{-1}), 
 \end{align*}
 which covers all generators of $\Lift_\Xi$ by \autoref{lemma: fundamental pairs in a hexagon}. Note that the relations \ref{L1}--\ref{L3} between the generators $t_{a,b}$ of $\Lift_\Xi$ correspond to the relations in $\U$. 
 
 In the case that $\Xi$ is of near-regular type, the elements
 \[
  t_{a,b},\quad t_{b,a},\quad t_{a^{-1},-a^{-1}b},\quad t_{-a^{-1}b,a^{-1}},\quad t_{b^{-1},-ab^{-1}},\quad t_{-ab^{-1},b^{-1}}
 \]
 of $\Lift_\Xi$ are pairwise distinct and $\varphi$ is an isomorphism, which establishes our claim for hexagons of near-regular type.
 
 In the case that $\Xi$ is of dyadic type, $\Stab_{D_3}(a,b)=\gen\sigma$ for some $(a,b)\in\Xi^\pair$, and thus $\varphi$ induces an isomorphism $\U/\gen\sigma\to\Lift_\Xi$. By \cite[Prop.\ 5.8]{Baker-Lorscheid20}, $\U/\gen\sigma\simeq\D$, which establishes our claim for hexagons of dyadic type.

 In the case that $\Xi$ is of hexagonal type, $\Stab_{D_3}(a,b)=\gen\rho$ for $(a,b)\in\Xi^\pair$, and thus $\varphi$ induces an isomorphism $\U/\gen\rho\to\Lift_\Xi$. By \cite[Prop.\ 5.8]{Baker-Lorscheid20}, $\U/\gen\rho\simeq\H$, which establishes our claim for hexagons of hexagonal type.

 In the case that $\Xi$ is of ternary type, $\Stab_{D_3}(-1,-1)=D_3$ for the unique element $(-1,-1)$ of $\Xi^\pair$, and thus $\varphi$ induces an isomorphism $\U/D_3\to\Lift_\Xi$. By \cite[Prop.\ 5.8]{Baker-Lorscheid20}, $\U/D_3\simeq\F_3$, which establishes our claim for hexagons of ternary type.
 
 Note that in all cases, the map $\lambda_P^\pair:\Lift_\Xi^\pair\to P^\pair$ restricts to a bijection $\Lift_\Xi^\pair\to\Xi^\pair$.
\end{proof}


\subsection{Ternary lifts}
\label{subsection: Ternary lifts}

The name of the following construction of $\Lift_{\ternary} P$ stems from the fact that every $P$-representation of a ternary matroid lifts to $\Lift_{\ternary} P$, as explained in \autoref{subsection: The lift theorem for ternary matroids}.

\begin{df}
 Let $P$ be a pasture. The \emph{ternary lift of $P$} is the pasture 
 \[
  \Lift_{\ternary} P \ = \ \bigotimes_{\Xi\in\Hex(P)} \Lift_\Xi,
 \]
 together with the morphism $\lambda_P=\bigotimes\lambda_\Xi:\Lift_{\ternary} P\to P$ which sends an element $\Big(\otimes t_{a_\Xi,b_\Xi}\Big)$ of $\Lift_{\ternary} P \ = \ \bigotimes \Lift_\Xi$ to
 \[
  \lambda_P\Big( \otimes t_{a_\Xi,b_\Xi}) \Big) \ = \ \prod_{\Xi\in\Hex(P)} a_\Xi.
 \]
\end{df}

\begin{lemma}\label{lemma: fundamental pairs of a pasture equal those of the ternary lift}
 Let $P$ be a pasture and $\lambda_P:\Lift_{\ternary} P\to P$ its ternary lift. Let $\iota_\Xi:\Lift_\Xi\to\Lift_{\ternary} P$ be the canonical inclusion for $\Xi\in\Hex(P)$. Then both maps
 \[
  \amalg \iota_\Xi^\pair: \ \coprod_{\Xi\in\Hex(P)} \ \Lift_\Xi^\pair \ \longrightarrow \ \Lift_{\ternary} P^\pair \qquad \text{and} \qquad \lambda_P^\pair: \ \Lift_{\ternary} P^\pair \ \longrightarrow \ P^\pair
 \]
 are bijections.
\end{lemma}

\begin{proof}
 By \autoref{prop: isomorphism types of hexagon lifts}, the composition 
 \[
  \Lift_\Xi^\pair \ \stackrel{\iota_\Xi^\pair}{\longrightarrow} \ \Lift_{\ternary} P^\pair \ \stackrel{\lambda_P^\pair}{\longrightarrow} \ P^\pair
 \]
 is a bijection onto $\Xi^\pair$. By \autoref{prop: nullset and hexagons}, $P^\pair=\coprod_{\Xi\in\Hex(P)}\Xi^\pair$, and therefore the composition
 \[
  \coprod_{\Xi\in\Hex(P)} \Lift_\Xi^\pair \ \stackrel{\amalg\iota_\Xi^\pair}{\longrightarrow} \ \Lift_{\ternary} P^\pair \ \stackrel{\lambda_P^\pair}{\longrightarrow} \ P^\pair
 \]
 is a bijection. The assertion of the lemma follows if we can show that $\amalg\iota_\Xi^\pair$ is surjective.
 
 To do so, consider $(a,b)\in \Lift_{\ternary} P^\pair$, i.e.\ $a+b-1\in N_{\Lift_{\ternary} P}$. By the construction of the tensor product, we have
 \[
  N_{\Lift_{\ternary} P} \ = \ \bigcup_{\Xi\in\Hex(P)} \Lift_{\ternary} P^\times \cdot \Big\{ \iota_\Xi(\tilde a)+\iota_\Xi(\tilde b)+\iota_\Xi(\tilde c) \, \Big| \, \tilde a+\tilde b+\tilde c\in N_{\Lift_\Xi} \Big\}.
 \]
 Thus there are $\Xi\in\Hex(P)$, $\tilde a,\tilde b,\tilde c\in\Lift_\Xi$ and $d\in\Lift_{\ternary} P^\times$ such that $a=d\iota_\Xi(\tilde a)$, $b=d\iota_\Xi(\tilde b)$ and $-1=d\iota_\Xi(\tilde c)$. In particular, $d=-\iota_\Xi(\tilde c)^{-1}=\iota_\Xi(-\tilde c^{-1})$, and therefore $a=\iota_\Xi(a')$ and $b=\iota_\Xi(b')$ for $a'=-\tilde c^{-1}\tilde a$ and $b'=-\tilde c^{-1}\tilde b$. This shows that $(a',b')\in\Lift_\Xi^\pair$ and $(a,b)=\iota_\Xi^\pair(a',b')$. We conclude that $\amalg\iota_\Xi^\pair$ is surjective, which completes the proof.
\end{proof}


\subsection{Examples}
\label{subsection: Examples of lifts}

Thanks to \autoref{prop: isomorphism types of hexagon lifts} and the descriptions in \autoref{subsection: Four types of hexagons}, we see at once that each of $\U$, $\D$, $\H$ and $\F_3$ is isomorphic to its respective lift; more precisely, each of 
\[
 \lambda_\U:\Lift_{\ternary}\U \stackrel\sim\longrightarrow \U, \quad \lambda_\D:\Lift_{\ternary}\D \stackrel\sim\longrightarrow \D, \quad  \lambda_\H:\Lift_{\ternary}\H \stackrel\sim\longrightarrow \H, \quad \lambda_{\F_3}:\Lift_{\ternary}\F_3 \stackrel\sim\longrightarrow \F_3
\]
is an isomorphism.

With \autoref{prop: isomorphism types of hexagon lifts} and the descriptions in \autoref{subsection: Examples of hexagons}, we readily compute the following examples of lifts:
\begin{align*}
 \Lift_{\ternary}\F_2 \ &\simeq \ \Funpm, & \Lift_{\ternary}\F_4 \ &\simeq \ \H,            & \Lift_{\ternary}\F_5    \ &\simeq \ \D,            & \Lift_{\ternary}\F_7    \ &\simeq \ \D\otimes\H,          \\
 \Lift_{\ternary}\F_8 \ &\simeq \ \U,     & \Lift_{\ternary}\F_9 \ &\simeq \ \F_3\otimes\U, & \Lift_{\ternary}\F_{11} \ &\simeq \ \D\otimes\U,   & \Lift_{\ternary}\F_{13} \ &\simeq \ \D\otimes\H\otimes\U, \\ 
 \Lift_{\ternary}\G   \ &\simeq \ \U,     & \Lift_{\ternary}\S   \ &\simeq \ \D,            & \Lift_{\ternary}\W      \ &\simeq \ \F_3\otimes\D, & \Lift_{\ternary}\K      \ &\simeq \ \F_3.
\end{align*}


\subsection{The universal property}
\label{subsection: the universal property of ternary lifts}

Let $\Lifts_{\ternary}$ be the full subcategory of $\Pastures$ whose objects are pastures that are isomorphic to a possibly infinite tensor product of copies of $\U$, $\D$, $\H$ and $\F_3$. We call objects in $\Lifts_{\ternary}$ \emph{ternary lifts}, a terminology that is justified by the following result.

\begin{prop} \label{prop: universal property of ternary lifts}
 The association $\Lift_\ternary$ defines a coreflection from $\Pastures$ to $\Lifts_\ternary$, i.e.\ for every morphism $\varphi:L\to P$ from a ternary lift $L$ to a pasture $P$, there is a unique morphism $\hat\varphi:L\to\Lift_{\ternary} P$ such that $\varphi=\lambda_P\circ\hat\varphi$.
\end{prop}

\begin{proof}
 Let $L\simeq\bigotimes_{i\in I} F_i$ with $F_i\in\{\U,\D,\H,\F_3\}$ and let $\iota_i:F_i\to L$ be the canonical inclusion. By the universal property of the tensor product (\autoref{lemma: universal property of tensor products}), it suffices to show for every factor $F_i$ that the composition $\varphi_i: F_i\stackrel{\iota_i}{\to} L\stackrel\varphi\to P$ lifts uniquely along $\lambda_P:\Lift_{\ternary} P\to P$ to a morphism $\hat\varphi_i:F_i\to\Lift_{\ternary} P$. Thus we can assume without loss of generality that $L=F_i\in\{\U,\D,\H,\F_3\}$.
 
 A lift of $\varphi$, if it exists, is unique since $\lambda_P^\pair:\Lift_{\ternary} P^\pair\to P^\pair$ is a bijection by \autoref{lemma: fundamental pairs of a pasture equal those of the ternary lift}, and thus a lift of $\varphi$ necessarily maps a fundamental element $a$ to $t_{\varphi(a),\varphi(b)}$ if $(a,b)\in F^\pair$.
 
 We are left with the proof of existence. We first investigate the case $L=\U=\pastgen\Funpm{x,y}{x+y-1}$. By the universal property of quotients of a free algebra (\autoref{prop: universal property of algebras and quotients}), a morphism $\varphi:\U\to P$ determines a fundamental pair $\Big(\varphi(x),\varphi(y)\Big)\in P^\pair$, and conversely every fundamental pair in $P^\pair$ stems from a unique morphism $\U\to P$. Since $(t_{\varphi(x),\varphi(y)},t_{\varphi(y),\varphi(x)})$ is a fundamental pair of $\Lift_{\ternary} P$, the associations $x\mapsto t_{\varphi(x),\varphi(y)}$ and $y\mapsto t_{\varphi(y),\varphi(x)}$ define a morphism $\hat\varphi:\U\to\Lift_{\ternary} P$, which is a lift of $\varphi$ since $\lambda_P\circ\hat\varphi(x)=\lambda_P(t_{\varphi(x),\varphi(y)})=\varphi(x)$, and similarly with $x$ and $y$ exchanged.
 
 A morphism $\varphi:\D=\pastgen\Funpm{z}{z+z-1}\to P$ corresponds to a fundamental pair $(a,b)=\Big(\varphi(z),\varphi(z)\Big)$ in $P$ such that $(a,b)=(b,a)$ and thus $t_{a,b}=t_{b,a}$. This shows that the association $z\mapsto t_{a,b}$ defines a morphism $\hat\varphi:\D\to\Lift_{\ternary} P$, which lifts $\varphi$ since$\lambda_P\circ\hat\varphi(z)=\lambda_P(t_{\varphi(z),\varphi(z)})=\varphi(z)$.
 
 A morphism $\varphi:\H=\pastgen\Funpm{z}{z+z^{-1}-1}\to P$ corresponds to a fundamental pair $(a,b)=\Big(\varphi(z),\varphi(z^{-1})\Big)$ in $P$ such that $(a,b)=(b^{-1},-ab^{-1})=(-a^{-1}b,a^{-1})$ and thus $t_{a,b}=t_{b^{-1},-ab^{-1}}=t_{-a^{-1}b,a^{-1}}$. This shows that the association $z\mapsto t_{a,b}$ defines a morphism $\hat\varphi:\H\to\Lift_{\ternary} P$, which lifts $\varphi$ since $\lambda_P\circ\hat\varphi(z)=\lambda_P(t_{\varphi(z),\varphi(z^{-1})})=\varphi(z)$.
 
 There is a morphism $\varphi:\F_3\to P$ if and only if $(-1,-1)$ is a fundamental pair in $P$. In this case $\Xi(-1,-1)^\pair=\Big\{(-1,-1)\Big\}$ and $\Lift_{\Xi(-1,-1)}\simeq\F_3$. Thus, if it exists, the morphism $\varphi:\F_3\to P$ lifts to $\hat\varphi:\F_3\stackrel\sim\to\Lift_{\Xi(-1,-1)}\to\Lift_{\ternary} P$, which completes the proof.
\end{proof}

\begin{rem}\label{rem: properties of lifts}
 As a formal consequence of the fact that $\Lifts_{\ternary}$ is a coreflective subcategory of $\Pastures$, we obtain:
 \begin{enumerate}
  \item\label{lift1} The ternary lift is idempotent, i.e.\ $\lambda_{\Lift_{\ternary} P}:\Lift_{\ternary}\Lift_{\ternary} P\to\Lift_{\ternary} P$ is an isomorphism for every pasture $P$.
  \item\label{lift2} Colimits in $\Lifts_{\ternary}$ can be computed in $\Pastures$.
 \end{enumerate} 
 The following additional properties are not deep, but require multiple case considerations which we do not know how to present in a compact way. Since these results are not used elsewhere in the paper, we omit the proofs.
 \begin{enumerate}\addtocounter{enumi}{2}
  \item\label{lift3} Limits in $\Lifts_{\ternary}$ agree with the ternary lifts of the corresponding limits in $\Pastures$. In particular, $\prod\lambda_{P_i}:\Lift_{\ternary}\Big(\prod\Lift_{\ternary} P_i\Big)\to\Lift_{\ternary}\Big(\prod P_i\Big)$ is an isomorphism for every family of pastures $\{P_i\}_{i\in I}$.
  \item\label{lift4} If there is at most one $i\in I$ such that $1+1-1\in N_{P_i}$ and $1+1\notin N_{P_i}$, then there is a canonical isomorphism $\Lift_{\ternary}\Big(\bigotimes P_i\Big)\to\bigotimes\Lift_{\ternary} P_i$.
  \item\label{lift5} Let $P=\bigotimes P_i$ and let $\iota_i:P_i\to P$ be the canonical inclusion. If there is at most one $i\in I$ such that $1+1-1\in N_{P_i}$ or $1+1+1\in N_{P_i}$, then
        \[
         P^\pair \ = \ \coprod_{i\in I} \iota_i^\pair(P_i^\pair) \qquad \text{and} \qquad \Hex(P) \ = \ \coprod_{i\in I} \iota_i^\hexagon\Big(\Hex(P_i)\Big).
        \]
 \end{enumerate}
 To understand the necessity of the assumptions in \eqref{lift3}, note that $\Lift_{\ternary}\F_3\times\Lift_{\ternary}\F_3=\F_3\times\F_3$ is not isomorphic to $\Lift_{\ternary}\Big(\F_3\times\F_3\Big)\simeq\F_3$, which shows that limits of ternary lifts, computed in $\Pastures$, are not ternary lifts in general. To understand the necessity of the assumptions in \eqref{lift4}, note that $\Lift_{\ternary}\Big(\S\otimes\S\Big)\simeq\Lift_{\ternary}\S\simeq\D$ differs from $\Lift_{\ternary}\S\otimes\Lift_{\ternary}\S\simeq\D\otimes\D$. To understand the necessity of the assumptions in \eqref{lift5}, note that $(\F_3\otimes\F_3)^\pair=\F_3^\pair$ differs from $\F_3^\pair\amalg\F_3^\pair$ and that $\Hex(\F_3\otimes\F_3)=\Hex(\F_3)$ differs from $\Hex(\F_3)\amalg\Hex(\F_3)$.
\end{rem}


\subsection{The lift theorem for ternary matroids}
\label{subsection: The lift theorem for ternary matroids}

The key result in the proof of the lift theorem for ternary matroids is that every ternary foundation is contained in $\Lifts_{\ternary}$. 

\begin{thm}[Structure theorem for $\WLUM$-foundations, {\cite[Thm.\ 5.9]{Baker-Lorscheid20}}]\label{thm: structure theorem for foundations of matroids wlum}
 Let $M$ be a matroid without large uniform minors and let $F_M$ be its foundation. Then 
 \[
  F_M \ \simeq \ F_1\otimes \dotsb \otimes F_r
 \]
 for some $r\geq 0$ and pastures $F_1,\dotsc,F_r\in\{\U,\D,\H,\F_3,\F_2\}$.
\end{thm}

\begin{thm}[Structure theorem for ternary foundations]\label{thm: structure theorem for foundations of ternary matroids}
 Let $M$ be a ternary matroid and let $F_M$ be its foundation. Then 
 \[
  F_M \ \simeq \ F_1\otimes \dotsb \otimes F_r
 \]
 for some $r\geq 0$ and pastures $F_1,\dotsc,F_r\in\{\U,\D,\H,\F_3\}$.
\end{thm}

\begin{proof}
 Let $M$ be a ternary matroid of rank $r$ on $E$ with foundation $F_M$. Since $M$ is $\F_3$-representable, there exists a rescaling class $[\Delta]$ of $M$ over $\F_3$ and therefore a morphism $\chi_{\Delta}:F_M\to \F_3$. Since a ternary matroid does not have any $U^2_5$ or $U^3_5$-minors, we conclude by \autoref{thm: structure theorem for foundations of matroids wlum} that $F_M\simeq F_1\otimes \dotsb \otimes F_r$ for some $F_1,\dotsc,F_r\in\{\U,\D,\H,\F_3,\F_2\}$. The composition with the canonical inclusions $\iota_i:F_i\to F_M$ into the tensor product yields morphisms $\chi_\Delta\circ\iota_i:F_i\to \F_3$. Since there is no morphism from $\F_2$ to $\F_3$, we conclude that $F_1,\dotsc,F_r\in\{\U,\D,\H,\F_3\}$.
\end{proof}

\begin{thm}[Lift theorem for ternary matroids]\label{thm: lift theorem for ternary matroids}
 Let $P$ be a pasture and let $\lambda_P:\Lift_{\ternary} P\to P$ be its ternary lift. Let $M$ be a ternary matroid. Then every $P$-representation of $M$ lifts uniquely to $\Lift_{\ternary} P$ up to rescaling equivalence.
\end{thm}

\begin{proof}
 By the structure theorem for foundations of ternary matroids (\autoref{thm: structure theorem for foundations of ternary matroids}), the foundation $F_M$ of $M$ is isomorphic a tensor product $F_1\otimes \dotsb \otimes F_r$ for some $F_1,\dotsc,F_r\in\{\U,\D,\H,\F_3\}$. By the universal property of the ternary lift (\autoref{prop: universal property of ternary lifts}), every morphism $\chi:F_M\to P$ lifts to a unique morphism $\hat\chi:F_M\to\Lift_{\ternary} P$ along $\lambda_P$. By the defining property of the foundation of $M$, this shows that every $P$-rescaling class of $M$ lifts to a unique $\Lift_{\ternary} P$-rescaling class of $M$, and thus our claim.
\end{proof}

\begin{rem}\label{rem: ternary lifts are optimal}
 As claimed in \cite[Rem.\ 5.11]{Baker-Lorscheid20}, each of $\U$, $\D$, $\H$ and $\F_3$ occurs as the foundation of a ternary matroid. Therefore we can write every ternary lift as a colimit of a ternary foundation, which shows that $\Lift_\ternary$ is the coreflective hull of all ternary foundations. This shows that the ternary lift $\Lift_\ternary$ is the best-possible categorical lift for representations of ternary matroids.
\end{rem}


\subsection{The lift theorem for \texorpdfstring{$\WLUM$}{WLUM}-matroids}
\label{subsection: The lift theorem for matroids wlum}

A slight variant of the ternary lift which keeps track of the relation $-1=1$, if present, yields a lift theorem for $\WLUM$-matroids. 

We define the category $\Lifts_{\wlum}$ of \emph{$\WLUM$-lifts} as the full subcategory of $\Pastures$ whose objects are pastures which are (possibly infinite) tensor products of copies of $\U$, $\D$, $\H$, $\F_3$ and $\F_2$. Note that if $-1=1$ in a pasture $P$, then there exists a unique morphism $\iota_{P}:\F_2\to P$, and thus every morphism $f:P'\to P$ extends uniquely to a morphism
\[
 f\otimes \iota_P: \ P'\otimes\F_{2} \ \longrightarrow \ P.
\]

\begin{df}
 Let $P$ be a pasture. The \emph{$\WLUM$-lift of $P$} is the pasture
 \[
  \Lift_{\wlum} P \ = \ \bigg(\bigotimes_{\Xi\in\Hex(P)} \Lift_\Xi \bigg)\otimes\im(\Funpm\to P) \ = \ \begin{cases}
                                                                    \Lift_{\ternary} P           & \text{if $-1\neq 1$ in $P$,}\\
                                                                    \Lift_{\ternary} P\otimes\F_2 & \text{if $-1=1$ in $P$,}
                                                                   \end{cases}
 \]
 together with the canonical morphism 
 \[
  \lambda_{\WLUM,P}: \ \Lift_{\wlum} P \ \longrightarrow \ P,
 \]
 which is equal to $\lambda_{\ternary,P}$ in case that $\Lift_{\wlum}P=\Lift_{\ternary}P$ and which is equal to $\lambda_{\ternary}\otimes\iota_P$ in case that $\Lift_{\wlum}P=\Lift_{\ternary}P\otimes\F_2$.
\end{df}

If the context is clear, we write $\lambda_P=\lambda_{\WLUM,P}:\Lift_{\wlum}P\to P$ for the $\WLUM$-lift of $P$.

\begin{prop}\label{prop: universal property of WLUM-lifts}
 Let $P$ be a pasture with $\WLUM$-lift $\lambda_P:\Lifts_{\wlum}P\to P$. Then for every $\WLUM$-lift $F$ and every morphism $\varphi:F\to P$, there is a unique morphism $\hat\varphi:F\to\Lift_{\wlum} P$ such that $\varphi=\lambda_P\circ\hat\varphi$.
\end{prop}

\begin{proof}
Since $F$ is isomorphic to a tensor product $\bigotimes_{i\in I} F_i$ with $F_i\in\{\U,\D,\H,\F_3,\F_2\}$, the morphism $\varphi: F\to P$ lifts uniquely to $\Lift_{\wlum} P$ if and only if the induced morphisms $\varphi_i:F_i\to P$ lift uniquely to $\Lift_{\wlum} P$ for every $i\in I$. For factors $F_i\in \{\U,\D,\H,\F_3\}$, this follows by the same argument as in the proof of the universal property of ternary lifts (\autoref{prop: universal property of ternary lifts}).
 
If $F\simeq \bigotimes_{i\in I} F_i$ has a factor $F_i=\F_2$, then $-1=1$ in $F$ and consequently also in $P$ and $\Lift_{\wlum} P$. Thus there is a morphism $\F_2\to \Lift_{\wlum} P$, which is the unique lift of $\F_2\to P$ by the uniqueness of morphisms from $\F_2$ into other pastures.
\end{proof}

\begin{thm}[Lift theorem for $\WLUM$-matroids]\label{thm: lift theorem for matroids wlum}
 Let $M$ be a matroid without large uniform minors and $P$ a pasture with $\WLUM$-lift $\lambda_P:\Lift_{\wlum} P\to P$. Then every $P$-representation of $M$ lifts uniquely up to rescaling equivalence along $\lambda_P:\Lift_{\wlum} P\to P$.
\end{thm}

\begin{proof}
 This follows at once from the structure theorem for $\WLUM$-foundations (\autoref{thm: structure theorem for foundations of matroids wlum}), which shows that the foundation $F_M$ of $M$ is a $\WLUM$-lift, and the universal property of $\WLUM$-lifts (\autoref{prop: universal property of WLUM-lifts}).
\end{proof}

\begin{rem}\label{rem: WLUM-lifts are optimal}
 Besides $\U$, $\D$, $\H$ and $\F_3$ (cf.\ \autoref{rem: ternary lifts are optimal}), $\F_2$ is also a foundation, e.g.\ of the Fano matroid (cf.\ \cite[Lemma 4.18]{Baker-Lorscheid20}). Therefore we can write every $\WLUM$-lift as a colimit of a $\WLUM$-foundation, which shows that $\Lift_\wlum$ is the coreflective hull of all $\WLUM$-foundations. This shows that the $\WLUM$-lift $\Lift_\wlum$ is the best-possible categorical lift for representations of $\WLUM$-matroids.
\end{rem}



\begin{rem}
We conclude this section with some remarks on ternary and $\WLUM$-lifts of partial fields.
To begin with, we note that both the ternary and the $\WLUM$-lift of a partial field is again a partial field. 
 
 In the case of ternary lifts, this follows from the fact that tensor products of copies of the universal rings
 \par\needspace{2\baselineskip}  
 \begin{align*}
  R_\U \ &= \ \Z[x^{\pm1},y^{\pm1}]\,/\,\gen{x+y-1}, &  
  R_\D \ &= \ \Z[z^{\pm1}]\,/\,\gen{z+z-1}, \\
  R_\H \ &= \ \Z[z^{\pm1}]\,/\,\gen{z^3+1,z^2+z-1}, &
  R_{\F_3} \ &= \ \F_3
 \end{align*}
 contain each factor as a subring. (In the case of $\WLUM$-lifts, we have to avoid the pairing of a factor $\F_2$ with a factor of type $\D$ or $\F_3$, which does not occur thanks to \autoref{prop: hexagons in integral partial fields}.)
 
 Pendavingh and van Zwam consider in \cite[Section 3]{Pendavingh-vanZwam10b} morphisms of partial fields $f: P'\to P$ that allow for a lifting function $l:P^\fund\to P'$, which is a function with 
 \[
  f\circ l(a) \ = \ a, \qquad l(a^{-1}) \ = \ l(a)^{-1}, \qquad \text{and} \qquad l(a)+l(b) \ = \ 1
 \]
 for all $a\in P^\fund$ and $(a,b)\in P\pair$. Using the fact that the map $P^\pair\to P^\fund$ with $(a,b)\mapsto a$ is a bijection for partial fields, and following through the construction of $\Lift_\ternary P$, we see that a lifting function $l:P\to P'$ corresponds to the morphism $\Lift_\ternary P\to P'$ that maps $t_{a,b}$ to $l(a)$. 
 
From these observations, it is not difficult to recover \cite[Theorem 3.5]{Pendavingh-vanZwam10b} for ternary matroids from our methods. In particular, note that the Fano matroid $F_7$, the uniform matroid $U_5^2$ and their duals form a complete list of forbidden minors for ternary matroids. 

With a bit more effort, the general case of \cite[Theorem 3.5]{Pendavingh-vanZwam10b} can also be deduced from our methods by mixing $\GRS$-lifts and ternary lifts. We omit the details. 
\end{rem}

\section{Applications}
\label{section: Applications}


\subsection{Applications of ternary and \texorpdfstring{$\WLUM$}{WLUM}-lifts}
\label{subsection: Applications of ternary and WLUM-lifts}

There are a number of interesting results which are immediate consequences of the lift theorem for ternary matroids (\autoref{thm: lift theorem for ternary matroids}). As a first example, we have the following short proof of a celebrated theorem of Tutte (\cite{Tutte58b}).

\begin{thm}\label{thm: a matroid is regular iff it is binary and ternary}
 A matroid is regular if and only if it is binary and ternary.
\end{thm}

\begin{proof}
 Since there are morphisms $\Funpm\to\F_2$ and $\Funpm\to\F_3$, every regular matroid is binary and ternary. The converse follows from \autoref{thm: lift theorem for ternary matroids}, noting that $\Lift_{\ternary}\F_2=\Funpm$.
\end{proof}

Further consequences are the following. In each case (with exception of the unique lifting of orientations to $\D$-rescaling classes, which has been proven in \cite[Thm.\ 6.9]{Baker-Lorscheid20}), the uniqueness assertion is novel.

\begin{thm}\label{thm: unique lifts of representations of ternary matroids}
 Let $M$ be a ternary matroid. Then up to rescaling equivalence,
 \begin{enumerate}
  \item every $\F_4$-representation of $M$ lifts uniquely to $\H$;
  \item every $\F_5$-representation of $M$ lifts uniquely to $\D$;
  \item every $\F_7$-representation of $M$ lifts uniquely to $\D\otimes\H$;
  \item every $\F_8$-representation of $M$ lifts uniquely to $\U$;
  \item every $\F_9$-representation of $M$ lifts uniquely to $\F_3\otimes\U$;
  \item every $\F_{11}$-representation of $M$ lifts uniquely to $\D\otimes\U$;
  \item every $\F_{13}$-representation of $M$ lifts uniquely to $\D\otimes\H\otimes\U$;
  \item every $\G$-representation of $M$ lifts uniquely to $\U$;
  \item every $\S$-representation of $M$ lifts uniquely to $\D$;
  \item every $\W$-representation of $M$ lifts uniquely to $\F_3\otimes\D$.
 \end{enumerate}
\end{thm}

\begin{proof}
 This follows at once from \autoref{thm: lift theorem for ternary matroids} and the examples of ternary lifts provided in \autoref{subsection: Examples of lifts}.
\end{proof}

Following the same template, the lift theorem for $\WLUM$-matroids (\autoref{thm: lift theorem for matroids wlum}) yields the following result.

\begin{thm}\label{thm: unique lifts of representations of matroids wlum}
 Let $M$ be a matroid without large uniform minors. Then up to rescaling equivalence,
 \begin{enumerate}
  \item every $\F_4$-representation of $M$ lifts uniquely to $\F_2\otimes\H$;
  \item every $\F_8$-representation of $M$ lifts uniquely to $\F_2\otimes\U$.
 \end{enumerate}
 Moreover, the conclusions of all parts of \autoref{thm: unique lifts of representations of ternary matroids} except for (1) and (4) still hold if we replace the assumption that $M$ is a ternary matroid by the assumption that $M$ is a matroid without large uniform minors.
\end{thm}

\begin{proof}
 The numbered items (1) and (2) follow from \autoref{thm: lift theorem for matroids wlum} and the characterizations $\Lift_{\ternary} \F_4\simeq\H$ and $\Lift_{\ternary}\F_8\simeq\U$ from \autoref{subsection: Examples of lifts}.
 The remaining assertions about \autoref{thm: unique lifts of representations of ternary matroids} follow directly from the observation that $\Lift_{\ternary}P = \Lift_{\wlum} P$ if $-1\neq 1$ in $P$.
\end{proof}


\subsection{Incidences for rescaling classes}
\label{subsection: Incidences for rescaling classes}

Let $M$ be a ternary matroid with foundation $F_M$, let $P$ be a pasture, and let 
$\cX_M(P)=\Hom(F_M,P)$ be the corresponding rescaling class space. By \autoref{thm: lift theorem for ternary matroids}, the ternary lift $\lambda_P:\Lift_{\ternary} P\to P$ induces a bijection
\[
 \cX_M(\Lift_{\ternary} P) \ = \ \Hom(F_M,\Lift_{\ternary} P) \ \stackrel{\lambda_{P,\ast}}\longrightarrow \ \Hom(F_M,P) \ = \ \cX_M(P)
\]
between the rescaling classes of $M$ over $\Lift_{\ternary} P$ and $P$. Therefore an isomorphism $\varphi:\Lift_{\ternary} P\to\Lift_{\ternary} Q$ of ternary lifts induces a bijection
\[
 \cX_M(P) \ \stackrel{\lambda_{P,\ast}^{-1}}\longrightarrow \ \Hom(F_M,\Lift_{\ternary} P) \ \stackrel{\varphi_\ast}\longrightarrow \ \Hom(F_M,\Lift_{\ternary} Q) \ \stackrel{\lambda_{Q,\ast}}\longrightarrow \ \cX_M(Q)
\]
of rescaling class spaces for every ternary matroid $M$. We will exploit this observation throughout this section.

A first application is the following.

\begin{thm}\label{thm: rescaling class spaces of G and S}
 Let $M$ be a ternary matroid. Then there are natural bijections
 \[
  \cX_M(\G) \ \stackrel\sim\longrightarrow \ \cX_M(\F_8) \qquad \text{and} \qquad \cX_M(\S) \ \stackrel\sim\longrightarrow \ \cX_M(\F_5).
 \]
\end{thm}

\begin{proof}
 This follows at once from the previous considerations due to the identifications $\Lift_{\ternary}\G\simeq\U\simeq\Lift_{\ternary}\F_8$ and $\Lift_{\ternary}\S\simeq\D\simeq\Lift_{\ternary}\F_5$.
\end{proof}


\subsection{Hexagons in products}
\label{subsection: Hexagons in products}

We will find even more interesting incidences between products of rescaling class spaces in the next section. First, we need to understand how the hexagons in a product of pastures relate to the hexagons of the factors.

Let $P_1$ and $P_2$ be pastures and let $P=P_1\times P_2$ be their product. Since for $a=(a_1,a_2)$ and $b=(b_1,b_2)$ in $P$, $a+b=1$ if and only if $a_i+b_i=1$ for $i=1,2$, we obtain a bijection
\[
 \begin{array}{cccc}
  \psi: & P^\pair & \longrightarrow & P_1^\pair\times P_2^\pair, \\
        & (a,b)   & \longmapsto     & \Big((a_1,b_1),(a_2,b_2)\Big)
 \end{array}
\]
which is readily verified to be $D_3$-invariant and therefore induces a map
\[
 \begin{array}{cccc}
  \Psi: & \Hex(P)  & \longrightarrow & \Hex(P_1)\times \Hex(P_2) \\
        & \Xi(a,b) & \longmapsto     & \Big(\Xi(a_1,b_1),\Xi(a_2,b_2)\Big).
 \end{array}
\]
Recall from \autoref{def: orbit length} that the orbit length of a hexagon $\Xi$ is $\mu_\Xi=\#\Xi^\pair$.

\begin{prop}\label{prop: hexagons in products}
 Let $P=P_1\times P_2$ be the product of two pastures with associated maps $\psi:P^\pair\to P_1^\pair\times P_2^\pair$ and $\Psi:\Hex(P)\to\Hex(P_1)\times \Hex(P_2)$. Then
 \begin{enumerate}
  \item\label{prodhex1} $\mu_{\Xi_i}$ divides $\mu_\Xi$ for all $\Xi\in\Hex(P)$ and $i=1,2$, where $(\Xi_1,\Xi_2)=\Psi(\Xi)$;
  \item\label{prodhex2} for all $\Xi_1\in\Hex(P_1)$ and $\Xi_2\in\Hex(P_2)$, we have
        \[
         \mu_{\Xi_1}\cdot \mu_{\Xi_2} \ = \ \sum_{\Xi\in\Psi^{-1}(\Xi_1,\Xi_2)} \mu_\Xi.
        \]
 \end{enumerate}
 Let $\Xi_1\in\Hex(P_1)$ and $\Xi_2\in\Hex(P_2)$. Then the cardinality $r=\#\Psi^{-1}(\Xi_1,\Xi_2)$ and the orbit lengths $(\mu_{\tilde\Xi_1},\dotsc,\mu_{\tilde\Xi_r})$ of the hexagons $\tilde\Xi_1,\dotsc,\tilde\Xi_r$ in $\Psi^{-1}(\Xi_1,\Xi_2)$ only depend on $(\mu_{\Xi_1},\mu_{\Xi_2})$, up to a permutation of $\tilde\Xi_1,\dotsc,\tilde\Xi_r$, and are as in \autoref{table: the orbit lengths of hexagons in a product of hexagons}. 
\end{prop}

\begin{table}[ht]
 \caption{The orbit lengths of hexagons in $\Psi^{-1}(\Xi_1,\Xi_2)$}
 \label{table: the orbit lengths of hexagons in a product of hexagons}
  \begin{tabular}{|c||c|c|c|c|}
   \hline
   $\mu_{\Xi_1}\ \backslash\ \mu_{\Xi_2}$  & $1$   & $2$     & $3$       & $6$ \\
   \hline\hline
   $1$                                     & $(1)$ & $(2)$   & $(3)$     & $(6)$ \\ 
   \hline
   $2$                                     & $(2)$ & $(2,2)$ & $(6)$     & $(6,6)$ \\ 
   \hline
   $3$                                     & $(3)$ & $(6)$   & $(3,6)$   & $(6,6,6)$ \\ 
   \hline
   $6$                                     & $(6)$ & $(6,6)$ & $(6,6,6)$ & $(6,6,6,6,6,6)$ \\
   \hline
  \end{tabular}
\end{table}

\begin{proof}
 We begin with \eqref{prodhex1}. Consider $\Xi\in\Hex(P)$ and $(\Xi_1,\Xi_2)=\Psi(\Xi)$. Then $\psi$ restricts to a $D_3$-equivariant map $\psi\vert_{\Xi}:\Xi^\pair\to\Xi_1^\pair\times\Xi_2^\pair$. Composing this map with the $i$-th projection yields a $D_3$-equivariant map $\Xi^\pair\to\Xi_i^\pair$ for $i=1,2$. Since both the domain and codomain consist of a single $D_3$-orbit, we conclude that $\mu_{\Xi_i}=\#\Xi_i^\pair$ divides $\mu_\Xi=\#\Xi^\pair$, which proves \eqref{prodhex1}.
 
 We continue with \eqref{prodhex2}. The sets $\Xi_1^\pair$ and $\Xi_2^\pair$ are orbits of the $D_3$-action on $P_1^\pair$ and $P_2^\pair$, respectively. Thus the action of $D_3$ on $P^\pair$ restricts to $(\Xi_1\times\Xi_2)^\pair$ and this latter set decomposes into a disjoint union of $D_3$-orbits, which are precisely the sets of fundamental pairs of the hexagons in the fibre $\Psi^{-1}(\Xi_1,\Xi_2)$. Thus we obtain
 \[
  \mu_{\Xi_1}\cdot \mu_{\Xi_2} \ = \ \Big(\#\Xi_1^\pair\Big) \cdot \Big(\#\Xi_2^\pair\Big) \ = \ \#\Big(\Xi_1^\pair\times\Xi_2^\pair\Big) \ = \ \#\hspace{-0.4cm}\coprod_{\Xi\in\Psi^{-1}(\Xi_1,\Xi_2)}\hspace{-0.4cm}\Xi^\pair \ = \ \sum_{\Xi\in\Psi^{-1}(\Xi_1,\Xi_2)} \mu_\Xi,
 \]
 which establishes \eqref{prodhex2}.
 
 Given hexagons $\Xi_1\in\Hex(P_1)$ and $\Xi_2\in\Hex(P_2)$ with $\Psi^{-1}(\Xi_1,\Xi_2)=\{\tilde\Xi_1,\dotsc,\tilde\Xi_r\}$, we know by properties \eqref{prodhex1} and \eqref{prodhex2} that $\mu_{\Xi_i}$ divides $\mu_{\tilde\Xi_j}$ for all $i=1,2$ and $j=1,\dotsc,r$ and that $\mu_{\tilde\Xi_1}+\dotsb+\mu_{\tilde\Xi_r}=\mu_{\Xi_1}\cdot\mu_{\Xi_2}$. These properties determine $r$ and $(\mu_{\tilde\Xi_1},\dotsc,\mu_{\tilde\Xi_r})$ uniquely for all $(\mu_{\Xi_1},\mu_{\Xi_2})$, as presented in \autoref{table: the orbit lengths of hexagons in a product of hexagons}, with the single exception of the case $(\mu_{\Xi_1},\mu_{\Xi_2})=(3,3)$, for which the outcome could be either $(3,6)$ or $(3,3,3)$. We settle this case by analyzing the stabilizers of the relevant fundamental elements.

 To explain, a hexagon $\Xi$ is of near-regular type if and only if for any of its fundamental elements $(a,b)\in\Xi^\pair$, the stabilizer $\Stab_{D_3}(a,b)$ is trivial, and it is of dyadic type if and only if $\Stab_{D_3}(a,b)$ is cyclic of order $2$, i.e.\ if $\Stab_{D_3}(a,b)$ is generated by a reflection $\tau\sigma\tau^{-1}$ for some $\tau\in D_3$. If both $\Xi_1$ and $\Xi_2$ are of dyadic type and $(a_i,b_i)\in\Xi_i^\pair$ for $i=1,2$, then 
 \[
  \Stab_{D_3}\Big((a_1,b_1),(a_2,b_2)\Big) \ = \ \tau_1\gen\sigma\tau_1^{-1} \cap \tau_2\gen\sigma\tau_2^{-1} \ = \ \begin{cases}
                                                                                                                     \tau_1\gen\sigma\tau_1^{-1} & \text{if }\tau_1\tau_2^{-1}\in\gen\sigma, \\
                                                                                                                     \{e\}                       & \text{if }\tau_1\tau_2^{-1}\notin\gen\sigma, \\
                                                                                                                    \end{cases}
 \]
 where $\tau_1,\tau_2\in D_3$ depend on $\Big((a_1,b_1),(a_2,b_2)\Big)\in(\Xi_1\times\Xi_2)^\pair$. Since both cases $\tau_1\tau_2^{-1}\in\gen\sigma$ and $\tau_1\tau_2^{-1}\notin\gen\sigma$ occur, there is at least one hexagon $\tilde\Xi_1$ of dyadic type and at least one hexagon $\tilde\Xi_2$ of near-regular type among the hexagons $\tilde\Xi_1,\dotsc,\tilde\Xi_r$. Since $\#(\Xi_1\times\Xi_2)=3\cdot 3=3+6=\#\tilde\Xi_1+\#\tilde\Xi_2$, we conclude that there are no hexagons other than these two, which completes the proof.
\end{proof}

We equip ourselves with an additional fact about fundamental elements in product pastures.

\begin{lemma}\label{lemma: fundamental elements in product pastures}
 Let $P=P_1\times P_2$ be the product of two pastures $P_1$ and $P_2$. Then we have an equality
 \[
  \bigcup_{\substack{\Xi_1\in\Hex(P_1),\\ \Xi_2\in\Hex(P_2)\hfill}} \norm{\Xi_1}\times\norm{\Xi_2} \ = \ \bigcup_{\Xi\in\Hex(P)} \norm{\Xi}
 \]
 of subsets of $P$. If $P_1$ and $P_2$ are partial fields, then both unions are disjoint. 
\end{lemma}

\begin{proof}
 Since $\norm{\Xi_i}=\{a\in P\mid (a,b)\in\Xi_i^\pair\}$, the identity $\Xi_1^\pair\times\Xi_2^\pair=\coprod_{\Xi\in\Psi^{-1}(\Xi_1,\Xi_2)}\Xi^\pair$ implies that $\norm{\Xi_1}\times\norm{\Xi_2}=\bigcup_{\Xi\in\Psi^{-1}(\Xi_1,\Xi_2)}\norm{\Xi_i}$. Taking the union over all hexagons in $P_1$, $P_2$ and $P$ yields the first claim.

If $P_1$ and $P_2$ are partial fields, then by \autoref{lemma: product of partial fields is a partial field}, $P_1\times P_2$ is a partial field, and \autoref{prop: hexagons in partial fields} implies that $\bigcup\norm{\tilde\Xi}$ is disjoint union. Thus $\bigcup\norm{\Xi_1}\times\norm{\Xi_2}$ is also a disjoint union.
\end{proof}

\begin{ex*}\label{ex: fundamental elements in SxS}
 Note that in general, the union $\norm{\Xi_1}\times\norm{\Xi_2} =\bigcup_{i=1}^r \norm{\tilde\Xi_i}$ is not disjoint. For example, let $\Xi_1=\Xi_2=\hex1111{-1}{-1}$ be the unique hexagon of $\S=\past\Funpm{\{1+1-1\}}$, which is of dyadic type. The product
 \[
  \S\times\S \ = \ \past{\{0,(\pm1,\pm1)\}}{\{(1,1)+(1,1)+(-1,-1),\ (1,-1)+(-1,1)+(-1,-1)\}}
 \]
 has two hexagons
 \[\textstyle
  \tilde\Xi_1 \ = \ \underset{\text{(dyadic type)}}{\hex{(1,1)}{(1,1)}{(1,1)}{(1,1)}{(-1,-1)}{(-1,-1)}} \quad \text{and} \quad \tilde\Xi_2 \ = \ \underset{\text{(near-regular type)}}{\hex{(1,-1)}{(-1,1)}{(1,-1)}{(-1,1)}{(1,1)}{(1,1)}}, 
 \]
 and $\norm{\Xi_1}\times\norm{\Xi_2}=\norm{\tilde\Xi_1}\cup\norm{\tilde\Xi_2}$ is not a disjoint union.
\end{ex*}

As a sample consequence of \autoref{prop: hexagons in products}, we find:

\begin{cor}\label{cor: rescaling class spaces of U, H, and S}
Let $M$ be a matroid. Then there is a natural bijection
 \[
  \cX_M(\U) \ \stackrel\sim\longrightarrow \ \cX_M(\H) \times \cX_M(\S).
 \]
\end{cor}

\begin{proof}
 By \autoref{prop: hexagons in products}, the pasture $\H \times \S$ has a single hexagon of near-regular type. The result now follows from the fact that $ \cX_M(\U)$ and $\cX_M(\H)$ are both empty if $M$ is not ternary, together with the identification $\Lift_{\ternary}\U\simeq \U \simeq\Lift_{\ternary}(\H \times \S)$. 
\end{proof}


\subsection{Incidences for products of rescaling class spaces}
\label{subsection: Incidences for products of rescaling class spaces}

\begin{thm}\label{thm: identifications of products of rescaling class spaces}
 Let $p_1$ and $p_2$ be two prime powers such that $q=(p_1-2)(p_2-2)+2$ is a prime power that is not divisible by $3$. Then there is an identification
 \[
  \cX_M(\F_q) \ = \ \cX_M(\F_{p_1}) \times \cX_M(\F_{p_2}) 
 \]
 for every ternary matroid $M$. 
\end{thm}

\begin{proof}
 Let $M$ be a ternary matroid with foundation $F_M$. By \autoref{thm: characterizing property of the foundation} and the universal property of products, we have identifications
 \[
  \cX_M(\F_{p_1})\times \cX_M(\F_{p_2}) \ = \ \Hom(F_M,\F_{p_1})\times\Hom(F_M,\F_{p_2}) \ = \ \Hom(F_M,\F_{p_1}\times\F_{p_2})
 \]
 and by \autoref{thm: lift theorem for ternary matroids}, we have 
 \[
  \Hom(F_M,\F_{p_1}\times\F_{p_2}) \ = \ \Hom\Big(F_M,\Lift_{\ternary}(\F_{p_1}\times\F_{p_2})\Big).
 \]
 By the principle that we discussed in the beginning of \autoref{subsection: Incidences for rescaling classes}, an isomorphism $\cL\F_q\to\Lift_{\ternary}(\F_{p_1}\times\F_{p_2})$ induces a bijection $\cX_M(\F_q) \to \cX_M(\F_{p_1}) \times \cX_M(\F_{p_2})$. 
 
 By \autoref{prop: isomorphism types of hexagon lifts}, both ternary lifts $\cL\F_q$ and $\Lift_{\ternary}(\F_{p_1}\times\F_{p_2})$ are isomorphic to a tensor product of copies $\U$, $\D$, $\H$ and $\F_3$, one for each hexagon of the corresponding type in $\F_q$ and $\F_{p_1}\times\F_{p_2}$, respectively. Thus the two ternary lifts are isomorphic if and only if the numbers of hexagons of each type coincide for $\F_q$ and $\F_{p_1}\times\F_{p_2}$.
 
 \autoref{cor: numbers of hexagons in Fq by type} determines the number of hexagons in $\F_q$ and their types: there are $\lfloor\frac{q-2}6\rfloor$ hexagons of near-regular type; there is one hexagon of dyadic type if $q$ is odd and none if $q$ is even; there is one hexagon of hexagonal type if $q\equiv 1\pmod3$ and none otherwise; and there is no hexagon of ternary type since $q$ is not divisible by $3$ by our assumptions. By \autoref{prop: hexagons in partial fields}, we have 
 \[
  \sum_{\Xi\in\Hex(\F_q)}\#\Xi^\fund \ = \ q-2,
 \]
 and by \autoref{lemma: fundamental elements in product pastures} and \autoref{prop: hexagons in partial fields}, we have
 \[
  \sum_{\tilde\Xi\in\Hex(\F_{p_1}\times\F_{p_2})} \hspace{-0.7cm}\#\norm{\tilde\Xi} \ = \ \bigg(\sum_{\Xi_1\in\Hex(\F_{p_1})} \hspace{-0.4cm} \#\norm{\Xi_1} \bigg)\cdot\bigg(\sum_{\Xi_2\in\Hex(\F_{p_2})} \hspace{-0.4cm} \#\norm{\Xi_2} \bigg) \ = \ (p_1-2)\cdot (p_2-2).
 \]
 Since both numbers $q-2$ and $(p_1-2)\cdot (p_2-2)$ coincide by our assumptions and since the sum $3+2=5$ of the number of elements in a hexagon of dyadic type and a hexagon of hexagonal type is less than the number $6$ of elements in a hexagon of near-regular type, it suffices to show that $\F_{p_1}\times\F_{p_2}$ has at most one hexagon of dyadic type, at most one hexagon of hexagonal type and none of ternary type. Note that each of $\F_{p_1}$ and $\F_{p_2}$ has at most one hexagon of dyadic, hexagonal and ternary type by \autoref{cor: numbers of hexagons in Fq by type}.
 
 Examining the different constellations of products of hexagons in \autoref{prop: hexagons in products}, we see that $\F_{p_1}\times\F_{p_2}$ can only have a hexagon of ternary type if both $\F_{p_1}$ and $\F_{p_2}$ have a hexagon of ternary type. By \autoref{cor: numbers of hexagons in Fq by type} this means that $p_1\equiv p_2\equiv0\pmod3$ and thus $q=(p_1-2)(p_2-2)+2\equiv0\pmod3$, which we excluded by our assumptions. We conclude that $\F_{p_1}\times\F_{p_2}$ does not have a hexagon of ternary type.
 
 Similarly, we see that $\F_{p_1}\times\F_{p_2}$ can only have more than one hexagon of hexagonal type if both $\F_{p_1}$ and $\F_{p_2}$ have a hexagon of hexagonal type. By \autoref{cor: numbers of hexagons in Fq by type} this means that $p_1\equiv p_2\equiv1\pmod3$ and thus, again, $q=(p_1-2)(p_2-2)+2\equiv0\pmod3$, which we excluded by our assumptions. We conclude that $\F_{p_1}\times\F_{p_2}$ has at most one hexagon of hexagonal type.
 
 \autoref{prop: hexagons in products} also shows that $\F_{p_1}\times\F_{p_2}$ cannot have more than one hexagon of dyadic type. This verifies that $\Lift_{\ternary}\F_q$ and $\Lift_{\ternary}(\F_{p_1}\times\F_{p_2})$ are isomorphic.
\end{proof}

\begin{rem}\label{rem: Zagier}
 There are many instances of identifications of the form
 \[
  \cX_M(\F_q) \ = \ \cX_M(\F_{p_1}) \times \cX_M(\F_{p_2}).
 \]
 Trivially, we have for every prime power $q$ that 
 \[
  \cX_M(\F_2) \ = \ \cX_M(\F_{2}) \times \cX_M(\F_{q}) \qquad \text{and} \qquad \cX_M(\F_q) \ = \ \cX_M(\F_{3}) \times \cX_M(\F_{q}).
 \]
 But one also easily discovers many triples $(q,p_1,p_2)$ with $q,p_1,p_2\geq4$:
 \[
  \begin{array}{llllll}
   (  8, 4, 5),\, & ( 29, 5,11),\, & ( 47, 5,17),\, & ( 83, 5,29),\, & (125, 5,43),\, & (137,11,17), \\
   ( 11, 5, 5),\, & ( 32, 4,17),\, & ( 47, 7,11),\, & ( 83,11,11),\, & (128, 8,23),\, & (149, 9,23), \\
   ( 16, 4, 9),\, & ( 32, 5,11),\, & ( 51, 5,19),\, & ( 89, 5,31),\, & (128,11,16),\, & (163, 9,25), \\
   ( 17, 5, 7),\, & ( 32, 7, 8),\, & ( 71, 5,25),\, & (101,11,13),\, & (137, 5,47),\, & (167,13,17), \\
   ( 23, 5, 9),\, & ( 37, 7, 9),\, & ( 79, 9,13),\, & (121, 9,19),\, & (137, 7,29),\, & (173, 5,59).
  \end{array}
 \]
 According to a heuristic communicated to us by Don Zagier, the number of such triples $(q,p_1,p_2)$ up to some bound $N$ should grow roughly like $N/(\log N)^2$ for $N$ large, because there are roughly $N\log N$ solutions of the equation $q=(p_1-2)(p_2-2)+2$ in integers $p_1, p_2, q<N$, and the probability of all three being prime is roughly $1/(\log N)^3$. 
\end{rem}

\begin{rem}\label{rem: identifications of products of rescaling class spaces}
 We have formulated \autoref{thm: identifications of products of rescaling class spaces} in the most restrictive and at the same time most applicable way. In the following, we remark on generalizations and the necessity of our assumptions.
 \begin{enumerate}
  \item 
   Since $\cX_M(\F_3)$ is empty for non-ternary matroids $M$, it follows from \autoref{thm: identifications of products of rescaling class spaces} that there is an identification
   \[
    \cX_M(\F_q)\times\cX_M(\F_3) \ = \ \cX_M(\F_{p_1}) \times \cX_M(\F_{p_2}) \times\cX_M(\F_3)
   \]
   for every matroid $M$ if $q$, $p_1$ and $p_2$ satisfy the assumptions of the theorem.
  \item 
   The assumption that not $q$ is not divisible by $3$ is necessary as the following example shows. Since $27=(7-2)(7-2)+2$, the number of fundamental pairs in $\F_{27}$ and $\F_7\times\F_7$ are equal. However, there is no $D_3$-equivariant bijection between $\F_{27}^\pair$ and $(\F_7\times\F_7)^\pair$ since $\F_7\times\F_7$ has two hexagons of hexagonal type while $\F_{27}$ has no hexagon of hexagonal type; cf.\ \autoref{prop: hexagons in integral partial fields} and \autoref{prop: hexagons in products}.
   
   Note that an equality $q=(p_1-2)(p_2-2)+2$ where all of $q$, $p_1$ and $p_2$ are powers of $3$ leads also to an identification $\cX_M(\F_q)=\cX_M(\F_{p_1})\times\cX_M(\F_{p_2})$. We have excluded this case from the statement of \autoref{thm: identifications of products of rescaling class spaces} since $3^k=(3^i-2)(3^j-2)+2$ with $i,j,k\geq 1$ implies that $i=1$ or $j=1$, as a comparison modulo $9$ implies. In this case, we gain the trivial identification $\cX_M(\F_q)=\cX_M(\F_q)\times\cX_M(\F_3)$.

  \item
   By the same methods as we have proven \autoref{thm: identifications of products of rescaling class spaces}, we can prove the following more general statement. Let $p_1,\dotsc,p_n,q_1,\dotsc,q_m$ be prime powers such that 
   \[
    \prod_{i=1}^n (p_i-2) \ = \ \prod_{j=1}^m (q_j-2) 
   \]
   and such that either both products are $\equiv0\pmod3$ or the number of the factors $\equiv2\pmod3$ is the same for both products. Then there is an identification
   \[
    \prod_{i=1}^n\cX_M(\F_{p_i}) \ = \ \prod_{i=1}^m\cX_M(\F_{q_j})
   \]
   for every ternary matroid $M$. However, we did not find any such identification that we could not derive by combining identities of the type which appear in \autoref{thm: identifications of products of rescaling class spaces}.
 \end{enumerate}
\end{rem}


\begin{small}
 \bibliographystyle{plain}
 \bibliography{matroid}
\end{small}

\end{document}